\documentclass[11pt]{article} 
\usepackage[utf8]{inputenc} 

\usepackage[margin=1in]{geometry} 
\usepackage{graphicx} 

\usepackage{booktabs} 
\usepackage{array} 
\usepackage{paralist} 
\usepackage{verbatim} 
\usepackage{mathrsfs}
\usepackage{amssymb}
\usepackage{amsthm}
\usepackage{amsmath,amsfonts,amssymb}
\usepackage{esint}
\usepackage{graphics}
\usepackage{enumerate}
\usepackage{mathtools}
\usepackage{xfrac}
\usepackage{subcaption}
\usepackage{stmaryrd}
 \usepackage{mathabx}
 \usepackage[export]{adjustbox}
 \usepackage{float}

\usepackage[usenames,dvipsnames]{xcolor}
\usepackage[colorlinks=true, pdfstartview=FitV, linkcolor=blue, citecolor=blue, urlcolor=blue]{hyperref}
\usepackage[normalem]{ulem}

\usepackage{tikz-cd}

\numberwithin{equation}{section}
\numberwithin{figure}{section}

\newtheorem{theorem}{Theorem}[section]

\newtheorem{proposition}[theorem]{Proposition}
\newtheorem{lemma}[theorem]{Lemma}
\theoremstyle{definition}
\newtheorem{definition}[theorem]{Definition}
\newtheorem*{definition*}{Definition}

\newtheorem{remark}[theorem]{Remark}

\newtheorem*{note}{Note}

\newcommand*{\N}{\ensuremath{\mathbb{N}}}

\newcommand*{\Z}{\ensuremath{\mathbb{Z}}}

\newcommand*{\R}{\ensuremath{\mathbb{R}}}

\renewcommand*{\tilde}{\widetilde}

\renewcommand{\phi}{\varphi}

\DeclareSymbolFont{boldoperators}{OT1}{cmr}{bx}{n}
\SetSymbolFont{boldoperators}{bold}{OT1}{cmr}{bx}{n}
\usepackage{accents}

\newcommand{\T}{\mathbb{T}}

\def\XXint#1#2#3{{\setbox0=\hbox{$#1{#2#3}{\int}$}
\vcenter{\hbox{$#2#3$}}\kern-.5\wd0}}

\let\originalleft\left
\let\originalright\right
\renewcommand{\left}{\mathopen{}\mathclose\bgroup\originalleft}
\renewcommand{\right}{\aftergroup\egroup\originalright}

\renewcommand{\hat}{\widehat}

\makeatletter
\pgfmathdeclarefunction{erf}{1}{%
  \begingroup
    \pgfmathparse{#1 > 0 ? 1 : -1}%
    \edef\sign{\pgfmathresult}%
    \pgfmathparse{abs(#1)}%
    \edef\x{\pgfmathresult}%
    \pgfmathparse{1/(1+0.3275911*\x)}%
    \edef\t{\pgfmathresult}%
    \pgfmathparse{%
      1 - (((((1.061405429*\t -1.453152027)*\t) + 1.421413741)*\t 
      -0.284496736)*\t + 0.254829592)*\t*exp(-(\x*\x))}%
    \edef\y{\pgfmathresult}%
    \pgfmathparse{(\sign)*\y}%
    \pgfmath@smuggleone\pgfmathresult%
  \endgroup
}
\makeatother

\usepackage{titlesec}

\newcommand{\addperiod}[1]{#1.}
\titleformat{\section}
   {\centering\normalfont\Large}{\thesection.}{0.5em}{}
\titleformat*{\subsection}{\bfseries}
\titleformat{\subsubsection}[runin]
  {\normalfont\bfseries}
  {\thesubsubsection.}
  {0.5em}
  {\addperiod}
\titleformat*{\subsubsection}{\normalfont\itshape}
\titleformat*{\paragraph}{\bfseries}
\titleformat*{\subparagraph}{\large\bfseries}

\title{Higher-order propagation of chaos in $L^2$ for interacting diffusions}

\author{Elias Hess-Childs\thanks{Courant Institute of Mathematical Sciences, New York University.
{\footnotesize \href{mailto:elias.hess-childs@courant.nyu.edu}{elias.hess-childs@courant.nyu.edu}.}
}
\and 
Keefer Rowan\thanks{Courant Institute of Mathematical Sciences,  New York University.
{\footnotesize \href{mailto:keefer.rowan@cims.nyu.edu}{keefer.rowan@cims.nyu.edu}.}
}
}
\date{\today}

\usepackage[nottoc,notlot,notlof]{tocbibind}

\begin{document}

\maketitle

\vspace{-3mm}
\begin{abstract}
    In this paper, we study diffusions with bounded pairwise interaction. We show for the first time propagation of chaos on arbitrary time horizons in a stronger $L^2$-based distance, as opposed to the usual Wasserstein or relative entropy distances. The estimate is based on iterating inequalities derived from the BBGKY hierarchy and does not follow directly from bounds on the full $N$-particle density. This argument gives the optimal rate in $N$, showing the distance between the $j$-particle marginal density and the tensor product of the mean-field limit is $O(N^{-1})$. We use cluster expansions to give perturbative higher-order corrections to the mean-field limit. For an arbitrary order $i$, these provide ``low-dimensional'' approximations to the $j$-particle marginal density with error $O(N^{-(i+1)})$.
\end{abstract}

\setcounter{tocdepth}{1} 
\tableofcontents

\vspace{-3mm}
\section{Introduction}

In this paper we consider systems of $N$ interacting particles in $\Omega = \R^d$ or $\T^d$ of the form
\begin{equation}\label{eq.SDE}
\begin{cases}
dX_{j,N}(t)= \tfrac{1}{N}\sum_{k=1}^N K(X_{j,N}(t),X_{k,N}(t))dt+\sqrt{2}dW_j(t),&j\in\{1,\cdots,N\}\\
X_{j,N}(0)=Y_j,
\end{cases}
\end{equation}
where the $W_j(t)$ are independent standard Brownian motions, $Y_j$ are i.i.d. random variables with probability density $f(x)$, and $K(x,y)$ denotes the drift a particle at position $y$ induces on a particle at position $x$. Particle systems of this form arise in many contexts such as vortices in viscous fluids~\cite{onsager_statistical_1949,marchioro_mathematical_1994}, the training of large neural networks~\cite{chizat_global_2018,rotskoff_trainability_2022}, and aggregation and collective motion of microscopic organisms~\cite{topaz_nonlocal_2006,perthame_transport_2007}.

We recall that the law of the vector $(X_{1,N}, X_{2,N},\dotsc,X_{N,N})$ has a density $f_{N,N} : [0,\infty) \times \Omega^{N} \to \R$, which solves the Liouville 
equation,
\begin{equation}
\label{eq.liouville}
\begin{cases}\partial_t f_{N,N} - \Delta f_{N,N} + \frac{1}{N} \sum_{k,\ell=1}^N \nabla_{x_k} \cdot (K(x_k, x_\ell) f_{N,N}) =0,\\
f_{N,N}(0,x) = \prod_{k=1}^N f(x_k) = f^{\otimes N}(x).
\end{cases}
\end{equation}

By integrating the equation \eqref{eq.liouville} over $x_{j+1},\dotsc,x_{N}$ one finds that the marginal densities $f_{j,N}$ satisfy the PDE hierarchy
\begin{equation}
    \partial_t f_{j,N} - \Delta f_{j,N}+\frac{1}{N} \sum_{k,\ell=1}^j \nabla_{x_k}\hspace{-0.8mm}\cdot(K(x_k, x_\ell) f_{j,N}) ={-\frac{N-j}{N}} \sum_{k=1}^j \nabla_{x_k}\hspace{-0.8mm}\cdot\hspace{-0.2mm}\int K(x_k, x_*) f_{j+1,N}(x,x_*)\,dx_*
    \label{eq.bbgky-sketch}
\end{equation}
with initial data
\[f_{j,N}(0,\cdot) = f^{\otimes j}.\]
We note that $f_{N,N}$ is exchangeable, therefore the $j$-particle marginals $f_{j,N}$ are exchangeable and independent of which $N-j$ coordinates were integrated over. 

We study the \textit{propagation of chaos} of the system~\eqref{eq.SDE}, that is for any fixed $j$, the convergence as $N\to \infty$ of the marginal density $f_{j,N} \to \rho^{\otimes j}$, where $\rho$ solves the \textit{McKean-Vlasov} equation
\begin{equation}\label{eq.Mckean-Vlasov-equation}
    \begin{cases}
        \partial_t\rho(t,x)-\Delta\rho(t,x)+\nabla\cdot\big(\int K(x,x_*)\rho(t,x_*)\,dx_*\rho(t,x)\big)=0,\\
        \rho(t,\cdot)=f(\cdot).
    \end{cases}
\end{equation}

Propagation of chaos has been shown under a wide range of conditions on $f,\rho,$ and $K$ and under various distances; for some recent results see~\cite{durmus_elementary_2020,li_law_2020,guillin_systems_2023} and for a review of the vast literature see~\cite{chaintron_propagation_2022}. Recently, there has been lots of activity around quantitative propagation of chaos using relative entropy as a distance. In particular, global bounds---that is bounds on the relative entropy between $f_{N,N}$ and $\rho^{\otimes N}$---have been used to show quantitative propagation of chaos such as in~\cite{ben_arous_increasing_1999,jabin_mean_2016,lacker_strong_2018,jabir_rate_2019} for non-singular interactions. Additionally, estimates of this kind have been used for a large class of singular interactions~\cite{jabin_quantitative_2018, bresch_mean_2019,de_courcel_sharp_2023,rosenzweig_modulated_2023}.

Results based on global bounds at best show 
\[\sqrt{H(f_{j,N} \mid \rho^{\otimes j})} = O\Big(\sqrt{\tfrac{j}{N}}\Big),\]
where $H(f \mid g)$ is the relative entropy of $f$ with respect to $g$. This was widely believed to be optimal, but in~\cite{lacker_hierarchies_2023} it was shown that
\[\sqrt{H(f_{j,N} \mid \rho^{\otimes j})} = O\Big(\frac{j}{N}\Big),\]
for a class of interactions satisfying an exponential integrability condition. Further, this rate was shown to be optimal by constructing an example that saturates the bound. Instead of using global bounds,~\cite{lacker_hierarchies_2023} uses the BBGKY hierarchy~\eqref{eq.bbgky-sketch} to get bounds on $H(f_{j,N} \mid \rho^{\otimes j})$ in terms of 
$H(f_{j+1,N}~\mid~\rho^{\otimes (j+1)})$. The optimal rate is then shown by an iteration of these bounds.

In this paper, we instead prove bounds in an $L^2$ norm. In particular, we show for initial conditions $f \in L^1$ and bounded interaction, that for any $j = o(N^{2/3}),$
\[D_{j,N} := \bigg(\int \Big| \frac{f_{j,N} - \rho^{\otimes j}}{\rho^{\otimes j}}\Big|^2 \rho^{\otimes j}\,dx\bigg)^{1/2} = O\Big(\frac{j}{N}\Big).\]
We note that  $ D_{j,N}^2 = \chi^2(f_{j,N} \mid \rho^{\otimes j}),$ where $\chi^2(\mu \mid \nu)$ is the \textit{chi-squared divergence} of $\mu$ with respect to $\nu$. Pinsker's inequality and~\cite[Theorem 2]{sason_f_2016} respectively imply the inequalities  
\[\|\mu-\nu\|_{TV}^2 \leq \frac{1}{2} H(\mu \mid \nu) \leq \frac{1}{2} \chi^2(\mu \mid \nu),\]
for any probability measures $\mu$ and $\nu$. The $L^2$-type convergence of $f_{j,N} \to \rho^{\otimes j}$ thus implies the relative entropy convergence at the same optimal rate as in~\cite{lacker_hierarchies_2023}, which in turn implies $TV$ convergence.

The $L^2$ bounds are shown using somewhat analogous techniques to~\cite{lacker_hierarchies_2023}, controlling $D_{j,N}$ by $D_{j+1,N}$ and iterating these bounds. In contrast to relative entropy, no sufficiently strong global bounds on $D_{N,N}$ are available. In fact, the best bound we can show is 
\[D_{N,N} \leq C e^{CNt}.\]
This bound is far from sufficient to directly imply that $D_{j,N} \to 0$ for fixed $j$. As such, this $L^2$ distance is not amenable to global techniques and so necessitates analysis of the BBGKY hierarchy. We note in the recent preprint~\cite{bresch_new_2022}, propagation of chaos is shown in certain $L^p$ spaces and even for certain singular interactions, but only for sufficiently short times. In this paper, we show convergence on any time horizon.

One heuristic justification for propagation of chaos involves discarding terms of order $N^{-1}$ in the BBGKY hierarchy~\eqref{eq.bbgky-sketch} and noting that the $\rho^{\otimes j}$ are a solution to the resulting hierarchy of equations, as is explained in Subsection~\ref{ss.overview}. This suggests that the tensor product of the McKean-Vlasov solution $\rho^{\otimes j}$ is the $0$th-order term of a perturbative expansion of $f_{j,N}$ in powers of $N^{-1}$. This turns out to be in fact true, as we show in this paper by constructing this perturbative expansion and showing the appropriate bounds. Finding the correct perturbative approximation of order greater than $0$ requires the introduction of the \textit{cluster (or cumulant) expansion}, which rewrite the marginal densities $f_{j,N}$ in terms of certain sums of products of \textit{cluster functions} $g_{j,N}$, made precise in~\eqref{eq.f-jN-cluster-exp-intro}. After the correct perturbative approximations are found through cluster expansion, proving that they approximate $f_{j,N}$ to the appropriate order follows by the same analysis as the $0$th-order case discussed above.

The higher-order terms of the perturbative approximation are, unlike the $0$th-order terms $\rho^{\otimes j}$, not positive. In fact, in order to preserve the mean of the perturbative approximation, the higher-order terms are mean-zero and hence take both positive and negative values. As such, without strong pointwise control on the higher-order terms, the positivity of the higher-order approximations is unclear. This makes analyzing the error between $f_{j,N}$ and its approximation less amenable to probabilistic techniques such as relative entropy. In contrast, there are no such issues in the $L^2$ analysis.

Cluster expansions---and related expansions using correlation errors or $v$-functions---have been used in a wide variety of contexts to study asymptotics of statistical particle systems, for example~\cite{de_masi_mathematical_1991,pulvirenti_boltzmanngrad_2017,bodineau_statistical_2020} and citations therein. Somewhat relevant to our current study, \cite{duerinckx_size_2021} uses Glauber calculus to estimate cluster functions in the kinetic setting without noise in the evolution. In that work, the author uses a non-hierarchical technique and requires strong bounds on the regularity of the interaction: to go to arbitrary order the interaction must be $C^\infty$. 

The pair of papers~\cite{paul_size_2019,paul_asymptotic_2019} use a correlation error expansion and hierarchical techniques. The authors consider an abstract setting that covers both quantum mean-field models as well as stochastic jump processes such as the Kac model. Their analysis relies on iterating across the BBGKY hierarchy to pass estimates on the \textit{correlation errors}, which in turn implies propagation of chaos. In~\cite{paul_asymptotic_2019}, they take perturbative expansions of the correlation errors to construct higher order corrections to propagation of chaos. In their setting, the time evolution is unitary, allowing the use of techniques that are not clearly applicable to the setting we consider. Additionally, their perturbative expansion involves approximations which meaningfully depend on $N$, and the number of equations one must solve to construct the approximation of $f_{j,N}$ to order $i$ depends on $j$. In contrast, in our approach neither of these properties appear.

\subsection{Statement of main results}

Since the force a particle exerts on itself is given by $K(x,x)$, solutions to~\eqref{eq.liouville} are determined not only by $K(x,y)$ up to a.e.\ equivalence with respect to the Lebesgue measure on $\Omega^2$ but also by the a.e.\ equivalence class of the diagonal $K(x,x)$ with respect to the Lebesgue measure on $\Omega$. As such, we make the following definition.

\begin{definition*}
    For functions $K : \Omega^2 \to \R^d,$ we denote
    \[\|K\|_{L^\infty_\delta(\Omega^2)} := \|K(x,y)\|_{L^\infty(\Omega^2)} + \|K(x,x)\|_{L^\infty(\Omega)}.\]
\end{definition*}

We also introduce some notation necessary to express the perturbative expansions.
\begin{definition*}
    For any set $A$, we use the notation $\pi \vdash A$ to denote that $\pi$ is a partition of $A$. When appearing in a sum
    \[\sum_{\pi \vdash A}\]
    we mean that the sum is taken over all possible partitions of $A.$ We often take $\pi \vdash [j]$ for some $j \in \N$ where
    \[[j] :=\{1,\dotsc,j\}.\]
\end{definition*}

\begin{definition*}
    Fix some finite set $A$ and let $(h_j)_{1 \leq j \leq |A|}$ be a family of exchangeable functions such that $h_j : \Omega^j \to \R$. Then for any partition $\pi \vdash A,$ we denote 
    \[\prod_{P \in \pi} h_P : \Omega^A \to \R\]
    such that
    \[\prod_{P \in \pi} h_P(x) := \prod_{P \in \pi} h_{|P|}(x^P),\]
    where for $x \in \Omega^A$
    \[x^P := (x_k)_{k \in P}.\]
    Note by exchangeability, the order of the $x_k$ in $x^P$ doesn't matter.
\end{definition*}

\begin{definition*}
    For any partition $\pi \vdash j$ where $\pi = \{P_1,...,P_k\}$, by
    \[\sum_{\substack{(i_P)_{P \in \pi}\\ \sum i_P = i}}\]
    we denote that the sum is over all choices of $i_{P_1},\dotsc,i_{P_k} \in \N$ such that $\sum_{P\in \pi} i_P = i$.
\end{definition*}

Our first results are on the existence and representation of the perturbative approximation.

\begin{definition}\label{def.order}
    Let $D = \{(i,j) \in \N^2 : 1 \leq j \leq i+1, j \geq 1\}$. We define an ordering on $D$ by saying $(i,j) \leq (k,\ell)$ if
    \[i < k\ \text{or both}\ i=k\ \text{and}\ j \geq \ell.\]
\end{definition}

\begin{figure}[H]
\[\begin{tikzcd}
	&&& {} \\
	0 & 0 & {g_3^2} & {} \\
	0 & {g_2^1} & {g_2^2} \\
	{g_1^0} & {g_1^1} & {g_1^2}
	\arrow[from=1-4, to=4-3]
	\arrow[from=2-3, to=4-2]
	\arrow[from=3-2, to=4-1]
	\arrow[from=3-3, to=2-3]
	\arrow[from=4-2, to=3-2]
	\arrow[from=4-3, to=3-3]
\end{tikzcd}\]
\caption{The dependency graph for $g_j^i$.}
\end{figure}
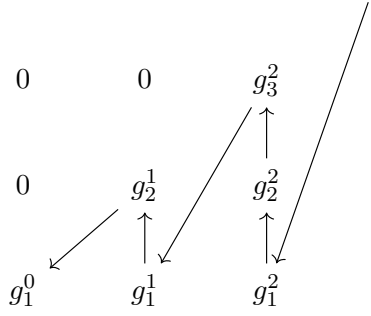

\begin{proposition}
    \label{prop.g-ij-intro}
    Suppose the initial distribution $f \in L^1(\Omega)$ and the interaction $K \in L_\delta^\infty(\Omega^{2})$. Then there exists a family of functions $g^i_j \in C^0_{loc}([0,\infty), L^1(\Omega^{j})),$ where $j \in \{1,2,\dotsc\}$ and $i \in \{0,1,\dotsc\}$ so that $g^i_j$ solve the equations~\eqref{eq.g-ij-PDE}\footnote{In the linked equation, the operators $H_k$ and $S_{k,\ell}$ appear. These are defined below in Definition~\ref{def.H-k-S-kl}.} with initial data~\eqref{eq.g-ij-initial-conditions}. More so, the $g^i_j$ have the properties:
    \begin{enumerate}
        \item \label{item.g-ij-vanish} For $(i,j) \not \in D$, $g^i_j = 0$.
        \item \label{item.g-ij-depend} For $(i,j) \in D$, the equation for $g^i_j$ depends only on the $g^k_\ell$ with $(k,\ell) \leq (i,j)$ under the ordering on $D$. More so, the equation for $g^i_j$ is linear in $g^i_j$ for $(i,j) > (0,1).$
        \item $g^0_1 = \rho$, the unique solution to the McKean-Vlasov equation~\eqref{eq.Mckean-Vlasov-equation}.
        \item\label{item.g-ij-unique} Assuming Property \ref{item.g-ij-vanish}, these solutions are unique for fixed $f$.
    \end{enumerate}
\end{proposition}

The functions $g^i_j$ at this stage may be somewhat opaque, but they are the natural perturbative expansion of the cluster functions $g_{j,N}$, as will be made clear in Subsection~\ref{ss.overview}.

\begin{theorem}
    \label{thm.f-ij-existence-intro}
    Suppose the initial distribution $f \in L^1(\Omega)$ and the interaction $K \in L^\infty_\delta(\Omega^2)$. Then let 
    \begin{equation}
    \label{eq.f-ij-def}
    f^i_j := \sum_{\pi \vdash [j]} \sum_{\substack{(i_P)_{P \in \pi}\\ \sum i_P = i}} \prod_{P \in \pi} g^{i_P}_P
    \end{equation}
    where the $g^i_j$ are as in Proposition~\ref{prop.g-ij-intro}. Then 
\begin{align}
    \partial_t f^i_j - \Delta f^i_j &=-\sum_{k=1}^j \nabla_{x_k} \cdot \int  K(x_k,x_*) f^i_{j+1}(x,x_*)\,dx_* - \sum_{k,\ell=1}^j \nabla_{x_k} \cdot ( K(x_k,x_\ell) f^{i-1}_j)
    \notag\\&\qquad+  j \sum_{k=1}^j \nabla_{x_k} \cdot \int K(x_k,x_*) f^{i-1}_{j+1}(x,x_*)\,dx_*, 
    \label{eq.f-ij-eqs-intro}
\end{align}
    where we take the convention $f^{-1}_j = 0$ and the $f^i_j$ have initial data
    \[f^i_j(0,\cdot) = \begin{cases} f^{\otimes j} & i=0,\\ 0 & i \geq 1. \end{cases}\]
    In particular, we have
    \[f^0_j = \rho^{\otimes j},\]
    where $\rho$ is the unique solution to the McKean-Vlasov equation~\eqref{eq.Mckean-Vlasov-equation}.
\end{theorem}

We note that the equation~\eqref{eq.f-ij-eqs-intro} is what one gets from formally expanding $f_{j,N} = \sum_{i=0}^\infty N^{-i}f^i_j$, plugging the right hand side into the BBGKY hierarchy~\eqref{eq.bbgky-sketch}, and collecting orders. Thus we expect any $f^i_j$ solving~\eqref{eq.f-ij-eqs-intro} to be such that
\[f_{j,N} = \sum_{k=0}^i N^{-k} f^k_j + O(N^{-(i+1)}).\]

Theorem~\ref{thm.f-ij-existence-intro} gives an explicit representation~\eqref{eq.f-ij-def} of solutions $f^i_j$ to~\eqref{eq.f-ij-eqs-intro}. Further, properties~\ref{item.g-ij-vanish} and~\ref{item.g-ij-depend} of Proposition~\ref{prop.g-ij-intro} ensure that the expression~\eqref{eq.f-ij-def} for $f^i_j$ is computable in terms of the finite collection $\{g^k_\ell : k \leq i, (k,\ell) \in D\}$ which depends only on $i$, not on $j$ or $N$. That is, in order to compute $f^i_j$ for any $j$, one only needs to solve $\tfrac{1}{2}(i+2)(i+1)$ equations. It is in this sense that we say that $f^i_j$ give ``low-dimensional'' approximations to the $f_{j,N}.$

The main result of this paper is then to show that the $f^i_j$ as constructed in Theorem~\ref{thm.f-ij-existence-intro} appropriately approximate $f_{j,N}$. 
\begin{theorem}
    \label{thm.main-result-intro}
    Suppose $f \in L^1(\Omega), K \in L^\infty_\delta(\Omega^2),$ then for each $i \in \N,$ there exists $C(\|K\|_{L^\infty_\delta(\Omega^{2})},i)< \infty$ such that for any $N$ and any $j$ with
    \[j \leq C^{-1} e^{-Ct^2} N^{2/3},\]
    we have the bound
    \begin{equation}
    \label{eq.main-result-bound}
    \int \bigg|\frac{f_{j,N}- \sum_{k=0}^i N^{-k}f^k_j }{\rho^{\otimes j}}\bigg|^2 \rho^{\otimes j}\,dx \leq  Ce^{Ct} \Big(\frac{j}{N}\Big)^{2(i+1)},
    \end{equation}
    with the $f^k_j$ given as in Theorem~\ref{thm.f-ij-existence-intro}.
\end{theorem}

\begin{remark}
    For $i \geq 1$, the $L^2$-type distance between $f_{j,N}$ and $\sum_{k=0}^i N^{-k} f_j^k$ bounded in~\eqref{eq.main-result-bound} is not a chi-squared divergence, hence does not bound the relative entropy. Nevertheless, an application of H\"older's inequality implies that under the same conditions of Theorem~\ref{thm.main-result-intro},
    \[\Big\|f_{j,N} - \sum_{k=0}^i N^{-k} f_j^k\Big\|_{TV} \leq Ce^{Ct} \Big(\frac{j}{N}\Big)^{i+1}.\]
\end{remark}

\begin{remark}
    We note that in the $i=0$ case, Theorem~\ref{thm.main-result-intro} gives the estimate
    \[\sqrt{\chi^2(f_{j,N} \mid \rho^{\otimes j})}\leq Ce^{Ct}\frac{j}{N},\]
    showing convergence in chi-squared divergence (and hence in relative entropy and total variation) with optimal rate in $N^{-1}$.
\end{remark}

\begin{remark}
    A simple argument shows that the rate 
    \[\int \bigg|\frac{f_{j,N}- \sum_{k=0}^i N^{-k}f^k_j }{\rho^{\otimes j}}\bigg|^2 \rho^{\otimes j}\,dx = O(N^{-2(i+1)})\]
    is optimal for some fixed $j$ and $i$, provided the next order correction $f_j^{i+1}$ is not identically zero. It is not completely straightforward to construct examples for which one can show that $f_j^{i+1} \ne 0$ for some $j$, but it would be extremely surprising if there were no such examples. If that were the case, there would be some $i_*$ such that for any $f$ and $K$, we would have $f^{i}_j =0$ for all $i \geq i_*$ and all $j$. In particular, this would imply that $f_{j,N} - \sum_{k=0}^{i_*} N^{-k} f^k_j$ vanishes faster than any polynomial rate in $N^{-1}$.
\end{remark}

\begin{remark}
    Throughout this paper we assume that the initial data of $f_{j,N}$ is completely tensorized, that is $f_{j,N}(0,\cdot) = f^{\otimes j}$. As is usual, we don't strictly need this to be true; one can show the same bound~\eqref{eq.main-result-bound} at order $i$, provided for all $j,$ 
 the initial data $f_{j,N}(0,\cdot) \in L^1(\Omega^j)$ and satisfies the quantitative bound
        \begin{equation*}
    \int \bigg|\frac{f_{j,N}(0,\cdot)- f^{\otimes j} }{f^{\otimes j}}\bigg|^2 f^{\otimes j}\,dx \leq  C_0\Big(\frac{j}{N}\Big)^{2(i+1)},
    \end{equation*}
    for some $C_0$ independent of $j$. Of course then the constant $C$ of the bound~\eqref{eq.main-result-bound} would then depend on $C_0.$ We omit this argument as it adds notational complexity without adding any real content.
\end{remark}

\begin{remark}
    We note the restriction $j = o(N^{2/3})$. This is not very constraining and still shows strong bounds along a broad class of simultaneous limits of $(j, N) \to (\infty,\infty)$---these simultaneous limits are sometimes called \textit{increasing propagation of chaos}~\cite{ben_arous_increasing_1999,miclo_genealogies_2001}. The restriction is however worse than in~\cite{lacker_hierarchies_2023}, which allows $j = O(N)$. This restriction originates from the prefactor $\frac{j^3}{N^2}$ that appears on a term in the fundamental energy-type estimate given in Proposition~\ref{prop.l2-hierarchy}. We need the prefactor of this term to be $O(1)$ in order to not cause growth when the hierarchy of differential inequalities is iterated, thus we give the requirement that $j = O(N^{2/3})$. The time decay in the upper bound on $j$, that $j \leq C^{-1} e^{-Ct^2} N^{2/3}$---and hence that $j = o(N^{2/3})$---then comes from the iteration of a short time argument which requires us to restrict to a smaller set of $j$ on each iteration.
\end{remark}

\begin{remark}
    Theorem~\ref{thm.main-result-intro} in particular implies that for fixed $j$
    \[N(f_{j,N} - \rho^{\otimes j}) \stackrel{TV}{\to} f^1_j,\]
    and similarly for the higher-order $f^i_j$. This justifies that the $f^i_j$ are the natural next order corrections. We note that due to the $N$-dependence of the higher-order corrections in~\cite{paul_asymptotic_2019}, no such result is available in their analysis.
\end{remark}

\begin{remark}
    The interaction $K$ has not been assumed to be symmetric nor has $K(x,x)$ been assumed to be $0$. In particular, we allow
    \[K(x,y):=b(x)+\hat K(x-y)\]
    where $b$ is a drift affecting all particles and $\hat K$ is a translation-invariant pairwise interaction.
\end{remark}

\begin{remark}
    We note that we could have additionally taken $K$ to depend on time, that is the proof follows identically taking $K \in L^\infty([0,\infty), L^\infty_\delta(\Omega^2)).$ Since we have no particularly interesting examples involving time-dependent $K$ in mind, we omit this case for brevity. 

    We also note that we could consider inhomogeneous diffusions, that is allowing the Brownian motion to have a space-time dependent diffusion matrix prefactor $\sigma(t,x)$ in the defining SDE~\eqref{eq.SDE}. For uniformly elliptic diffusions, that is those such that there exist $\Lambda<\infty$ such that for every $x \in \Omega, t \in [0,\infty),$ and $v \in \R^d,$ 
    \[\Lambda^{-1} |v|^2 \leq |\sigma(t,x) v|^2 \leq \Lambda|v|^2,\]
    the formal computations allowing us to deduce Theorem~\ref{thm.main-result-intro} follow identically, with a new constant depending on $\Lambda$. Rigorously justifying all of the necessary computations with this inhomogeneous diffusion would require significantly more technical parabolic theory than we currently use, though we strongly believe that this would introduce no new essential difficulties. For simplicity and legibility, we omit this case.
\end{remark}

\begin{remark}
    \label{rmk.i=1-case}
     In order to make the general higher-order corrections more concrete, here we explicitly give the first-order corrections: the $g^1_j$ and $f^1_j$. All the $g^1_j=0$ except for $j=1,2$. Letting $\rho$ be the unique solution to the McKean-Vlasov equation~\eqref{eq.Mckean-Vlasov-equation}, $g^1_2$ solves the equation 
    \begin{align*}
        \partial_t g^1_2 &- \Delta g^1_2 + \nabla_x \cdot \int K(x,x_*) \big(\rho(x) g^1_2(y,x_*)+\rho(x_*) g^1_2(x,y)\big)\,dx_* 
        \\&\qquad\qquad+ \nabla_y \cdot \int K(y,x_*)  \big(\rho(y) g^1_2(x,x_*)   +\rho(x_*) g^1_2(x,y)\big)\,dx_* 
        \\&= \nabla_x \cdot \int K(x,x_*)\rho(x_*) \rho(x)\rho(y)\,dx_* + \nabla_y \cdot \int K(y,x_*)\rho(x_*) \rho(x)\rho(y)\,dx_* 
        \\&\qquad\qquad- \nabla_x \cdot \big(K(x,y) \rho(x)\rho(y)\big) -\nabla_y \cdot\big(K(y,x) \rho(x)\rho(y)\big).
    \end{align*}
    The equation for $g^1_1$ is
    \begin{align*}
        \partial_t g^1_1 &- \Delta g^1_1 + \nabla \cdot \int K(x,x_*) \big(g^1_1(x_*) \rho(x) + \rho(x_*) g^1_1(x)\big)\,dx_* 
        \\&= \nabla \cdot \int K(x,x_*) \big(\rho(x_*)\rho(x) -g^1_2(x,x_*)\big)\,dx_* - \nabla \cdot \big(K(x,x) \rho(x)\big).
    \end{align*}
    Then for any $j$, $f^1_j$ is given by
    \[f^1_j = \sum_{k=1}^j g^1_1(x_k) \rho^{\otimes (j-1)}(x^{[j] - \{k\}}) + \sum_{1 \leq k <\ell \leq j} g^1_2(x_k,x_\ell) \rho^{\otimes (j-2)}(x^{[j] - \{k,\ell\}}).\]
    We note that the equation for $g^1_2$ only depends on $\rho$, the equation for $g^1_1$ only depends on $\rho$ and $g^1_2$, and $f^1_j$ is computable for any $j$ in terms of the three functions $\rho, g^1_1$, and $g^1_2.$
\end{remark}

\subsection{Overview of the argument}
\label{ss.overview}

We first introduce the motivation and construction of the higher-order corrections $f^i_j$ through the cluster expansion and perturbation theory. We then explain the $L^2$ analysis of the BBGKY hierarchy that allows us to prove the bound~\eqref{eq.main-result-bound}.

\subsubsection{Higher-order corrections to propagation of chaos}

One formal argument for propagation of chaos is given by discarding terms of order $N^{-1}$ in the hierarchy~\eqref{eq.bbgky-sketch}, which gives the hierarchy
\[\begin{cases}
    \partial_t f_{j}^0 - \Delta f_{j}^0={-}\sum_{k=1}^j \nabla_{x_k} \cdot \int K(x_k, x_*) f_{j+1}^0(x,x_*)\,dx_*,\\
    f_{j}^0(0,\cdot) = f^{\otimes j},
\end{cases}\]
where the notation $f_j^0$ is due to the fact we are only keeping track of terms to $0$th order in $N^{-1}$. One can then note that $f_{j}^0 := \rho^{\otimes j}$ is a solution to this system. Thus the tensor product $\rho^{\otimes j}$ is formally the $0$th order term of a perturbative expansion of $f_{j,N}$. We are then interested in the higher-order terms of this expansion, so we formally suppose
\begin{equation}
    \label{eq.f-jN-perturb-exp-intro}
f_{j,N} = \sum_{i=0}^\infty N^{-i}f^i_j. 
\end{equation}
Collecting orders of $N^{-1}$, we get that $f^i_j$ solves the equation~\eqref{eq.f-ij-eqs-intro}. We note that for each $i$, this is an infinite hierarchy of equations in $j$, with forcing depending on $\{f^{i-1}_j  :j \in \N\}$. It is not at all clear how to directly construct solutions to these hierarchies. 

\begin{figure}[H]
\[\begin{tikzcd}
	{} & {} & {} \\
	{f_3^0} & {f_3^1} & {f_3^2} & {} \\
	{f_2^0} & {f_2^1} & {f_2^2} & {} \\
	{f_1^0} & {f_1^1} & {f_1^2} & {}
	\arrow[from=2-1, to=1-1]
	\arrow[from=2-2, to=1-2]
	\arrow[from=2-2, to=2-1]
	\arrow[from=2-3, to=1-3]
	\arrow[from=2-3, to=2-2]
	\arrow[from=2-4, to=2-3]
	\arrow[from=3-1, to=2-1]
	\arrow[from=3-2, to=2-2]
	\arrow[from=3-2, to=3-1]
	\arrow[from=3-3, to=2-3]
	\arrow[from=3-3, to=3-2]
	\arrow[from=3-4, to=3-3]
	\arrow[from=4-1, to=3-1]
	\arrow[from=4-2, to=3-2]
	\arrow[from=4-2, to=4-1]
	\arrow[from=4-3, to=3-3]
	\arrow[from=4-3, to=4-2]
	\arrow[from=4-4, to=4-3]
\end{tikzcd}\]
\caption{The dependency graph for $f_j^i$.}
\end{figure}
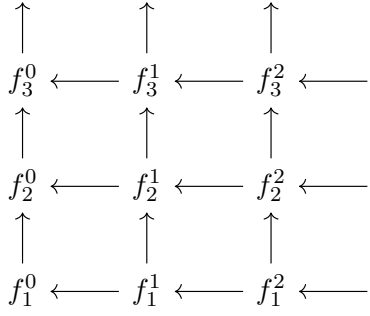

To solve this problem, we use the \textit{cluster (or cumulant) expansion}. That is, we express the $f_{j,N}$ in terms of a family of exchangeable functions $g_{1,N},...,g_{N,N}$, namely
\begin{equation}
    \label{eq.f-jN-cluster-exp-intro}
f_{j,N} = \sum_{\pi \vdash [j]} \prod_{P \in \pi} g_{|P|,N}(x^P).
\end{equation}
From this ansatz, one can deduce an inversion formula
\begin{equation} 
\label{eq.g-ij-interms-f-ij}
g_{j,N} = \sum_{\pi \vdash [j]} (-1)^{|\pi| -1} (|\pi|-1)!\prod_{P \in \pi} f_{|P|,N}(x^P),\end{equation}
which defines the $g_{\ell,N}$ in terms of the $f_{j,N}$. The BBGKY hierarchy~\eqref{eq.bbgky-sketch} then induces the hierarchy of equations~\eqref{eq.cluster-pde} on the $g_{\ell,N}$. We can then take formal perturbative expansion of the $g_{\ell,N}$, writing
\begin{equation}
\label{eq.g-kN-expansion-intro}
g_{\ell,N} = \sum_{k=0}^\infty N^{-k} g^k_\ell,
\end{equation}
and collect orders in equation~\eqref{eq.cluster-pde} to get equation~\eqref{eq.g-ij-PDE} on the $g^k_\ell$. Unlike the equations~\eqref{eq.f-ij-eqs-intro}, the equations for the $g^k_\ell$ can be inductively solved. Then plugging the expansion~\eqref{eq.g-kN-expansion-intro} into the cluster expansion~\eqref{eq.f-jN-cluster-exp-intro} and collecting orders of $N^{-1}$, we formally find representation of the $f^i_j$ of~\eqref{eq.f-jN-perturb-exp-intro} in terms of the $g^i_j$, as given in~\eqref{eq.f-ij-def}. Theorem~\ref{thm.f-ij-existence-intro} then gives that this expression for $f^i_j$ in terms of the $g^i_j$ actually solves the equation~\eqref{eq.f-ij-eqs-intro} we formally expect it to.

\begin{remark}
    Note~\eqref{eq.main-result-bound} gives that
    \[f_{j,N} = \sum_{k=0}^i N^{-k} f^k_j + O(N^{-(i+1)}).\]
    Inserting this approximation into~\eqref{eq.g-ij-interms-f-ij} and using the definition of the $f^k_j$,~\eqref{eq.f-ij-def}, one can show that
    \[g_{j,N} = \sum_{k=0}^i N^{-k} g^k_j + O(N^{-(i+1)}).\]
    Since, as noted in Proposition~\ref{prop.g-ij-intro} $g^k_j = 0$ for $k \leq j-2$, we see by letting $i= j-2$ that 
    \[g_{j,N} = O(N^{-(j-1)}).\]
    This shows that $g_{j,N}$ are small for all $j\geq 2$ and in particular allows estimates on the joint cumulants of observables on $j$ particles. That is, for any $\phi_1,\dotsc,\phi_j\in C^0(\Omega,\R)$ we have that
    \begin{align*}
    \kappa\Big(\phi_1\big(X_{1,N}(t)\big),\dotsc, \phi_j\big(X_{j,N}(t)\big)\Big)&= \int \prod_{k=1}^j \phi_k(x_k)\, g_{j,N}(t,x_1,\dotsc,x_j)\,dx 
    \\&\leq \prod_{k=1}^j\|\phi_k\|_{C^0} \|g_{j,N}\|_{TV} 
    \\&= O(N^{-(j-1)}),
    \end{align*}
    where $\kappa(Z_1,..,Z_j)$ denotes the joint cumulant of $Z_1,...,Z_j$. Thus the results of this paper in particular show the smallness of joint cumulants of observables of many particles, with a rate getting very small as the number of particles gets large. We note that these estimates on cumulants are related to the Bogolyubov corrections---a version of these bounds on the cumulants in the context of second-order in time interacting particle systems conjectured by physicists~\cite{bogolyubov_problems_1960}.
\end{remark}

\subsubsection{$L^2$ hierarchy estimates}

We now sketch the $L^2$-based estimates on the BBGKY hierarchy. Fundamentally we are concerned with estimating the size of solutions to the hierarchy
\begin{equation}
\label{eq.gamma-j-bbgky}
        \partial_t \gamma_{j} - \Delta \gamma_{j}+\frac{1}{N} \sum_{k,\ell=1}^j \nabla_{x_k} \cdot (K(x_k, x_\ell) \gamma_{j}) + \frac{N-j}{N} \sum_{k=1}^j \nabla_{x_k} \cdot \int K(x_k, x_*) \gamma_{j+1}(x,x_*)\,dx_* = \nabla \cdot R_j,
\end{equation}
where the $\gamma_j$ have initial data $\gamma_j(0,\cdot)=0.$ In particular, for Theorem~\ref{thm.main-result-intro} we take for fixed $i,$ 
\[\gamma_j = f_{j,N} - \sum_{k=0}^i N^{-k} f^i_j.\]
By construction, this $\gamma_j$ satisfies~\eqref{eq.gamma-j-bbgky} with an error $R_j$ such that $R_j = O(N^{-(i+1)}).$ The goal then is to show that $\gamma_j = O(R_j)$. This is accomplished by noting
    \begin{align}
        \frac{d}{dt} \int \Big|\frac{\gamma_j}{\rho^{\otimes j}}\Big|^2 \rho^{\otimes j}\,dx &\leq 2 j \|K\|_{L^\infty_\delta}^2 \Bigg(\int \Big|\frac{\gamma_{j+1}}{\rho^{\otimes (j+1)}}\Big|^2 \rho^{\otimes (j+1)}\,dx_*dx - \int \Big|\frac{\gamma_j}{\rho^{\otimes j}}\Big|^2 \rho^{\otimes j}\,dx \Bigg)
        \notag\\&\qquad + 4\frac{j^3}{N^2}\|K\|_{L^\infty_\delta}^2 \int \Big|\frac{\gamma_j}{\rho^{\otimes j}}\Big|^2 \rho^{\otimes j}\,dx + 2 \int \Big|\frac{R_j}{\rho^{\otimes j}}\Big|^2 \rho^{\otimes j}\,dx,
        \label{eq.l2-hierarchy-intro}
    \end{align}
which is shown (see Proposition~\ref{prop.l2-hierarchy} for details) by directly expanding the time derivative on the left hand side and using the equations solved by $\gamma_j$ and $\rho^{\otimes j}$. This estimate is in many way analogous to~\cite[Equation (1-17)]{lacker_hierarchies_2023}. The bound is also used similarly. In particular letting $\beta := 4\|K\|_{L^\infty_\delta}^2$ and
\[x_j :=  \int \Big|\frac{\gamma_j}{\rho^{\otimes j}}\Big|^2 \rho^{\otimes j}\,dx,\quad r_j := 2 \int \Big|\frac{R_j}{\rho^{\otimes j}}\Big|^2 \rho^{\otimes j}\,dx,\]
\eqref{eq.l2-hierarchy-intro} implies
\begin{equation}
    \label{eq.x-j-intro}
\dot x_j \leq \beta j (x_{j+1} - x_j) +\beta\frac{j^3}{N^2}x_j + r_j.
\end{equation}
In Proposition~\ref{prop.r-ij-bound}, we show that
\[r_j \leq C e^{Ct} \Big(\frac{j}{N}\Big)^{2(i+1)},\]
where $C$ does not depend on $j$. It is worth noting that a more naive bound on $r_j$ would give a suboptimal rate in $j$---giving $j^{3+4i}$ instead of $j^{2(i+1)}$---but by taking advantage of certain $L^2$ orthogonality, the above bound can be shown.

Now what remains to be shown is the $r_j$ are the leading order contribution to the size of the $x_j$. To see this, $x_j$ is controlled by iteratively applying Gr\"onwall's inequality to~\eqref{eq.x-j-intro}, which gives an estimate of the form
\begin{equation}
    \label{eq.ode-hierarchy-no-t0-intro}
    x_j(t) \leq CI^\ell_{j}(t)\sup_{s \in [0, t]} x_{j+\ell}(s) + C\sum_{k=0}^{\ell-1} I_j^{k+1}(t) \frac{r_{j+k}}{j+k},
\end{equation}
provided $j+\ell \leq N^{2/3}$ and where the $I^\ell_j$, defined in Definition~\ref{def.I-lj-def}, are certain iterated exponential integrals. The $I^\ell_j$ admit the estimates
\begin{equation}
\label{eq.I-lj-poly-decay-intro}
I^\ell_j(t) \leq \Big(\frac{j+b}{j+\ell}\Big)^b e^{\beta b t},\quad b \in \N.
\end{equation}
By taking $b = 2i+3$ and applying this bound we appropriately control the second term of~\eqref{eq.ode-hierarchy-no-t0-intro}, namely
\[C\sum_{k=0}^{\ell-1} I_j^{k+1}(t) \frac{r_{j+k}}{j+k} \leq C e^{Ct} \Big(\frac{j}{N}\Big)^{2(i+1)}.\]
The remaining issue is to bound $I^\ell_{j}(t)\sup_{s \in [0, t]} x_{j+\ell}(s)$. In~\cite{lacker_hierarchies_2023}, a simple \textit{a priori} bound\footnote{By \textit{a priori} we mean that the bound is in some way independent of the perturbation theory and is rather just an initial estimate of size.} on the analog to $x_k$ was available, giving that $x_k \leq Ckt.$ No such bound is available in our setting. The best \textit{a priori} bound we have is given by~\eqref{eq.l2-apriori} which implies that $x_k \leq Ce^{Ckt}.$
Thus the best bound of the remaining term that we have available is 
\[I^\ell_{j}(t)\sup_{s \in [0, t]} x_{j+\ell}(s)\leq C e^{C(j+\ell)t} I^\ell_{j}(t).\]
For this, using~\eqref{eq.I-lj-poly-decay-intro} with fixed $b$ is insufficient, as the exponential growth will always beat polynomial decay. Instead, by optimally choosing $b$ in~\eqref{eq.I-lj-poly-decay-intro}, one can deduce the exponential decay estimate
    \[I^\ell_j(t) \leq \exp(-\tfrac{1}{3} e^{-\beta t -1} \ell),\quad \text{for } j \leq \tfrac{1}{3} e^{-\beta t-1} \ell.\]
By correctly choosing $\ell$ and constraining $j$, using this estimate one can show for sufficient small times,
\[I^\ell_{j}(t)\sup_{s \in [0, t]} x_{j+\ell}(s)\leq C e^{C(j+\ell)t} I^\ell_{j}(t) \leq C \Big(\frac{j}{N}\Big)^{2(i+1)}.\]
From this, we then get for some $t_*$ and for all $t \leq t_*$, $j \leq C^{-1} N^{2/3}$,  
\[\int \Big|\frac{\gamma_j}{\rho^{\otimes j}}\Big|^2 \rho^{\otimes j}\,dx = x_j \leq C \Big(\frac{j}{N}\Big)^{2(i+1)}.\]
This is of course only a short time result and is essentially what is shown in Lemma~\ref{lem.l2-iter}. It turns out one can essentially iterate this argument to get the result for all times, though substantial care needs to be taken in propagating the correct estimates in time. See Lemma~\ref{lem.l2-iter} for details.

\begin{remark}
    We emphasize here that it is essential to consider 
    \[\int \Big|\frac{\gamma_j}{\rho^{\otimes j}}\Big|^2 \rho^{\otimes j}\,dx\] 
    instead of $\int |\gamma_j|^2\,dx$. That is all $L^2$-type norms should always be taken with respect to the measure $\rho$, not with respect to the Lebesgue measure. To see this, we note that the perturbative corrections $f^i_j$ are small in the $\rho$-adapted $L^2$-type norm but large in the $L^2$ norm with respect to Lebesgue measure. That is
    \[  \int \Big|\frac{f^i_j}{\rho^{\otimes j}}\Big|^2 \rho^{\otimes j}\,dx \leq C j^{2i} \text{ but } \int |f^i_j|^2\, dx \approx C^j,\]
    with a constant independent of $j$. This is particularly clear for $f^0_j = \rho^{\otimes j}$, for which 
        \[  \int \Big|\frac{\rho^{\otimes j}}{\rho^{\otimes j}}\Big|^2 \rho^{\otimes j}\,dx =1 \text{ but } \int |\rho^{\otimes j}|^2\, dx = \|\rho\|_{L^2}^j.\]
    This smallness in the $\rho$-adapted norm is important in many places. In particular, it allows us to show the remainder term in the equation for $\gamma_j$~\eqref{eq.gamma-j-bbgky} are appropriately small.
\end{remark}

\subsubsection{Organization of the argument}

In Section~\ref{s.algebra}, we give all of the algebraic results of the cluster expansion and perturbation theory. In Subsection~\ref{ss.cluster-expansion}, we introduce the cluster expansion and the hierarchy of equations solved by the terms of the cluster expansion. In Subsection~\ref{ss.cluster-perturbative}, we perturbatively expand the terms of the cluster expansion and introduce the hierarchy of equations solved by these perturbative approximations. We then use the perturbative expansion of the terms of the cluster expansion to construct a perturbative expansion of the marginal densities. We then note the equations solved by the terms of the perturbative expansion of the marginal densities. In Subsection~\ref{ss.basic-properties}, we supply a proof of Proposition~\ref{prop.g-ij-intro}. Theorem~\ref{thm.f-ij-existence-intro} is a direct consequence of the results of Section~\ref{s.algebra}, as will be made clear. Many of the proofs of the propositions stated in Section~\ref{s.algebra} will be deferred to Section~\ref{s.algebraic-proofs}, as they elementary and unenlightening. 

In Section~\ref{s.hierarchy}, we proceed with the analytic work of proving Theorem~\ref{thm.main-result-intro}. In order to justify the integration by parts necessary to show~\eqref{eq.l2-hierarchy-intro}, in this section we restrict our analysis to the case that $\Omega = \T^d$ and the initial data $f$ is such that $f + f^{-1} \in L^\infty(\T^d)$. This is done to guarantee the formal integration by parts computations of Section~\ref{s.hierarchy} are justified. We start by proving a hierarchical ``energy estimate'' for the difference between $f_{j,N}$ and its perturbative approximation to finite order. The resulting bound can be viewed as a hierarchy of differential inequalities only involving time derivatives. We then note basic estimates of the terms involved, though the proofs of these estimates are deferred to the end of the section, Subsection~\ref{ss.proofs-of-bounds}, in order to not distract from the main analytic techniques for showing the $L^2$ bound. In Subsection~\ref{ss.ODE}, we prove estimates on hierarchies of differential inequalities. In Subsection~\ref{ss.proof-of-main-result}, we use the estimates of Subsection~\ref{ss.ODE} together with the ``energy estimate'' to prove Theorem~\ref{thm.main-result-intro} in this restricted setting, giving Proposition~\ref{prop.main-result-special}. The resulting constants are uniform in the $L^\infty$ bounds on $f + f^{-1}$ as well as the sidelength of the torus. In Section~\ref{s.extend}, we use this uniformity in the constants to complete the proof of Theorem~\ref{thm.main-result-intro} in the general setting. 

\subsection{Acknowledgments}

The authors thank Daniel Lacker and Sylvia Serfaty for stimulating discussions about the presentation and context of the main results. EHC acknowledges the support of the Natural Sciences and Engineering Research Council of Canada (NSERC), [funding reference number PGSD3-557776-2021]. La recherche de EHC a  \'et\'e financ\'ee par le Conseil de recherches en sciences naturelles et en g\'enie du Canada (CRSNG), [num\'ero de r\'ef\'erence PGSD3-557776-2021]. KR was partially supported by NSF grants DMS-1954357 and DMS-2000200 as well as a Simons Foundation grant.

\section{Cluster expansions and perturbation theory}
\label{s.algebra}

For the remainder of the paper, we suppress the dependence of $f_{j,N}$ on $N$, simply writing $f_j$. In order to simplify the presentation of the algebra, we introduce abstract notation for the operators appearing the in BBGKY hierarchy. 
\begin{definition}
    \label{def.H-k-S-kl}
    Let $P \subseteq \N \cup \{*\}$\footnote{Note here we are taking $*$ as some index distinct from $k \in \N$. It will always be used for the non-local operator $H_k$, for which it acts an index for the integration variable.} with $|P| <\infty$ and $h : \Omega^P \to \R$. Then for any $k,\ell \in P$ such that $k,\ell \ne *$, we define
    \[S_{k,\ell} h : \Omega^P \to \R\]
    by
    \[S_{k,\ell} h(x) := \nabla_{x_k} \cdot (K(x_k,x_\ell) h(x)).\]
    Then, provided $* \in P$, for any $k \in P$ such that $k \ne *$, we define
    \[H_k h : \Omega^{P - \{*\}} \to \R\]
    by
    \[H_k h(x^{P - \{*\}}) := \nabla_{x_k} \cdot \int K(x_k, x_*) h(x)\,dx_*.\]
\end{definition}

With this notation, we can rewrite the BBGKY hierarchy~\eqref{eq.bbgky-sketch} abstractly as
\begin{equation}\label{eq.hierarchy-equation}
    \partial_t f_{j} - \Delta f_{j} + \frac{N-j}{N}\sum_{k \in [j]} H_k f_{[j] \cup \{*\}} + \frac{1}{N} \sum_{k,\ell \in [j]} S_{k,\ell} f_{j} =0.
\end{equation}

\subsection{Cluster expansion}

\label{ss.cluster-expansion}

We now introduce the cluster expansion of the $f_j$.

\begin{definition}
    Let $g_j : \Omega^{j} \to \R$ be the exchangeable functions given by
    \begin{equation}\label{eq.forward-cluster-definition}
        g_j := \sum_{\pi \vdash [j]} (-1)^{|\pi| -1} (|\pi|-1)!\prod_{P \in \pi} f_P.
    \end{equation}
    We call the function $g_j$ the \textit{$j$th cluster function of the distribution.}
\end{definition}

\begin{remark}
    We note the expression of $g_j$ in terms of the $f_k$ is analogous to the expression the joint cumulant of a collection of $j$ random variable in terms of their joint moments.
\end{remark}

Although this dependence is suppressed, the $g_j$ depend on $N$ through their dependence on the $f_j$. The $g_j$ are defined exactly so that the following expansion for the $f_j$ holds.

\begin{proposition}\label{prop.cluster-representation}
    \[f_j = \sum_{\pi \vdash [j]} \prod_{P \in \pi} g_P.\]
\end{proposition}

This is a classical combinatorial fact frequently used to relate moments and cumulants of random variables. The proof is found in Section~\ref{s.algebraic-proofs}.

\begin{remark}
    The $g_j$ have the marginalization property that for any $j\geq 2$ and any $1 \leq \ell \leq j$,
    \[\int g_j\,dx_\ell =0.\]
    We don't need this property, so we omit its proof. The argument uses the same elementary combinatorics as the rest of the proofs of the results of this section. We will however use the same marginalization property for the terms perturbative expansion of the $g_j$, $g^i_j$, which is noted in Lemma~\ref{lem.marginalization}.
\end{remark}

By taking the time derivative of~\eqref{eq.forward-cluster-definition} and using the BBGKY hiearchy~\eqref{eq.hierarchy-equation}, we see that the $g_j$ themselves solve equations, which we now give.

\begin{proposition}\label{prop.cluster-function-equation}
    For fixed $N$, the cluster functions $g_j$, $1\leq j\leq N$, solve the hierarchy of equations 
    \begin{align}
        \partial_t g_j - \Delta g_j &\notag= -\frac{N-j}{N}\sum_{k=1}^j H_k g_{[j]\cup\{*\}}+ \sum_{k=1}^j \sum_{W \subseteq [j] - \{k\}} \frac{j-1 - |W|}{N} H_k g_{W \cup \{k,*\}} g_{[j] - \{k\} - W}
        \\&\notag\quad - \frac{N-j}{N}\sum_{k=1}^j \sum_{W \subseteq [j] -\{k\}} H_k g_{W \cup \{k\}} g_{[j]\cup \{*\} - W -\{k\}}
        \\&\notag\quad + \sum_{k=1}^j \sum_{W \subseteq [j] - \{k\}} \sum_{R \subseteq [j] -\{k\} - W}\frac{j -1 - |W| - |R|}{N} H_k g_{W \cup \{k\}}  g_{R \cup \{*\}} g_{[j] - R - W-\{k\}}
        \\&\quad - \frac{1}{N}\sum_{k,\ell=1}^j S_{k,\ell} g_j- \frac{1}{N} \sum_{\substack{k,\ell=1\\ k \ne \ell}}^j \sum_{W \subseteq [j] - \{k,\ell\}} S_{k,\ell} g_{W \cup \{k\}} g_{[j] - \{k\} - W},\label{eq.cluster-pde}
    \end{align}
    with initial conditions
    \[
    g_j(0,\cdot)=
    \begin{cases}
        f&j=0,\\
        0&j\geq 1.
    \end{cases}
    \]
\end{proposition}
The proof of this proposition involves expanding out the $g_j$ in terms of $f_i$, using the equations for $f_i$, and then re-expanding the $f_i$ in terms of $g_k$. One must then carefully collect constant factors before identical terms. We defer the unenlightening proof to Section~\ref{s.algebraic-proofs}.

\subsection{Perturbative expansion of the cluster functions}
\label{ss.cluster-perturbative}

We note this subsection primarily consists of formal arguments, motivating the correct equations for $g^i_j$ and $f^i_j$. The actual analytic content of this section isn't realized until we prove that this formal perturbation theory gives good approximations to the true marginal densities in Section~\ref{s.hierarchy}.

We are interested in computing solutions to this hierarchy perturbatively in $N^{-1}$. Thus we now take the perturbative ansatz for $g_j$,
\[g_j = \sum_{i=0}^\infty N^{-i} g_j^i,\]
where the $g_j^i$ are assumed to be $N$-independent. Plugging this into \eqref{eq.cluster-pde} and collecting orders of $N^{-1}$, we find that such $g_j^i$ should be solutions to
\begin{align}
\notag\lefteqn{\partial_t g^i_j - \Delta g^i_j+\sum_{k=1}^j H_k g^0_{\{k\}}  g^{i}_{[j]\cup \{*\} - \{k\}}+\sum_{k=1}^j H_k g^i_{[j]}  g^{0}_{\{*\}}}\quad&
\\&\notag=-\sum_{k=1}^j H_k g^i_{[j]\cup\{*\}}-\sum_{k=1}^j \sum_{W \subseteq [j] -\{k\}} \sum_{m=1}^{i-1}H_k g^m_{W \cup \{k\}}  g^{i-m}_{[j]\cup \{*\} - W -\{k\}}
\\&\notag\quad +j \sum_{k=1}^j H_k g^{i-1}_{[j]\cup\{*\}}+\sum_{k=1}^j \sum_{W \subseteq [j] - \{k\}} (j-1 - |W|) \sum_{m=0}^{i-1} H_k g^m_{W \cup \{k,*\}} g^{i-1-m}_{[j] - \{k\} - W}
 \\&\notag\quad +j\sum_{k=1}^j \sum_{W \subseteq [j] -\{k\}} \sum_{m=0}^{i-1}H_k g^m_{W \cup \{k\}} g^{i-1-m}_{[j]\cup \{*\} - W -\{k\}}
        \\&\notag\quad  +\sum_{k=1}^j \sum_{W \subseteq [j] - \{k\}} \sum_{R \subseteq [j] -\{k\} - W} (j -1 - |W| - |R|) \sum_{m=0}^{i-1} \sum_{n=0}^{i-1-m} H_k g^m_{W \cup \{k\}}  g^n_{R \cup \{*\}} g^{i-1-m-n}_{[j] - R - W - \{k\}}
        \\&\quad - \sum_{k,\ell=1}^j S_{k,\ell} g^{i-1}_j- \sum_{\substack{k,\ell=1\\k\neq \ell}}^j \sum_{W \subseteq [j] - \{k,\ell\}} \sum_{m=0}^{i-1} S_{k,\ell} g^m_{W \cup \{k\}} g^{i-1-m}_{[j] - \{k\} - W}, \label{eq.g-ij-PDE}
\end{align}
where we take the convention that $g^{-1}_j = 0$ for any $j$ and $g^i_0 =0$ for any $i$. We also find that they should have initial conditions
\begin{equation}\label{eq.g-ij-initial-conditions}
g^i_j(0,\cdot)=
\begin{cases}
    f&i=0,j=1,\\
    0&\text{otherwise}.
\end{cases}
\end{equation}

Now that we have a representation of the perturbative expansion for the cluster functions $g_j$, we turn our attention back to the marginals $f_j$. We seek a representation of their perturbative expansion. To that end we write the formal expansions
\[\sum_{i=0}^\infty N^{-i} f^i_j = f_j = \sum_{\pi \vdash [j]} \prod_{P \in \pi} g_P =  \sum_{\pi \vdash [j]} \prod_{P \in \pi} \sum_{i_P=0}^\infty N^{-i_P} g_P^{i_P} = \sum_{i=0}^\infty N^{-i} \sum_{\pi \vdash [j]} \sum_{\substack{(i_P)_{P \in \pi}\\ \sum i_P = i}} \prod_{P \in \pi} g^{i_P}_P.\]
Collecting terms by order, we get
\[f^i_j := \sum_{\pi \vdash [j]} \sum_{\substack{(i_P)_{P \in \pi}\\ \sum i_P = i}} \prod_{P \in \pi} g^{i_P}_P.\]
Further, plugging the perturbative expansion for $f_j = \sum_{i=0}^\infty N^{-i} f^i_j$ into the BBGKY hierarchy~\eqref{eq.hierarchy-equation} and collecting orders, we formally get
\[\partial_t f^i_j - \Delta f^i_j +\sum_{k=1}^j  H_k f^i_{[j]\cup\{*\}} = j \sum_{k=1}^j H_k f^{i-1}_{[j]\cup\{*\}} -\sum_{k,\ell=1}^j S_{k,\ell} f^{i-1}_j.\]
Since the $f^i_j$ are defined in terms of the $g^i_j$, which themselves solve equations, we need to check that under our definition of the $f^i_j$, this equation is in fact solved, as is given by the next proposition.

\begin{proposition}\label{prop.f-ij-equations}
    Let $f^i_j$ be defined by~\eqref{eq.f-ij-def} where the $g^i_j$ solve the hierarchy~\eqref{eq.g-ij-PDE}. Then the $f^i_j$ solve the hierarchy of equations~\eqref{eq.f-ij-eqs-intro}.
\end{proposition}

\begin{remark}
    Theorem~\ref{thm.f-ij-existence-intro} is an immediate consequence of this proposition and Proposition~\ref{prop.g-ij-intro}, which will be proved in the next subsection.
\end{remark}

The proof of Proposition~\ref{prop.f-ij-equations} also proceeds by tedious but elementary algebraic manipulation, so has been deferred to Section~\ref{s.algebraic-proofs}.

\subsection{Existence of the $g^i_j$}

\label{ss.basic-properties}

In this section we prove Proposition~\ref{prop.g-ij-intro} by constructing solutions to the hierarchy of equations~\eqref{eq.g-ij-PDE}. We work with the mild formulation of the equation to define our notion of solutions. As we will see in the proof of Proposition~\ref{prop.g-ij-intro}, $g^1_0$ will solve the equation
\[\begin{cases}
\partial_t g^1_0 - \Delta g^1_0 + \nabla \cdot \int K(x,x_*) g^1_0(x_*) g^1_0(x)\,dx_* = 0\\
g^1_0(0,\cdot) = f.
\end{cases}
\]
This makes $g^1_0$ special among the $g^i_j$ in two ways, first it is the only $g^i_j$ whose equation is nonlinear in $g^i_j$ and second it is the only $g^i_j$ with non-trivial initial data. We note that the equation for $g^1_0$ is the McKean-Vlasov equation~\eqref{eq.Mckean-Vlasov-equation}. For the remainder of the paper, we often use $\rho$ in place of $g^1_0$.

We now prove that there actually is a solution to the hierarchy~\eqref{eq.g-ij-PDE}, which is the content of Proposition~\ref{prop.g-ij-intro}.
\begin{proof}[Proof of Proposition~\ref{prop.g-ij-intro}.]
The proof proceeds in two steps.

 \textit{Step 1:} We check that if we take $g_j^i=0$ for all $(i,j)\notin D$, then this does not contradict the equations \eqref{eq.g-ij-PDE}. That is to say, we just need to verify that if $(i,j)\notin D$, then all terms in the right hand side of the equation for $g_j^i$ involve $g_\ell^k$ for some $(k,\ell)\notin D$.  This is easy to check for all terms which do not involve products of the $g^k_\ell$. There are 5 terms which do involve products. We check the first term (in the order they appear in~\eqref{eq.g-ij-PDE}), which already sharply generates the constraint $j \geq i+2,$ the others follow similarly. Consider
\[H_k g^m_{W \cup \{k\}}  g^{i-m}_{[j]\cup \{*\} - W -\{k\}},\]
for $k \in [j], W \subseteq [j] - \{k\}.$ Then the only way $(m,|W|+1)$ and $(i-m,j-|W|)$ are both in $D$ is if
    \[|W| +1 \leq m+1\ \text{and}\ j-|W| \leq i-m+1.\]
    Simplifying and combining these constraints we find that $j\leq i+1,$ which contradicts the assumption that $(i,j)\notin D.$ The analysis of all other terms follows directly analogously.\footnote{In particular the first, fourth, and fifth terms sharply give the constraint $j \geq i+2$, while the second and third terms are not sharp.}

    \textit{Step 2:} Using Step 1, we will now make the \textit{a priori} assumption that $g_j^i=0$ for all $(i,j)\notin D$. We will inductively show unique existence  for function $g_j^i$ with $(i,j)\in T$ using the ordering defined above.
    
    First, in the base case $(i,j) = (0,1)$, the equation for $g_1^0$ reduces to the McKean-Vlasov equation~\eqref{eq.Mckean-Vlasov-equation}, for which unique existence in $C^0([0,\infty), L^1(\Omega))$ follows from a Picard iteration argument using the mild formulation.
    
     Now, assuming that $g^k_\ell$ have been shown to uniquely exist for $(k,\ell)<(i,j)\in D$, we consider the equation for $g_j^i$. We note that all the terms on the right hand side of the equation only involve terms which are zero or satisfy $(k, \ell)<(i,j)$, while the terms on the left hand side are linear in $g_j^i$. A Picard iteration using the mild formulation then gives unique existence of a $C^0_{loc}([0,\infty), L^1(\Omega))$ solution to~\eqref{eq.g-ij-PDE}. This completes the induction.
\end{proof}

\section{Hierarchy bounds}

\label{s.hierarchy}

In this section we prove Theorem~\ref{thm.main-result-intro} in the restricted setting of $\Omega = \T^d$ and $f+ f^{-1} \in L^\infty(\T^d).$ In this setting, the necessary integrations by parts appearing in the proofs of Propositions~\ref{prop.l2-hierarchy} and~\ref{prop.g-ij-bounds} are easily justified. In particular, by restricting to $\T^d$, we do not have to consider issues of decay at $\infty$ in the integration by parts. Additionally, by basic parabolic theory applied to their respective equations, we have that 
\[\rho+\rho^{-1}\in L^\infty_{loc}([0,\infty), L^\infty(\T^d)) \text{ and } f_j,g^i_j \in L^\infty_{loc}([0,\infty), L^\infty(\T^{dj})),\]
which makes rigorous justification of the formal computations a standard approximation argument. We note the presence of $\rho^{-1}$ makes it challenging to justify these computations without such assumptions. We emphasize that all bounds attained in this section are uniform in $\|f+f^{-1}\|_{L^\infty(\T^d)}$ as well as the sidelength of the torus. This will be needed in Section~\ref{s.extend} to extend the results to remove the restrictions of this section and complete the proof of Theorem~\ref{thm.main-result-intro}.  

We now introduce some additional notation.
\begin{definition}
    Letting $f^i_j$ be defined by~\eqref{eq.f-ij-def}, we let
         \begin{gather}
    \label{eq.phi-ij-def}
     \phi^i_j := \sum_{k=0}^i N^{-k} f^k_j,\\\notag
     R^i_j := \frac{1}{N^{i+1}} \sum_{k=1}^j e_k\otimes\sum_{\ell=1}^j \int K(x_k,x_*) f^i_{[j]\cup\{*\}}dx_*-K(x_k,x_\ell) f^i_j.
     \end{gather}
\end{definition}

\begin{remark}
    The tensor product notation used in the definition of $R^i_j$ is given such that 
    \[\nabla \cdot R^i_j = \frac{1}{N^{i+1}} \sum_{k=1}^j \nabla_{x_k}\cdot\sum_{\ell=1}^j \int K(x_k,x_*) f^i_{[j]\cup\{*\}}dx_*-K(x_k,x_\ell) f^i_j,\]
    where $\nabla \cdot$ denotes the divergence on $\T^{jd}.$
\end{remark}

One can readily check using the equations the $f^i_j$ solve that $\phi^i_j$ solves the following equation.
\begin{equation}
\label{eq.phi-ij-r-ij}
    \partial_t \phi_j^i - \Delta \phi_j^i +  \frac{N-j}{N} \sum_{k=1}^j \nabla_{x_k}\cdot\int K(x_k, x_{*}) \phi_{j+1}^i(x^{[j]\cup\{*\}})\,dx_* + \frac{1}{N}  \sum_{k,\ell=1}^j \nabla_{x_k} \cdot(K(x_k, x_\ell) \phi_j^i) = \nabla\cdot R_j^i.
\end{equation}

We now show the essential $L^2$ energy-type estimate for difference $\phi^i_j - f_j$. We note that at $t=0$, $\phi^i_j = f_j$, so this estimate allows us to control the size of $\phi^i_j - f_j$ for $t>0$ by a Gr\"onwall-type argument. We also give a somewhat brutal bound on the growth of $f_j$ that doesn't depend on $\phi^i_j$. We will use this brutal bound to ``close'' the hierarchy.

\begin{proposition}
    \label{prop.l2-hierarchy}
    Suppose $f +f^{-1}\in L^\infty(\T^d)$ and $K \in L^\infty_\delta(\T^{2d}).$ Then letting
    \begin{equation}
    \label{eq.gamma-ij-def}
        \gamma^i_j := \phi^i_j-f_j,
    \end{equation}
    we have that
    \begin{align}
        \frac{d}{dt} \int \Big|\frac{\gamma^i_j}{\rho^{\otimes j}}\Big|^2 \rho^{\otimes j}\,dx &\leq 2 j \|K\|_{L^\infty_\delta}^2 \Bigg(\int \Big|\frac{\gamma^i_{j+1}}{\rho^{\otimes (j+1)}}\Big|^2 \rho^{\otimes (j+1)}\,dx_*dx - \int \Big|\frac{\gamma^i_j}{\rho^{\otimes j}}\Big|^2 \rho^{\otimes j}\,dx \Bigg)
        \notag\\&\qquad + 4\frac{j^3}{N^2}\|K\|_{L^\infty_\delta}^2 \int \Big|\frac{\gamma^i_j}{\rho^{\otimes j}}\Big|^2 \rho^{\otimes j}\,dx + 2 \int \Big|\frac{R^i_j}{\rho^{\otimes j}}\Big|^2 \rho^{\otimes j}\,dx.
        \label{eq.l2-hierarchy}
    \end{align}
    We also have that
    \begin{equation}
        \label{eq.l2-apriori}
        \frac{d}{dt} \int \Big|\frac{f_j}{\rho^{\otimes j}}\Big|^2 \rho^{\otimes j}\,dx \leq 12 j \|K\|_{L^\infty_\delta}^2 \int \Big|\frac{f_j}{\rho^{\otimes j}}\Big|^2 \rho^{\otimes j}\,dx.
    \end{equation}
\end{proposition}

\begin{remark}
    Note the interesting---and essential---property of this estimate that all constants are independent of $\rho$.
\end{remark}

We will use the following marginalization property of the $g^i_j$ functions.
\begin{lemma}
\label{lem.marginalization}
    For $g^i_j$ as given in Proposition~\ref{prop.g-ij-intro}, for any $i,j$ and any $1 \leq \ell \leq j,$
    \[\int g^i_j\,dx_\ell = \begin{cases}
        1 & (i,j)=(0,1),\\
        0 & \text{otherwise.}
    \end{cases}\]
    Thus
    \[\int f^i_{j+1}\,dx_{j+1} = f^i_j \text{ and } \int \phi^i_{j+1}\,dx_{j+1} = \phi^i_j.\]
\end{lemma}
As the proof of the above is uninteresting we defer it to Section~\ref{ss.marginalization}.

\begin{proof}[Proof of Proposition~\ref{prop.l2-hierarchy}]
    For notational simplicity, let us fix $i \in \N$ and drop the $i$ dependence, writing $\gamma_j$ for $\gamma_j^i$. We note that $\gamma_j$ solves the equation 
    \[\partial_t \gamma_j -\Delta \gamma_j +\frac{N-j}{N} \sum_{k=1}^j \nabla_{x_k} \cdot \int K(x_k,x_*)\gamma_{j+1}(x^{[j] \cup \{*\}})\,dx_* +\frac{1}{N} \sum_{k,\ell=1}^j \nabla_{x_k} \cdot (K(x_k,x_\ell)\gamma_j) = \nabla \cdot R_j,\]
    where $R_j = R^i_j.$ We also have that
    \[\partial_t \rho^{\otimes j} - \Delta \rho^{\otimes j} + \sum_{k=1}^j \nabla_{x_k} \cdot \int K(x_k,x_*) \rho(x_*)\,dx_*\rho^{\otimes j} =0.\]
    We then compute
    \begin{align*}
        \frac{d}{dt} \int \frac{\gamma_j^2}{\rho^{\otimes j}}\,dx &= \int 2\frac{\gamma_j}{\rho^{\otimes j}} \partial_t \gamma_j - \frac{\gamma_j^2}{(\rho^{\otimes j})^2} \partial_t \rho^{\otimes j}\,dx
        \\&= \int - 2 \nabla \frac{\gamma_j}{\rho^{\otimes j}} \cdot \nabla \gamma_j + 2 \frac{N-j}{N} \nabla \frac{\gamma_j}{\rho^{\otimes j}} \cdot \sum_{k=1}^j e_k\otimes \int K(x_k,x_*) \gamma_{j+1}(x^{[j] \cup \{*\}})\,dx_*
        \\&\qquad + \frac{2}{N} \nabla \frac{\gamma_j}{\rho^{\otimes j}} \cdot \sum_{k,\ell=1}^j e_k\otimes K(x_k,x_\ell)\gamma_j - 2 \nabla \frac{\gamma_j}{\rho^{\otimes j}} \cdot R_j 
        \\&\qquad + 2 \frac{\gamma_j}{\rho^{\otimes j}} \nabla \frac{\gamma_j}{\rho^{\otimes j}} \cdot \nabla \rho^{\otimes j} - 2 \frac{\gamma_j}{\rho^{\otimes j}} \nabla \frac{\gamma_j}{\rho^{\otimes j}} \cdot \sum_{k=1}^j e_k\otimes \int K(x_k,x_*) \rho(x_*)\,dx_* \rho^{\otimes j} \,dx.
    \end{align*}
    We note that
    \[\rho^{\otimes j}\nabla \frac{\gamma_j}{\rho^{\otimes j}} = \nabla \gamma_j - \frac{\gamma_j}{\rho^{\otimes j}} \nabla \rho^{\otimes j}.\]
    Thus
    \[\int - 2 \nabla \frac{\gamma_j}{\rho^{\otimes j}} \cdot \nabla \gamma_j + 2 \frac{\gamma_j}{\rho^{\otimes j}} \nabla \frac{\gamma_j}{\rho^{\otimes j}} \cdot \nabla \rho^{\otimes j}\,dx = -2 \int \Big| \nabla \frac{\gamma_j}{\rho^{\otimes j}} \Big|^2 \rho^{\otimes j}\,dx.\]
    We then group terms,
    \begin{align*}
        \frac{d}{dt} \int \frac{\gamma_j^2}{\rho^{\otimes j}}\,dx &= -2\int \Big| \nabla \frac{\gamma_j}{\rho^{\otimes j}} \Big|^2 \rho^{\otimes j}\,dx 
        \\&\qquad+ 2 \frac{N-j}{N}\int  \nabla\frac{\gamma_j}{\rho^{\otimes j}} \cdot \sum_{k=1}^j e_k\otimes \int K(x_k,x_*)\Big( \frac{\gamma_{j+1}}{\rho^{\otimes (j+1)}} - \frac{\gamma_j}{\rho^{\otimes j}} \Big) \rho(x_*)\,dx_*\ \rho^{\otimes j}\,dx
        \\&\qquad - 2\frac{j}{N} \int \nabla \frac{\gamma_j}{\rho^{\otimes j}} \cdot \sum_{k=1}^j e_k\otimes \int K(x_k,x_*) \rho(x_*)\,dx_* \frac{\gamma_j}{\rho^{\otimes j}}\rho^{\otimes j}\,dx
        \\&\qquad+ \frac{2}{N}\int  \nabla \frac{\gamma_j}{\rho^{\otimes j}} \cdot \sum_{k,\ell=1}^j e_k\otimes K(x_k,x_\ell) \frac{\gamma_j}{\rho^{\otimes j}} \rho^{\otimes j} \,dx  - 2\int \nabla \frac{\gamma_j}{\rho^{\otimes j}} \cdot \frac{R_j}{\rho^{\otimes j}} \rho^{\otimes j} \,dx.
    \end{align*}
    Thus applying Young's inequality, we see that
    \begin{align*}
    \frac{d}{dt} \int \frac{\gamma_j^2}{\rho^{\otimes j}}\,dx &\leq 2j\int \Big|\int K(x_1,x_*) \Big( \frac{\gamma_{j+1}}{\rho^{\otimes (j+1)}} - \frac{\gamma_j}{\rho^{\otimes j}} \Big) \rho(x_*)\,dx_*\Big|^2 \rho^{\otimes j}\,dx
    \\&\qquad+ 4\frac{j^3}{N^2}\|K\|_{L^\infty_\delta}^2 \int \frac{\gamma_j^2}{\rho^{\otimes j}}\,dx + 2 \int \Big|\frac{R_j}{\rho^{\otimes j}}\Big|^2 \rho^{\otimes j}\,dx.
    \end{align*}
    We then note that by H\"older's inequality,
    \begin{align*}
        &\Big|\int K(x_1,x_*) \Big( \frac{\gamma_{j+1}}{\rho^{\otimes (j+1)}} - \frac{\gamma_j}{\rho^{\otimes j}} \Big) \rho(x_*)\,dx_*\Big|^2
        \\&\qquad\leq \int K(x_1,x_*)^2 \rho(x_*)\,dx_* \int \Big( \frac{\gamma_{j+1}}{\rho^{\otimes (j+1)}} - \frac{\gamma_j}{\rho^{\otimes j}} \Big)^2 \rho(x_*)\,dx_*
        \\&\qquad\leq \|K\|_{L^\infty_\delta}^2 \Bigg(\int \Big|\frac{\gamma_{j+1}}{\rho^{\otimes (j+1)}}\Big|^2 \rho(x_*)\,dx_*  - 2 \frac{\gamma_j}{(\rho^{\otimes j})^2} \int \gamma_{j+1}\,dx_*+ \Big|\frac{\gamma_j}{\rho^{\otimes j}}\Big|^2\int \rho(x_*)\,dx_*\Bigg)
        \\&\qquad= \|K\|_{L^\infty_\delta}^2 \Bigg(\int \Big|\frac{\gamma_{j+1}}{\rho^{\otimes (j+1)}}\Big|^2 \rho(x_*)\,dx_* - \Big|\frac{\gamma_j}{\rho^{\otimes j}}\Big|^2 \Bigg),
    \end{align*}
    where we use Lemma~\ref{lem.marginalization} for the last line.
    Combining this with the previous inequality, we get~\eqref{eq.l2-hierarchy}. Turning our attention~\eqref{eq.l2-apriori}, we note that repeating that above computations with $f_j$ in place of $\gamma_j$, we get that 
    \[
    \frac{d}{dt} \int \frac{f_j^2}{\rho^{\otimes j}}\,dx \leq 2j\int \Big|\int K(x_1,x_*) \Big( \frac{f_{j+1}}{\rho^{\otimes (j+1)}} - \frac{f_j}{\rho^{\otimes j}} \Big) \rho(x_*)\,dx_*\Big|^2 \rho^{\otimes j}\,dx
    + 4\frac{j^3}{N^2}\|K\|_{L^\infty_\delta}^2 \int \frac{f_j^2}{\rho^{\otimes j}}\,dx.\]
    Then we note that 
    \[ \Big|\int K(x_1,x_*) \Big( \frac{f_{j+1}}{\rho^{\otimes (j+1)}} - \frac{f_j}{\rho^{\otimes j}} \Big) \rho(x_*)\,dx_*\Big| \leq \|K\|_{L^\infty_\delta} \Big(\frac{f_j}{\rho^{\otimes j}} 
 + \frac{1}{\rho^{\otimes j}} \int f_{j+1}\,dx_*\Big) = 2 \|K\|_{L^\infty_\delta} \frac{f_j}{\rho^{\otimes j}}.\]
    Thus
    \[
    \frac{d}{dt} \int \frac{f_j^2}{\rho^{\otimes j}}\,dx \leq 8 j \|K\|_{L^\infty_\delta}^2 \int \Big|\frac{f_j}{\rho^{\otimes j}}\Big|^2 \rho^{\otimes j}\,dx +4\frac{j^3}{N^2}\|K\|_{L^\infty_\delta}^2 \int \frac{f_j^2}{\rho^{\otimes j}}\,dx \leq 12 j \|K\|_{L^\infty_\delta}^2 \int \Big|\frac{f_j}{\rho^{\otimes j}}\Big|^2 \rho^{\otimes j}\,dx,\]
    giving~\eqref{eq.l2-apriori}.
\end{proof}

With the energy estimate~\eqref{eq.l2-hierarchy} in hand, we now need to understand how to bound hierarchies of differential inequalities of the above sort. This is the focus of Subsection~\ref{ss.ODE}. Before that though, we need to note bounds on the terms involved. The bound given by Proposition~\ref{prop.r-ij-bound} is essential for estimating the contribution of the remainder terms of~\eqref{eq.l2-hierarchy}; the bounds of Proposition~\ref{prop.f-ij-bound} will turn out to be useful as well for somewhat subtler reasons. We defer the combinatorial proofs of these bounds to Subsection~\ref{ss.proofs-of-bounds} so as not to distract from the heart of the argument. 

\begin{proposition}
    \label{prop.f-ij-bound}
    Suppose $f +f^{-1}\in L^\infty(\T^d)$ and $K \in L^\infty_\delta(\T^{2d}).$ Then there exists $C(\|K\|_{L^\infty_\delta},i)< \infty$ such that 
        \[\int \bigg|\frac{f^i_j}{\rho^{\otimes j}}\bigg|^2 \rho^{\otimes j}\,dx\leq C e^{Ct} j^{2i},\]
    and so
    \[\int \bigg|\frac{\phi^i_j}{\rho^{\otimes j}}\bigg|^2 \rho^{\otimes j}\,dx\leq C e^{Ct}.\]
\end{proposition}

\begin{proposition}
    \label{prop.r-ij-bound}
    Suppose $f +f^{-1} \in L^\infty(\T^d)$ and $K \in L^\infty_\delta(\T^{2d})$. Then there exists $C(\|K\|_{L^\infty_\delta},i)< \infty$ such that 
        \[\int \bigg|\frac{R^i_j}{\rho^{\otimes j}}\bigg|^2\rho^{\otimes j}\,dx\leq C e^{Ct} \Big(\frac{j}{N}\Big)^{2(i+1)}.\]
\end{proposition}
\begin{remark}
    We note that the $j$ dependence of the above bounds is a consequence of certain $L^2$ orthogonality implicit in the definitions of the $f^i_j$, which in turn is a consequence of the $0$ marginalization property of the $g^i_j$ noted in Lemma~\ref{lem.marginalization}. Without exploiting this $L^2$ orthogonality, one can derive similar bounds but with worse rates in $j$. This then propagates to worse rates in $j$ in Theorem~\ref{thm.main-result-intro}. Thus the marginalization property of Lemma~\ref{lem.marginalization} is essential to getting good bounds in $j$.
\end{remark}

\subsection{ODE hierarchy estimates}
\label{ss.ODE}

We now consider passing estimates on hierarchies of differential inequalities of the form~\eqref{eq.l2-hierarchy}. By repeatedly applying the Gr\"onwall inequality to the hierarchy, iterated exponential integrals appear. We introduce the following notation for these integrals.

\begin{definition}
    \label{def.I-lj-def}
    For $\beta := 4\|K\|_{L^\infty_\delta}^2,$ let
    \[I^\ell_j(t) := \beta^\ell \frac{(j+\ell-1)!}{(j-1)!} e^{-\beta j t} \int_0^t \int_0^{s_\ell} \cdots \int_0^{s_{2}} e^{-\beta \sum_{k=2}^{\ell} s_k} e^{\beta (j+\ell-1) s_1} \,ds_1 \cdots ds_\ell,\]
    where $I^0_j(t) := 1$ by convention.
\end{definition}

The $I^\ell_j$ are related in the following way.

\begin{proposition}
    \label{prop.phi-int}
    \[\beta j e^{-\beta j t} \int_0^t e^{\beta j s} I_{j+1}^\ell(s)\,ds = I^{\ell+1}_j(t).\]
\end{proposition}
\begin{proof}
    If $\ell =0,$ we see that
    \[\beta j e^{-\beta j t} \int_0^t e^{\beta j s} I_{j+1}^0(s)\,ds = \beta j e^{-\beta j t} \int_0^t e^{\beta (j+1 -1) s_1}\,ds_1 = I^{1}_j.\]
    Otherwise, we compute
    \begin{align*}
        &\beta j e^{-\beta j t} \int_0^t e^{\beta j s_{\ell+1}} I_{j+1}^\ell(s_{\ell+1})\,ds_{\ell+1} 
        \\&\quad = \beta^{\ell +1} \frac{(j+\ell)!}{(j-1)!} e^{-\beta j t}\int_0^t e^{\beta j s_{\ell+1}} e^{-\beta (j+1)s_{\ell+1}} \int_0^{s_{\ell+1}} \cdots \int_0^{s_2} e^{-\beta \sum_{k=2}^\ell s_k}  e^{\beta (j+1+\ell-1)s_1}\,ds_1 \cdots ds_\ell ds_{\ell+1}
        \\&\quad = \beta^{\ell +1} \frac{(j+\ell+1 -1)!}{(j-1)!} e^{-\beta j t}\int_0^t \int_0^{s_{\ell+1}} \cdots \int_0^{s_2} e^{-\beta \sum_{k=2}^{\ell+1} s_k}  e^{\beta (j+\ell+1 - 1)s_1}\,ds_1 \cdots ds_{\ell+1}
        \\&\quad = I^{\ell+1}_j(t),
    \end{align*}
    as desired.
\end{proof}

It is prefactors of $I^\ell_j(t)$ that will give sufficient decay when iterating up the hierarchy to prove the bounds we require. As such, we need to understand how the $I^\ell_j$ decay as $\ell$ gets large. The following proposition is the first such estimate and follows from a simple induction.

\begin{proposition}
    \label{prop.polynomial-bound}
    For any $j,\ell \in \N$ and any $b \in \N,$
    \[I^\ell_j(t) \leq \Big(\frac{j+b}{j+\ell}\Big)^b e^{\beta b t}.\]
\end{proposition}

\begin{proof}
    We note that for any $b \geq 0,$
    \begin{align*}
        &\int_0^t \int_0^{s_\ell} \cdots \int_0^{s_{2}} e^{-\beta \sum_{k=2}^{\ell} s_k} e^{\beta (j+\ell-1) s_1} \,ds_1 \cdots ds_\ell 
        \\&\quad\leq \int_0^t \int_0^{s_\ell} \cdots \int_0^{s_3}  e^{-\beta \sum_{k=2}^{\ell} s_k}\int_0^{s_{2}}  e^{\beta (j+\ell-1+b) s_1} \,ds_1 \cdots ds_\ell
        \\&\quad \leq \frac{1}{\beta(j+\ell - 1+b)}\int_0^t \int_0^{s_\ell} \cdots \int_0^{s_3}  e^{-\beta \sum_{k=2}^{\ell} s_k} e^{\beta(j+\ell -1 +b)s_2}  ds_2 \cdots ds_\ell 
        \\&\quad= \frac{1}{\beta(j+\ell-1+b)}\int_0^t \int_0^{s_\ell} \cdots \int_0^{s_4}e^{-\beta \sum_{k=3}^{\ell} s_k} \int_0^{s_3}   e^{\beta(j+\ell-2+b )s_2}  ds_2 \cdots ds_\ell
        \\&\quad \leq \cdots \leq e^{\beta (j+b)t} \prod_{i=1}^{\ell} \frac{1}{\beta (j+\ell-i+b)} = e^{\beta (j+b)t} \beta ^{-\ell}\prod_{i=0}^{\ell-1} \frac{1}{j+i+b}.
    \end{align*}
    Thus, for $b \in \N$, exploiting cancellation in the product,
    \[I^\ell_j(t) \leq e^{\beta b t}\prod_{i=0}^{\ell-1} \frac{j+i}{j+i+b} \leq \Big(\frac{j+b}{j+\ell}\Big)^b e^{\beta b t},\]
    allowing us to conclude.
\end{proof}

For some of the estimates, the above polynomial decay will be sufficient; for others, we will need an exponential rate of decay. This exponential rate can be found by simply choosing the polynomial power $b$ optimally in a time dependent way, as the below proposition shows.

\begin{proposition}
    \label{prop.exp-bound}
    For any $j, \ell \in \N$ and for any $t \geq 0,$ if
    \[j \leq \tfrac{1}{3} e^{-2\beta t-1} \ell,\]
    then
    \[I^\ell_j(t) \leq \exp(-\tfrac{1}{3} e^{-2\beta t -1} \ell).\]
\end{proposition}

\begin{remark}
    The above proposition is analogous to~\cite[Proposition 5.1]{lacker_hierarchies_2023}, although with a different proof using elementary techniques.
\end{remark}

\begin{proof}
    Let
    \[\delta := \tfrac{1}{3} e^{-2\beta t -1}.\]
    We note that
    \[1 \leq j \leq \delta \ell,\]
    thus $\lceil \delta \ell \rceil \leq 2 \delta \ell.$ Then, letting $b = \lceil \delta \ell \rceil,$
    by Proposition~\ref{prop.polynomial-bound} we have that
     \[I^\ell_j(t) \leq \Big(\frac{j+b}{j+\ell}\Big)^b e^{\beta b t} \leq (3\delta)^{\lceil \delta  \ell \rceil} e^{2\beta \delta \ell t} \leq \exp(\delta \ell (2\beta  t + \log(3\delta))) = e^{-\delta \ell},\]
    where we use that by definition
    \[2\beta t + \log(3\delta) = -1.\]
    Plugging the definition of $\delta$ into the bound, we conclude.
\end{proof}

Now that we have some control on the $I^\ell_j$, we are ready to bound the hierarchies of differential inequalities. The following is the first step to getting the correct bound, given by inductively applying Gr\"onwall's inequality and using Proposition~\ref{prop.phi-int}. The first bound~\eqref{eq.ode-hierarchy-no-t0} is sufficient to give Theorem~\ref{thm.main-result-intro} for short times, but the second bound~\eqref{eq.ode-hierarchy-t0} is necessary to get the result for all times.

\begin{proposition}
    Suppose $x_k \geq 0$ satisfy the hierarchy of differential inequalities
    \[\begin{cases}
    \dot x_k \leq \beta k ( \alpha_k x_{k+1} - x_k) + r_k\\
    x_k(0) =0,
    \end{cases}
    \]
    for some $r_k,\alpha_k \geq 0$ constants. Then we have the bounds for any $j,\ell \in \N$, for any $t \geq 0$
    \begin{equation}
    \label{eq.ode-hierarchy-no-t0}
    x_j(t) \leq A_j^\ell I^\ell_{j}(t)\sup_{s \in [0, t]} x_{j+\ell}(s) + \frac{1}{\beta}\sum_{k=0}^{\ell-1} A_j^{k+1}I_j^{k+1}(t) \frac{r_{j+k}}{\alpha_{j+k}(j+k)}, \end{equation}
    and for any $t_0 \geq 0, t \geq t_0,$
    \begin{align}
    x_j(t) &\leq A_j^\ell I^\ell_j(t-t_0) \sup_{s \in [t_0,t]} x_{j+\ell}(s) +  \sum_{k=1}^\ell A_j^\ell I^k_{j+\ell -k}(t_0) I^{\ell-k}_j(t-t_0) \sup_{s \in [0, t_0]} x_{j+\ell}(s) 
    \notag\\&\qquad+ \frac{1}{\beta}\sum_{k=0}^{\ell-1} A_j^{k+1}I_j^{k+1}(t) \frac{r_{j+k}}{\alpha_{j+k}(j+k)},
    \label{eq.ode-hierarchy-t0}
    \end{align}
    where
    \[A^k_j := \prod_{i=j}^{j+k-1} \alpha_i,\]
    and we take $A_j^0 = 1$ by convention.
\end{proposition}

\begin{proof}
    We note that by Gr\"onwall's inequality,
    \[x_j \leq \beta \alpha_j j e^{-\beta j t} \int_0^t e^{\beta j s} \Big(x_{j+1}(s) + \frac{r_j}{\beta \alpha_j j}\Big)\,ds.\]
    Note that
    \begin{equation}
    \label{eq.A-jk-induct}
    \alpha_j A_{j+1}^k = A_j^{k+1}
    \end{equation}
    We first prove~\eqref{eq.ode-hierarchy-no-t0}. We prove this bound inductively in $\ell,$ for all $j$. For $\ell=0$, the bound is direct from $I^0_j(t) = A^0_j = 1$. Then inductively, we use Gr\"onwall's inequality together with the inductive hypothesis to give that $x_j$ is bounded by
    \begin{align*}
        &\beta \alpha_j j e^{-\beta j t} \int_0^t e^{\beta j s}\Big(A_{j+1}^{\ell-1} I^{\ell-1}_{j+1}(s) \sup_{r \in [0,t]} x_{j+\ell}(r) + \frac{1}{\beta} \sum_{k=0}^{\ell-2} A_{j+1}^{k+1} I_{j+1}^{k+1}(s) \frac{r_{j+1+k}}{\alpha_{j+1+k}(j+1+k)} + \frac{r_j}{\beta \alpha_j j}\Big)\,ds
        \\&\qquad= \beta j e^{-\beta j t} \int_0^t e^{\beta j s}\Big(A_{j}^{\ell} I^{\ell-1}_{j+1}(s) \sup_{r \in [0,t]} x_{j+\ell}(r) + \frac{1}{\beta} \sum_{k=0}^{\ell-1} A_{j}^{k+1} I_{j+1}^{k}(s) \frac{r_{j+k}}{\alpha_{j+k}(j+k)}\Big)\,ds, 
    \end{align*}
    where we use~\eqref{eq.A-jk-induct} on the second line. The using Proposition~\ref{prop.phi-int}, we get~\eqref{eq.ode-hierarchy-no-t0}.

    We now turn our attention to~\eqref{eq.ode-hierarchy-t0}. Again we prove it inductively in $\ell$, for all $j$. For $\ell=0$, it is again direct from $I^0_j(t) = A^0_j = 1$. Then inductively, we use Gr\"onwall's inequality then~\eqref{eq.ode-hierarchy-no-t0} to control the integral on $[0,t_0]$ and the inductive hypothesis to control the integral on $[t_0,t]$. This  gives
    \begin{align}
        \label{eq.ode-hierarchy-t1}
        x_j &\leq \beta \alpha_j j e^{-\beta j t} \int_0^{t_0} e^{\beta j s} A_{j+1}^{\ell-1} I^{\ell-1}_{j+1}(s) \sup_{r \in [0,t_0]} x_{j+\ell}(r)\,ds
        \\&+\beta \alpha_j j e^{-\beta j t} \int_{t_0}^t e^{\beta js} \bigg(\sum_{k=1}^{\ell-1} A_{j+1}^{\ell-1} I^k_{j+\ell -k}(t_0) I^{\ell-1-k}_{j+1}(s-t_0) \sup_{r \in [0, t_0]} x_{j+\ell}(r) \notag\\&\qquad\qquad\qquad\qquad\qquad\qquad\qquad\qquad\qquad+A_{j+1}^{\ell-1} I^{\ell-1}_{j+1}(s-t_0) \sup_{r \in [t_0,t]} x_{j+\ell}(r)\bigg)\,ds
        \label{eq.ode-hierarchy-t2}
        \\& + \beta \alpha_j j e^{-\beta j t} \int_{0}^t e^{\beta js}\Big( \frac{1}{\beta}\sum_{k=0}^{\ell-2} A_{j+1}^{k+1}I_{j+1}^{k+1}(s) \frac{r_{j+1+k}}{\alpha_{j+1+k}(j+1+k)}+ \frac{r_j}{\beta \alpha_j j}\Big)\,ds.
        \label{eq.ode-hierarchy-t3}
    \end{align}
    We then note that top term~\eqref{eq.ode-hierarchy-t1} is equal to 
    \[e^{-\beta j (t-t_0)} A^\ell_j I^{\ell}_j(t_0) \sup_{s \in [0,t_0]} x_{j+\ell}(s) \leq  A^\ell_j I^{\ell}_j(t_0)I^0_j(t-t_0) \sup_{s \in [0,t_0]} x_{j+\ell}(s),\]
    where we use~\eqref{eq.A-jk-induct}, Proposition~\ref{prop.phi-int}, and the brutal bound
    $e^{-\beta j(t-t_0)} \leq 1$. Then the middle term~\eqref{eq.ode-hierarchy-t2} is equal to 
    \begin{align*}
        &\beta j e^{-\beta j (t-t_0)} \int_{0}^{t-t_0} e^{\beta j(s-t_0)} \Big(\sum_{k=1}^{\ell-1} A_{j}^{\ell} I^k_{j+\ell -k}(t_0) I^{\ell-1-k}_{j+1}(s) \sup_{r \in [0, t_0]} x_{j+\ell}(r)+A_{j}^{\ell} I^{\ell-1}_{j+1}(s) \sup_{r \in [t_0,t]} x_{j+\ell}(r)\Big)\,ds\\
        &\qquad = \sum_{k=1}^{\ell-1} A_{j}^{\ell} I^k_{j+\ell -k}(t_0) I^{\ell-k}_{j}(t-t_0) \sup_{s \in [0, t_0]} x_{j+\ell}(s)  +A_{j}^{\ell} I^{\ell}_{j}(s) \sup_{s \in [t_0,t]} x_{j+\ell}(s),
    \end{align*}
    we we again use~\eqref{eq.A-jk-induct} and Proposition~\ref{prop.phi-int}. Lastly, we note that the third term~\eqref{eq.ode-hierarchy-t3} is equal to
    \[\frac{1}{\beta}\sum_{k=0}^{\ell-1} A_j^{k+1}I_j^{k+1}(t) \frac{r_{j+k}}{\alpha_{j+k}(j+k)},\]
    where the computation follows exactly as in the proof of~\eqref{eq.ode-hierarchy-no-t0}. Combining these three equalities we get~\eqref{eq.ode-hierarchy-t0}.
\end{proof}

\begin{note}
    We remark that we take a very rough bound in the above argument, taking $e^{-\beta j(t-t_0)} \leq 1$. In other applications, one may wish to avoid taking this bound, but in this application, we will be interested in $t-t_0$ very small and $\sup_{s \in [0,t_0]} x_{j+\ell}$ already $O(1)$, as such we won't need the extra decay this exponential provides. Thus for simplicity, we discard it and get the above proposition.
\end{note}

We now can apply the exponential decay bound given by Proposition~\ref{prop.exp-bound} to~\eqref{eq.ode-hierarchy-t0} to give the following.

\begin{proposition}
    \label{prop.ode-hierarchy-t0-iter}
    Suppose $x_k \geq 0$ satisfy the hierarchy of differential inequalities
     \[\begin{cases}
    \dot x_k \leq \beta k ( \alpha_k x_{k+1} - x_k) + r_k\\
    x_k(0) =0,
    \end{cases}
    \]
    for some $r_k, \alpha_k \geq 0$ constants. Then for any $0 \leq t_0 \leq t$ and $j,\ell \in \N$ such that
    \begin{equation}
    \label{eq.j-cond}
    j \leq e^{-2\beta t -6}\ell,
    \end{equation}
    we have the bound
    \begin{align*}
            x_j(t) &\leq A^\ell_j  \exp(- e^{-2\beta (t-t_0)-3} \ell) \sup_{s \in [t_0,t]} x_{j+\ell}(s) + A_j^\ell e^{2\beta t +7}   \exp(- e^{-2\beta t -7}\ell)\sup_{s \in [0,t_0]}x_{j+\ell}(s) \
            \\&\qquad\qquad+ \frac{1}{\beta}\sum_{k=0}^{\ell-1} A_j^{k+1}I_j^{k+1}(t) \frac{r_{j+k}}{\alpha_{j+k}(j+k)}.
    \end{align*}
\end{proposition}

\begin{proof}
    By~\eqref{eq.ode-hierarchy-t0}, it suffices to bound
    \[I^\ell_j(t-t_0) \leq \exp(-e^{-2\beta (t-t_0)-3} \ell)\]
    and
    \[ \sum_{k=1}^\ell  I^k_{j+\ell -k}(t_0) I^{\ell-k}_j(t-t_0) \leq e^{2\beta t+7} \exp(-e^{-2\beta t -7} \ell).\]
    The first bound is direct from Proposition~\ref{prop.exp-bound} and the condition~\eqref{eq.j-cond} on $j$. For the second, we let
     \[\delta :=  \tfrac{1}{12} e^{-2\beta t_0 -1}\]
     and note that by~\eqref{eq.j-cond},
      \[j \leq \tfrac{1}{3}\delta e^{-2\beta(t-t_0)-1}\ell.\]
    Then we have that
    \begin{align*}
         \sum_{k=1}^\ell I^k_{j+\ell -k}(t_0) I^{\ell-k}_j(t-t_0) &=  \sum_{k=1}^{\lfloor(1-\delta) \ell\rfloor}  I^k_{j+\ell -k}(t_0) I^{\ell-k}_j(t-t_0) +  \sum_{k=\lfloor (1-\delta)\ell \rfloor+1}^\ell I^k_{j+\ell -k}(t_0) I^{\ell-k}_j(t-t_0).
    \end{align*}
    Then, since for $k \in \{1,\dotsc,\lfloor(1-\delta)\ell\rfloor\},$ $\ell -k \geq \delta \ell$ and 
    \[j \leq \tfrac{1}{3} e^{-2\beta (t-t_0)-1} \delta \ell = \tfrac{1}{36} e^{-2\beta t -2} \ell,\] we have from Proposition~\ref{prop.exp-bound}, using that $I^k_{j+\ell-k}(t_0)\leq 1,$
     \begin{align}
     \sum_{k=1}^{\lfloor(1-\delta)\ell\rfloor}  I^k_{j+\ell -k}(t_0) I^{\ell-k}_j(t-t_0) & \leq  \sum_{k=1}^{\lfloor(1-\delta)\ell\rfloor}  I^{\ell-k}_j(t-t_0) 
     \notag\\&\leq \ell \exp(-\tfrac{1}{3} e^{-2\beta (t-t_0)-1} (\ell-k)) 
     \notag\\&\leq \ell \exp(-\tfrac{1}{3}\delta e^{-2\beta (t-t_0)-1} \ell) = \ell \exp(-\tfrac{1}{36} e^{-2\beta t -2} \ell).
     \label{eq.l2-iter-ode-t1}
     \end{align}
    Then, for $k \in \{\lfloor (1-\delta)\ell \rfloor+1, \dotsc, \ell\},$
     \[j + \ell - k \leq j + \delta \ell \leq \tfrac{1}{6} e^{-2\beta t_0 -1}\ell \leq \tfrac{1}{3} e^{-2\beta t_0 -1}k ,\]
    using the definition of $\delta$ and that~\eqref{eq.j-cond} implies
    \[j \leq \tfrac{1}{12} e^{-2\beta t_0 -1}\ell.\]
    Thus Proposition~\ref{prop.exp-bound} gives that
    \begin{align}  
    \sum_{k=\lceil (1-\delta) \ell \rceil}^{\ell}  I^k_{j+\ell -k}(t_0) I^{\ell-k}_j(t-t_0)  &\leq \sum_{k=\lceil (1-\delta) \ell \rceil}^{\ell} I^k_{j+\ell -k}(t_0) 
    \notag\\&\leq \ell \exp(-\tfrac{1}{3} e^{-2\beta t_0 -1 } (\lfloor (1-\delta)\ell \rfloor+1)) \leq  \ell \exp(-\tfrac{1}{6} e^{-2\beta t_0 -1 } \ell).
    \label{eq.l2-iter-ode-t2}  
    \end{align}
    Combining~\eqref{eq.l2-iter-ode-t1} and~\eqref{eq.l2-iter-ode-t2}, we get
    \begin{align*}
    \sum_{k=1}^\ell I^k_{j+\ell -k}(t_0) I^{\ell-k}_j(t-t_0) &\leq 2 \ell \exp(-\tfrac{1}{36} e^{-2\beta t -2} \ell) 
    \\&\leq 128 e^{2\beta t+2} \exp(-\tfrac{1}{72} e^{-2\beta t -2} \ell) \leq e^{2\beta t +7}   \exp(- e^{-2\beta t -7}\ell),
    \end{align*}
    allowing us to conclude.
\end{proof}

\subsection{Proof of a special case of Theorem~\ref{thm.main-result-intro}}
\label{ss.proof-of-main-result}

With the bounds given by Proposition~\ref{prop.ode-hierarchy-t0-iter} in hand, we are now ready to prove Theorem~\ref{thm.main-result-intro} in the restricted setting of this section. The heart of the proof is captured in the following lemma, which we will iterate to get the full result.

\begin{lemma}
    \label{lem.l2-iter}
    Suppose $f +f^{-1}\in L^\infty(\T^d)$ and $K \in L^\infty_\delta(\T^{2d})$. Then there exists $C(\|K\|_{L^\infty_\delta},i)< \infty$ such that for any $t_0 \geq 0$, $L >0$ with $LN^{2/3} \leq N$ and for any $M \geq 1$ with
    \[\sup_{s \in [0,t_0]}\int \Big|\frac{\phi^i_{L N^{2/3}} - f_{L N^{2/3}}}{\rho^{\otimes {L N^{2/3}}}}\Big|^2 \rho^{\otimes {L N^{2/3}}}\,dx \leq M,\]
    then for $\delta>0$ defined to be 
    \[\delta e^{2\beta \delta}:= \frac{1}{48e^3 \|K\|_{L^\infty_\delta}^2}  \land 1,\]
    we have for all $j \in \N$ with
    \[j \leq Le^{-2\beta (t_0+1)-7} N^{2/3},\]
    and for all  $t_0 \leq t \leq t+ \delta,$ the bound
    \[ \int \Big|\frac{\phi^i_j - f_j}{\rho^{\otimes j}}\Big|^2 \rho^{\otimes j}\,dx \leq Ce^{Ct_0 + L^3} \Big( \Big(\frac{j}{N}\Big)^{2(i+1)} + \frac{M}{N^{2(i+1)}} (L^2N)^{-8i-1}\Big).\]
\end{lemma}
\begin{proof}
    We let
    \[x_k :=  \int \Big|\frac{\phi^i_k - f_k}{\rho^{\otimes k}}\Big|^2 \rho^{\otimes k}\,dx,\]
    so by~\eqref{eq.l2-hierarchy}, we have that
    \[\dot x_k \leq 4 \|K\|_{L^\infty_\delta}^2k (x_{k+1} - x_k) + 4\|K\|_{L^\infty_\delta}^2 \frac{k^3}{N^2} x_k + r_k \leq \beta k (\alpha_k x_{k+1} - x_k) + r_k,\]
    where
    \[\alpha_k := 1 + \frac{k^2}{N^2},\quad A^\ell_k := \prod_{i=k}^{k+\ell-1} \alpha_i, \quad r_k := 2 \int \Big| \frac{r^i_k}{\rho^{\otimes k}}\Big|^2 \rho^{\otimes k}\,dx.\]
    Then we note that
    \[\log A^{\ell}_k = \sum_{i=k}^{k+\ell-1} \log\Big(1 + \frac{i^2}{N^2}\Big) \leq \frac{1}{N^2}\sum_{i=k}^{k+\ell -1} i^2 \leq \frac{1}{N^2} \int_{k+1}^{k+\ell} x^2\,dx \leq \frac{(k+\ell)^3}{N^2}.\]
    Thus, for $k+\ell \leq L N^{2/3},$
    \begin{equation}
    \label{eq.A-bound}
    A^\ell_k \leq \exp\Big(\frac{L^3 N^2}{N^2}\Big) = e^{L^3}.
    \end{equation}
    By Proposition~\ref{prop.r-ij-bound}, for all $t \in [0,t_0+\delta],$ we can bound
    \begin{equation}
    \label{eq.nu-bound}
    r_k(t) \leq Ce^{Ct}\Big(\frac{k}{N}\Big)^{2(i+1)} \leq C e^{C t_0}\Big(\frac{k}{N}\Big)^{2(i+1)}.
    \end{equation}
    Then, for any 
     \[j \leq Le^{-2\beta t_0-7} N^{2/3},\]
     letting
     \[\ell := LN^{2/3} -j \geq \tfrac{1}{2} LN^{2/3},\]
     we have that
     \[j \leq e^{-2\beta (t_0+1) - 6} \ell,\]
    so for any $t \in [0,\delta],$
    we have by Proposition~\ref{prop.ode-hierarchy-t0-iter},~\eqref{eq.A-bound}, and~\eqref{eq.nu-bound} that $x_j(t)$ is bounded by
    \begin{align}
            &e^{L^3} \exp(- e^{-2\beta (t-t_0)-3} \ell) \sup_{s \in [t_0,t]} x_{LN^{2/3}}(s) + e^{L^3} e^{2\beta (t_0+1) +7}   \exp(- e^{-2\beta (t_0+1) -7}\ell)\sup_{s \in [0,t_0]}x_{LN^{2/3}}(s) 
            \notag\\&\qquad\qquad+ Ce^{Ct_0}\frac{e^{L^3}}{\beta N^{2(i+1)}}\sum_{k=0}^{\ell-1} I_j^{k+1}(t) (j+k)^{2i +1}.
            \label{eq.l2-iter-lem-terms}
    \end{align}
    We note that by Proposition~\ref{prop.polynomial-bound},
    \[\sum_{k=0}^{\ell-1} I_j^{k+1}(t) (j+k)^{2i +1} \leq e^{\beta (2i+3) t}\sum_{k=0}^{\ell-1} (j+2i+3)^{2i+3} (j+k)^{-2} \leq Ce^{Ct_0} j^{2i+3} \int_{x=j}^\infty x^{-2} \leq Ce^{Ct_0} j^{2(i+1)}.\]
    Thus
    \begin{equation}
    \label{eq.l2-iter-lem-t1}
    Ce^{Ct_0}\frac{e^{L^3}}{\beta N^{2(i+1)}}\sum_{k=0}^{\ell-1} I_j^{k+1}(t) (j+k)^{2i +1} \leq Ce^{Ct_0 + L^3} \Big(\frac{j}{N}\Big)^{2(i+1)}.
    \end{equation}
    Then we note that
    \begin{align*}
    \exp(- e^{-2\beta (t_0+1) -7}\ell)\sup_{s \in [0,t_0]}x_{LN^{2/3}}(s) &\leq M \exp(-e^{-2\beta (t_0+1) -8} LN^{2/3}) 
    \\&\leq \frac{M((3(i+1)+12i+3/2) e^{2\beta (t_0+1) +8} L^{-1})^{3(i+1)+12i+3/2}}{(N^{2/3})^{3(i+1)+12i+3/2}}  
    \\&\leq \frac{Ce^{Ct_0}M}{N^{2(i+1)}} (L^2N)^{-8i-1}, 
    \end{align*}
    where we use that
    \[x^m e^{-ax} \leq \Big(\frac{m}{a}\Big)^me^{-m}.\]
    Thus
    \begin{equation}
        \label{eq.l2-iter-lem-t2}
         e^{L^3} e^{2\beta (t_0+1) +7}   \exp(- e^{-2\beta (t_0+1) -7}\ell)\sup_{s \in [0,t_0]}x_{LN^{2/3}}(s) \leq C e^{Ct_0 + L^3} \frac{M}{N^{2(i+1)}}(L^2N)^{-8i-1}.
    \end{equation}
    For the last term in~\eqref{eq.l2-iter-lem-terms}, we need to control $x_{LN^{2/3}}(t)$ for $t \in [t_0,t_0+\delta]$. Let
    \[y_{LN^{2/3}}(t) := \int \Big|\frac{f_{L N^{2/3}}}{\rho^{\otimes {L N^{2/3}}}}\Big|^2 \rho^{\otimes {L N^{2/3}}}\,dx.\]
    We note that by the triangle inequality
    \[y_{LN^{2/3}}(t_0) \leq 2x_{LN^{2/3}}(t_0) + 2\int \Big|\frac{\phi^i_{L N^{2/3}}}{\rho^{\otimes {L N^{2/3}}}}\Big|^2(t_0) \rho^{\otimes {L N^{2/3}}}(t_0)\,dx \leq Ce^{Ct_0}M,\]
    where we use that $M \geq 1$ and Proposition~\ref{prop.f-ij-bound} to bound the term involving $\phi^i_j$. Then we note that~\eqref{eq.l2-apriori} gives that 
    \[\dot y_{LN^{2/3}} \leq 12\|K\|_{L^\infty_\delta}^2 L N^{2/3} y_{LN^{2/3}},\]
    thus, for $t_0 \leq t \leq t_0 +\delta $
    \[y_{LN^{2/3}}(t) \leq e^{12\|K\|_{L^\infty_\delta}^2 L N^{2/3} (t -t_0)} y_{LN^{2/3}}(t_0) \leq Ce^{Ct_0}M  e^{12\|K\|_{L^\infty_\delta}^2 L N^{2/3} \delta}.\]
    So
    \[x_{LN^{2/3}}(t) \leq 2 y_{LN^{2/3}}(t) + 2\int \Big|\frac{\phi^i_{L N^{2/3}}}{\rho^{\otimes {L N^{2/3}}}}\Big|^2\rho^{\otimes {L N^{2/3}}}\,dx \leq Ce^{Ct_0}M  e^{12\|K\|_{L^\infty_\delta}^2 L N^{2/3} \delta}.\]
    Note that $\ell \geq j$ and $j +\ell = LN^{2/3}$, so
    \[\ell \geq \tfrac{1}{2} LN^{2/3}.\]    
    Thus for $t_0 \leq t \leq t_0+\delta,$
    \begin{align*}
        \exp(- e^{-2\beta (t-t_0)-3} \ell) \sup_{s \in [t_0,t]} x_{LN^{2/3}}(s) &\leq Ce^{Ct_0}M \exp((12\|K\|_{L^\infty_\delta}^2\delta   -  \tfrac{1}{2}e^{-2\beta \delta-3}) LN^{2/3}) 
        \\&\leq Ce^{Ct_0}M \exp(-\tfrac{1}{4} e^{-2\beta \delta -3} LN^{2/3}),
    \end{align*}
    where we use that by the definition of $\delta$,
    \[12\|K\|_{L^\infty_\delta}^2\delta \leq \tfrac{1}{4}e^{-2\beta \delta-3}.\]
    Then
    \[\exp(-\tfrac{1}{4} e^{-2\beta \delta -3} LN^{2/3}) \leq \exp(- e^{-2\beta-5} LN^{2/3}) \leq \frac{C}{N^{2(i+1)}} (L^2N)^{-8i-1}.\]
    Thus,
    \begin{equation}
        \label{eq.l2-iter-lem-t3}
         e^{L^3} \exp(- e^{-2\beta (t-t_0)-3} \ell) \sup_{s \in [t_0,t]} x_{LN^{2/3}}(s) \leq Ce^{Ct_0 + L^3}\frac{M}{N^{2(i+1)}}(L^2N)^{-8i-1}.
    \end{equation}
    Then combining~\eqref{eq.l2-iter-lem-terms},~\eqref{eq.l2-iter-lem-t1},~\eqref{eq.l2-iter-lem-t2}, and~\eqref{eq.l2-iter-lem-t3}, we see that for any $j \leq Le^{-2\beta t_0-7} N^{2/3}$ and any $t_0 \leq t \leq t_0 + \delta,$
    \[x_j \leq  Ce^{Ct_0 + L^3} \Big( \Big(\frac{j}{N}\Big)^{2(i+1)} + \frac{M}{N^{2(i+1)}} (L^2N)^{-8i}\Big),\]
    as desired.
\end{proof}

We now prove the special case of Theorem~\ref{thm.main-result-intro} by iterating Lemma~\ref{lem.l2-iter}. The main difficulty is controlling the constants that appear in the iteration.

\begin{proposition}
    \label{prop.main-result-special}
    Suppose $f \in L^1(\T^d), K \in L^\infty_\delta(\T^{2d}).$ Suppose additionally $f + f^{-1} \in L^\infty(\T^d)$. Then for each $i \in \N,$ there exists $C(\|K\|_{L^\infty_\delta(\Omega^{2})},i)< \infty$ such that for any $N$ and any $j$ with
    \[j \leq C^{-1} e^{-Ct^2} N^{2/3},\]
    we have the bound
    \begin{equation*}
    \int \bigg|\frac{f_{j,N}- \sum_{k=0}^i N^{-k}f^k_j }{\rho^{\otimes j}}\bigg|^2 \rho^{\otimes j}\,dx \leq  Ce^{Ct} \Big(\frac{j}{N}\Big)^{2(i+1)}.
    \end{equation*}
\end{proposition}

\begin{proof}[Proof of Theorem~\ref{thm.main-result-intro} in the restricted setting]
    Fix $\delta$ as in Lemma~\ref{lem.l2-iter}, so that 
    \[\delta e^{2\beta \delta}:= \frac{1}{48e^3 \|K\|_{L^\infty_\delta}^2}  \land 1.\]
    Let $L_0 = 1$ and let
    \[L_{k+1} := \lfloor L_ke^{-2\beta \delta k -7 - 2\beta}N^{2/3}\rfloor N^{-2/3}.\]
    We then claim inductively that for 
    \[j \leq L_k N^{2/3};\quad 0 \lor (k-1)\delta \leq t \leq k\delta,\]
    we have the bound
    \[ \int \Big|\frac{\phi^i_j - f_j}{\rho^{\otimes j}}\Big|^2 \rho^{\otimes j}\,dx \leq  B_k \Big(\frac{j}{N}\Big)^{2(i+1)},\]
    where $B_0 = 1$ and 
    \[B_{k+1} =Ce^{C\delta k + L_k^3} (1+ B_k (L_k^2N)^{-8i-1}).\]
    We note the bound is trivially true for $k=0$. Then inductively, using Lemma~\ref{lem.l2-iter} with $t_0 = \delta k,$ $L = L_k$, and
    \[M= B_k,\]
    we get that for
    \[j \leq L_{k+1}N^{2/3} \leq L_k e^{-2\beta (k\delta+1) -7}N^{2/3},\]
    for $\delta k \leq t \leq \delta(k+1),$
    \[ \int \Big|\frac{\phi^i_j - f_j}{\rho^{\otimes j}}\Big|^2 \rho^{\otimes j}\,dx \leq Ce^{C\delta k + L_k^3} (1+ B_k (L_k^2N)^{-8i-1})\Big(\frac{j}{N}\Big)^{2(i+1)}\leq  B_{k+1} \Big(\frac{j}{N}\Big)^{2(i+1)},\]
    Thus the induction closes.
    
    Now we just need to control $B_k, L_k$. First note that
    \[L_{k} \leq L_{k-1}e^{-2\beta \delta (k-1)-7-2\beta} \leq L_{k-1} \leq \cdots \leq L_0 = 1,\]
    and also that
    \begin{align*}
    L_{k} &\geq L_{k-1}e^{-2\beta \delta (k-1)-7-2\beta} - N^{-2/3} 
    \\&\geq \cdots \geq \exp\Big({-}(7+2\beta)k -2\beta \delta \sum_{i=0}^{k-1} i\Big)L_0 - N^{-2/3}\sum_{\ell=0}^{k-1} \exp({-7} \ell) 
    \\&\geq \exp\Big({-}(7+2\beta)k -2\beta \delta \sum_{i=0}^{k-1} i\Big) -2N^{-2/3}.
    \end{align*}
    Recalling $\sum_{i=0}^{k-1} i = \tfrac{1}{2} (k^2 -k)$, we have
    \[L_k \geq \exp({-}(7+2\beta)k-\beta \delta (k^2 -k)) - 2 N^{-2/3} \geq \exp(-10 (1+\beta) k^2) - 2 N^{-2/3}.\]
    Thus, for $k \leq C^{-1}\sqrt{\log(N)} - C,$
    we have,
    \[2N^{-2/3} \leq \tfrac{1}{2}\exp(-10(1+\beta) k^2) \]
    so that
    \[L_k \geq \frac{1}{2} \exp(-10(1+\beta) k^2)\geq \exp(-11(1+\beta)k^2).\]
    Thus for $k \leq C^{-1}\sqrt{\log(N)} -C,$
    \[\frac{1}{L_k^2N} \leq \frac{1}{N} \exp(22(1+\beta)k^2) \leq N^{-1/2},\]
    which implies that
    \[
    B_{k+1} \leq Ce^{C k}(1+ B_k (L_k^2N)^{-8i-1})\leq  Ce^{C k}+\frac{Ce^{Ck}}{N^{1/2}}B_k \leq C(e^{Ck} + B_k) .\]
    Iterating this bound then gives
    \[B_k \leq Ce^{C k}.\]
    Therefore, for all $k \leq C^{-1} \sqrt{\log(N)} -C$, if 
    \[j \leq \exp(-11(1+\beta)k^2)L \leq L_k,\]
    and $0\leq t \leq k \delta$,
     \[ \int \Big|\frac{\phi^i_j - f_j}{\rho^{\otimes j}}\Big|^2 \rho^{\otimes j}\,dx \leq  Ce^{C k} \Big(\frac{j}{N}\Big)^{2(i+1)}.\]
     Choosing the optimal $k$, we get that for any $t \leq  C^{-1} \sqrt{\log(N)}-C,$ if $j \leq C^{-1}\exp(-C t^2)N^{2/3}$, we get the bound
    \[ \int \Big|\frac{\phi^i_j - f_j}{\rho^{\otimes j}}\Big|^2 \rho^{\otimes j}\,dx \leq  Ce^{Ct} \Big(\frac{j}{N}\Big)^{2(i+1)}.\]
    This is almost precisely the result, except with an additional restriction on $t$. We note though that the $t$ bound is superfluous as by choosing $C$ large enough, if $t \geq  C_0\sqrt{\log(N)}-C_0,$ then
    \[C^{-1}\exp(-C t^2)N^{2/3} <1,\]
    so the result holds vacuously in this case. Thus we can remove the $t$ restriction and conclude.
\end{proof}
\begin{note}
    We note by choosing the starting point of the induction $L_0$ to be larger, one can slightly expand the range of $j$ for which one can prove the bound. This adds complication without being of any particular interest, so we omit this argument.
\end{note}

\subsection{Proofs of $g^i_j$ marginalization}\label{ss.marginalization}

We wish to show that the $g^i_j$ have the same marginalization properties as the $g_j$. This is shown by inductively using the Gr\"onwall inequality, where the induction is done in the ordering on $D$.

\begin{proof}[Proof of Lemma~\ref{lem.marginalization}]
We will show this inductively using the order given in Definition \ref{def.order}.
The base case holds trivially since $g_1^0=\rho$ is a probability density.

Fixing $(i,j)$ such that $(i,j)\geq (0,1)$, suppose now that the marginalization holds for all $g^k_\ell$ with $(k,\ell)<(i,j).$
We define
\[\psi(x_1,...,x_{j-1}):= \int g^i_j(x)\, dx_j.
\]
Then integrating the equation \eqref{eq.g-ij-PDE} over $x_j$ we get

\begin{align}
\notag\lefteqn{\partial_t \psi - \Delta \psi+\sum_{k=1}^{j-1} H_k g^0_{\{k\}}  \psi(x^{[j-1]-\{k\}\cup\{*\}})+\sum_{k=1}^{j-1} H_k \psi(x)  g^{0}_{\{*\}}}\quad&
\\&\label{line.1}=-\sum_{k=1}^{j-1} H_k \int g^i_{[j]\cup\{*\}} \,dx_j-\sum_{k=1}^{j-1} \sum_{W \subseteq [j] -\{k\}} \sum_{m=1}^{i-1}H_k \int g^m_{W \cup \{k\}}  g^{i-m}_{[j]\cup \{*\} - W -\{k\}} \,dx_j
\\&\label{line.2}\quad +j \sum_{k=1}^{j-1} H_k \int g^{i-1}_{[j]\cup\{*\}}\,dx_j+\sum_{k=1}^{j-1} \sum_{W \subseteq [j] - \{k\}} (j-1 - |W|) \sum_{m=0}^{i-1} H_k \int g^m_{W \cup \{k,*\}} g^{i-1-m}_{[j] - \{k\} - W}\,dx_j
 \\&\label{line.3}\quad +j\sum_{k=1}^{j-1} \sum_{W \subseteq [j] -\{k\}} \sum_{m=0}^{i-1}H_k \int g^m_{W \cup \{k\}} g^{i-1-m}_{[j]\cup \{*\} - W -\{k\}}\,dx_j
        \\&\notag\quad  +\sum_{k=1}^{j-1} \sum_{\substack{W \subseteq [j] - \{k\}\\R \subseteq [j] -\{k\} - W}} (j -1 - |W| - |R|) \sum_{m=0}^{i-1} \sum_{n=0}^{i-1-m} H_k \int g^m_{W \cup \{k\}}  g^n_{R \cup \{*\}} g^{i-1-m-n}_{[j] - R - W - \{k\}}\,dx_j
        \\&\label{line.5}\quad - \sum_{k=1}^{j-1}\sum_{\ell=1}^j \int S_{k,\ell} g^{i-1}_j \,dx_j- \sum_{k=1}^{j-1}\sum_{\substack{\ell=1\\k\neq \ell}}^j \sum_{W \subseteq [j] - \{k,\ell\}} \sum_{m=0}^{i-1} \int S_{k,\ell} g^m_{W \cup \{k\}} g^{i-1-m}_{[j] - \{k\} - W}\,dx_j.
\end{align}
where we've used that
\[\int \nabla_{x_j} \cdot h\,dx_j=0\]
for any function $h$.
Both sums on line \eqref{line.1} are equal to zero by the induction hypothesis as all the superscripts are larger than 1. The induction assumption also implies that the first sum on line \eqref{line.2} equals 0 when $i\geq 2$ as then the superscript $i-1\geq 1$. When $i=1$, it also equals 0, but instead because $g^{0}_{[j]\cup\{*\}}=0$ as $|[j]\cup\{*\}|\geq 2$. 

The second sum on line \eqref{line.2} will be shown later to cancel with the first sum on line \eqref{line.5}, so we skip it for now.

For line \eqref{line.3}, we note that all terms in the sum corresponding to $0<m<i-1$ equal zero by the induction hypothesis. We are thus left with
\begin{align*}
j\sum_{k=1}^{j-1} \sum_{W \subseteq [j] -\{k\}} H_k \int g^0_{W \cup \{k\}} g^{i-1}_{[j]\cup \{*\} - W -\{k\}}\,dx_j+H_k \int g^{i-1}_{W \cup \{k\}} g^{0}_{[j]\cup \{*\} - W -\{k\}}\,dx_j.
\end{align*}
The terms in this sum can be broken into two cases, either $j\in W$ or $j\notin W$. If $j\in W$ then $|W\cup\{k\}|\geq 2$, thus $g^0_{W\cup\{k\}}=0$ and
\[\int g^{i-1}_{W\cup \{k\}} \,dx_j=0,\]
thus all these terms equal to 0. When $j\notin W$, then $|[j]\cup\{*\}-W-\{k\}|\geq 2$, thus by an analogous argument all the corresponding terms equal zero as well. This shows that line \eqref{line.3} equals zero as well.

We have so far simplified the entire equation to
\begin{align}
\notag\lefteqn{\partial_t \psi - \Delta \psi+\sum_{k=1}^{j-1} H_k g^0_{\{k\}}  \psi(x^{[j-1]-\{k\}\cup\{*\}})+\sum_{k=1}^{j-1} H_k \psi(x)  g^{0}_{\{*\}}}\quad&
\\&\label{line.21}=\sum_{k=1}^{j-1} \sum_{W \subseteq [j] - \{k\}} (j-1 - |W|) \sum_{m=0}^{i-1} H_k \int g^m_{W \cup \{k,*\}} g^{i-1-m}_{[j] - \{k\} - W}\,dx_j
        \\&\label{line.22}\quad  +\sum_{k=1}^{j-1} \sum_{\substack{W \subseteq [j] - \{k\}\\R \subseteq [j] -\{k\} - W}}  (j -1 - |W| - |R|) \sum_{m=0}^{i-1} \sum_{n=0}^{i-1-m} H_k \int g^m_{W \cup \{k\}}  g^n_{R \cup \{*\}} g^{i-1-m-n}_{[j] - R - W - \{k\}}\,dx_j
        \\&\label{line.23}\quad - \sum_{k=1}^{j-1}\sum_{\ell=1}^j \int S_{k,\ell} g^{i-1}_j \,dx_j
        \\&\label{line.24}\quad - \sum_{k=1}^{j-1}\sum_{\substack{\ell=1\\k\neq \ell}}^j \sum_{W \subseteq [j] - \{k,\ell\}} \sum_{m=0}^{i-1} \int S_{k,\ell} g^m_{W \cup \{k\}} g^{i-1-m}_{[j] - \{k\} - W}\,dx_j.
\end{align}
First we claim that the sum \eqref{line.21} can be reduced to
\[
\sum_{k=1}^{j-1}   H_k g^{i-1}_{[j-1] \cup \{*\}}.
\]
This is clear when $j=1.$ Using the induction hypothesis when $0<m<i-1$ we reduce \eqref{line.21} to
\[
\sum_{k=1}^{j-1} \sum_{W \subseteq [j] - \{k\}} (j-1 - |W|)\Bigg( H_k \int g^0_{W \cup \{k,*\}} g^{i-1}_{[j] - \{k\} - W}\,dx_j+H_k \int g^{i-1}_{W \cup \{k,*\}} g^0_{[j] - \{k\} - W}\,dx_j\Bigg)
\]
Since $|W\cup\{k,*\}|\geq 2$, $g^{i-1}_{W \cup \{k,*\}}=0$ hence this further reduces to 
\[
\sum_{k=1}^{j-1} \sum_{W \subseteq [j] - \{k\}} (j-1 - |W|)H_k \int g^{i-1}_{W \cup \{k,*\}} g^0_{[j] - \{k\} - W}\,dx_j.
\]
The integral
\[
\int g^{i-1}_{W \cup \{k,*\}} g^0_{[j] - \{k\} - W}\,dx_j=0
\]
unless $W=[j-1]-\{k\}$ hence 
\begin{align*}
    \sum_{k=1}^{j-1} \sum_{W \subseteq [j] - \{k\}} (j-1 - |W|)H_k \int g^{i-1}_{W \cup \{k,*\}} g^0_{[j] - \{k\} - W}\,dx_j
&=\sum_{k=1}^{j-1} H_k \int g^{i-1}_{[j-1]\cup\{*\}} g^0_{\{j\}}\,dx_j
\\&=\sum_{k=1}^{j-1}H_k g^{i-1}_{[j-1]\cup\{*\}},
\end{align*}
as claimed.
This cancels exactly with \eqref{line.23} 
since if $\ell\neq j$ then
\[\int S_{k,\ell} g^{i-1}_j \,dx_j=0\]
while when $\ell=j$ exchangeability implies
\[\int S_{k,j} g^{i-1}_j \,dx_j=H_k g^{i-1}_{[j-1]\cup\{*\}}.\]

Similarly, we can reduce \eqref{line.22} to
\[
\sum_{k=1}^{j-1} \sum_{W \subseteq [j-1] - \{k\}}  \sum_{m=0}^{i-1} H_k g^m_{W \cup \{k\}}  g^{i-1-m}_{ [j-1] -\{k\} - W \cup \{*\}}.\]
Indeed, if $j\in W$ or $j\in R$ then either
\[
\int g^m_{W \cup \{k\}} \,dx_j=0\ \text{or}\ \int g^n_{R \cup \{*\}} \,dx_j=0,
\]
respectively. When $j\notin R\cup  W$ the integral
\[
\int g^{i-1-m-n}_{[j] - R - W - \{k\}} \,dx_j=0
\]
unless $i-1-m-n=0$ and $R\cup W=[j-1]-\{k\}$. Thus
\begin{align}
&\sum_{k=1}^{j-1} \sum_{W \subseteq [j] - \{k\}} \sum_{R \subseteq [j] -\{k\} - W} (j -1 - |W| - |R|) \sum_{m=0}^{i-1} \sum_{n=0}^{i-1-m} H_k \int g^m_{W \cup \{k\}}  g^n_{R \cup \{*\}} g^{i-1-m-n}_{[j] - R - W - \{k\}}\,dx_j
\notag\\&\qquad=\sum_{k=1}^{j-1} \sum_{W \subseteq [j-1] - \{k\}} (j -1 - (j-2)) \sum_{m=0}^{i-1} H_k \int g^m_{W \cup \{k\}}  g^{i-1-m}_{[j-1]-W-\{k\} \cup \{*\}} g^{0}_{\{j\}}\,dx_j
\notag\\&\qquad=\sum_{k=1}^{j-1} \sum_{W \subseteq [j-1] - \{k\}} \sum_{m=0}^{i-1} H_k g^m_{W \cup \{k\}}  g^{i-1-m}_{[j-1]-W-\{k\} \cup \{*\}}.
\label{line.25}
\end{align}
This then will cancel with \eqref{line.24}. To see this not that when $\ell\neq j,$
\[ \int S_{k,\ell} g^m_{W \cup \{k\}} g^{i-1-m}_{[j] - \{k\} - W}\,dx_j=0\]
since if $j\in W$ then $|W \cup \{k\}|\geq 2$ and if $j\in [j] - \{k\} - W$ then $|[j] - \{k\} - W|\geq 2.$ The sum \eqref{line.24} thus reduces to
\[
\sum_{k=1}^{j-1} \sum_{W \subseteq [j] - \{k,\ell\}} \sum_{m=0}^{i-1} \int S_{k,j} g^m_{W \cup \{k\}} g^{i-1-m}_{[j] - \{k\} - W}\,dx_j
\]
which is then equal to~\eqref{line.25} by exchangeability, so the terms cancel exactly. 

We have thus shown that 
\[\partial_t \psi - \Delta \psi+\sum_{k=1}^{j-1} H_k g^0_{\{k\}}  \psi(x^{[j-1]-\{k\}\cup\{*\}})+\sum_{k=1}^{j-1} H_k \psi  g^{0}_{\{*\}}=0.\]
The claim is then completed by a Gr\"onwall argument on $\|\psi\|_{L^2}^2(t)$ using that $\psi(0,\cdot)=0.$
Using the marginalization of the $g^i_j$ to give the marginalization of the $f^i_j$ and the $\phi^i_j$ is direct from the definition~\eqref{eq.f-ij-def} of $f^i_j$ and then the definition~\eqref{eq.phi-ij-def} of $\phi^i_j$.
\end{proof}

\subsection{Proofs of bounds on the $g^i_j, f^i_j,$ and $R^i_j$}
\label{ss.proofs-of-bounds}

To bound the $R^i_j$ and $f^i_j$ we must first bound the $g^i_j$. The following proof follows similarly to Lemma~\ref{lem.marginalization}, where we inductively iterate up the hierarchy of equations satisfied by $g^i_j$ to find estimates.

\begin{proposition}\label{prop.g-ij-bounds} 
Suppose $f +f^{-1}\in L^\infty(\T^d)$ and $K \in L^\infty_\delta(\T^{2d})$. For all $i\geq 0,$ letting
\begin{equation}
    \tilde g_j^i:=\frac{g_j^i}{\rho^{\otimes j}},
\end{equation}
there exists a constant $C(\|K\|_{L^\infty_\delta},i)$ such that
\begin{equation}\label{eq.gij-bounds}
\int |\tilde g_j^i|^2 \rho^{\otimes j}\,dx \leq Ce^{Ct}.
\end{equation}
\end{proposition}

    \begin{proof}
    We will inductively show this bound holds for $(i,j)\in D$ under the order given in Definition~\ref{def.order}.

    The bound trivially holds in the base case $(i,j)=(0,1)$ since $\tilde g_1^0=1$, thus 
    \[ \int |\tilde g_1^0|^2 \rho^{\otimes j}\,dx=1.\]
    
    Assuming that for all $(k,\ell)<(i,j)$ the bound \eqref{eq.gij-bounds} holds, we group the terms on the right hand side of the equation for $g_j^i$ to write
    \[\partial_t g^i_j - \Delta g^i_j+\sum_{k=1}^j H_k g^0_{\{k\}}  g^{i}_{[j]\cup \{*\} - \{k\}}+\sum_{k=1}^j H_k g^i_{[j]}  g^{0}_{\{*\}}=\nabla\cdot F_j^i\]
    where 
    \[ F_j^i=\sum_{k=1}^8F_{j,k}^i\]
    with
    \[ F_{j,1}^i=-\sum_{k=1}^j e_k\otimes\int K(x_k,x_*) g^i_{[j]\cup\{*\}}\, dx_*\]
    \[ F_{j,2}^i=-\sum_{k=1}^j \sum_{W \subseteq [j] -\{k\}} \sum_{m=1}^{i-1}e_k\otimes \int K(x_k,x_*) g^m_{W \cup \{k\}}  g^{i-m}_{[j]\cup \{*\} - W-\{k\}}\,dx_*\]
    and the other $F_{j,k}^i$ are defined similarly in the order of the equation \eqref{eq.g-ij-PDE}. Taking a derivative we find
\begin{align*}
\frac{d}{dt}\int \big|\tilde g_j^i\big|^2\rho^{\otimes j}&=\int 2\partial_t g_j^i \tilde g_j^i-\partial_t \rho^{\otimes j}(\tilde g_j^i)^2\,dx
\\&=\int 2\Delta g_j^i\tilde g_j^i-\Delta\rho^{\otimes j}(\tilde g_j^i)^2\,dx+2\int \nabla \cdot F_j^i\tilde g_j^i\,dx
\\&\qquad-2\int \sum_{k=1}^j H_k g^0_{\{k\}}  g^{i}_{[j]\cup \{*\} - \{k\}}\tilde g_j^i\,dx-2\int\sum_{k=1}^j H_k g^i_{[j]}  g^{0}_{\{*\}} \tilde g_j^i\,dx
\\&\qquad+\int \nabla\cdot\bigg(\sum_{k=1}^j e_k\otimes H_k\rho^{\otimes (j+1)}\bigg)(\tilde g_j^i)^2\,dx
\\&=-2\int |\nabla \tilde g_j^i|^2\rho^{\otimes j}\,dx-2\int \frac{F_j^i}{\rho^{\otimes j}}\cdot\nabla\tilde g_j^i \rho^{\otimes j}\,dx
\\&\qquad+2\int \Bigg(\sum_{k=1}^j e_k\otimes \int K(x_k,x_*)\tilde g^{i}_{[j]\cup \{*\} - \{k\}}\rho(x_*) \,dx_*\Bigg)\cdot \nabla \tilde g_j^i \rho^{\otimes j}\,dx
\\&\qquad+2\int \tilde g^i_j\Bigg(\sum_{k=1}^j e_k\otimes \int K(x_k,x_*)  \rho(x_*)\, dx_*\Bigg)\cdot \nabla \tilde g_j^i \rho^{\otimes j}\,dx
\\&\qquad-2\int \tilde g_j^i\Bigg(\sum_{k=1}^j e_k\otimes \int K(x_k,x_*)\rho(x_*)\,dx_*\Bigg)\cdot \nabla\tilde g_j^i\rho^{\otimes j}\,dx
\\&\leq 2\int \bigg|\frac{F_j^i}{\rho^{\otimes j}}\bigg|^2\rho^{\otimes j}\,dx+2j\int \bigg|\int K(x_1,x_*)\rho(x_*) \tilde g^{i}_{[j]\cup \{*\} - \{1\}}\, dx_*\bigg|^2\rho^{\otimes j}\,dx,
\end{align*}
where the last line follows via Young's inequality. Using Jensen's inequality
\begin{align*}
  \int \bigg|\int K(x_1,x_*)\rho(x_*) \tilde g^{i}_{[j]\cup \{*\} - \{1\}}\, dx_*\bigg|^2\rho^{\otimes j}\,dx&\leq  \int |K(x_1,x_*)|^2  \big|\tilde g^{i}_{[j]\cup \{*\} - \{k\}}\big|^2\rho^{\otimes (j+1)}\,dxdx_*
   \\&\leq \|K\|_{L^\infty_\delta}\int |\tilde g_j^i|^2 \rho^{\otimes j}\,dx.
\end{align*}
This has the form of a Gr\"onwall term, thus all that remains is to bound the term involving $F_j^i$. First we use the triangle inequality to bound
\[\int \Big|\frac{F_j^i}{\rho^{\otimes j}}\Big|^2\rho^{\otimes j}\,dx\leq C \sum_{m=1}^8 \int \Big|\frac{F_{m,j}^i}{\rho^{\otimes j}}\Big|^2\rho^{\otimes j}\,dx.\]
The most intimidating term is $F_{6,j}^i$ which equals
\[\sum_{k=1}^j e_k\otimes\sum_{W \subseteq [j] - \{k\}} \sum_{R \subseteq [j] -\{k\} - W} (j -1 - |W| - |R|) \sum_{m=0}^{i-1} \sum_{n=0}^{i-1-m} \int K(x_k,x_*) g^m_{W \cup \{k\}}  g^n_{R \cup \{*\}} g^{i-1-m-n}_{[j] - R - W - \{k\}}\,dx_*.\]
 Using the triangle inequality over the sums and using exchangeability we find there exists a $j$ dependent constant such that
\begin{align*}
    &\int \Big|\frac{F_{k,j}^i}{\rho^{\otimes j}}\Big|^2\rho^{\otimes j}\,dx
    \\&\qquad\leq C\sum_{\substack{W \subseteq [j] - \{1\}\\R \subseteq [j] -\{1\} - W}}  \sum_{m=0}^{i-1} \sum_{n=0}^{i-1-m} \int \bigg|\int K(x_1,x_*) \tilde g^m_{W \cup \{1\}}  \tilde g^n_{R \cup \{*\}} \tilde g^{i-1-m-n}_{[j] - R - W - \{1\}}\rho(x_*)\,dx_*\Bigg|^2\rho^{\otimes j}\,dx.
\end{align*}
We note that $(m,|W|+1),(n,|R|+1),$ and $(i-1-m-n,j-|R|-|W|-1)$ are all less than $(i,j)$. We can thus bound
\begin{align*}
&\int \Big|\int K(x_1,x*) \tilde g^m_{W \cup \{1\}}  \tilde g^n_{R \cup \{*\}} \tilde g^{i-1-m-n}_{[j] - R - W - \{1\}}d\rho(x_*)\Big|^2\rho^{\otimes j}\,dx
\\&\qquad\leq \int \Big|\int K(x_1,x*)  \tilde g^n_{R \cup \{*\}}d\rho(x_*)\Big|^2|\tilde g^m_{W \cup \{1\}}|^2|\tilde g^{i-1-m-n}_{[j] - R - W - \{1\}}|^2 \rho^{\otimes j}\,dx
\\&\qquad\leq \|K\|_{L^\infty_\delta}^2\int |\tilde g^n_{R \cup \{*\}}|^2|\tilde g^m_{W \cup \{1\}}|^2|\tilde g^{i-1-m-n}_{[j] - R - W - \{1\}}|^2 \rho^{\otimes (j+1)}\,dxdx_*
\\&\qquad= \|K\|_{L^\infty_\delta}^2\int |\tilde g^n_{R \cup \{*\}}|^2\rho^{\otimes |R|+1}\,dx^{R\cup\{*\}}\int|\tilde g^m_{W \cup \{1\}}|^2\rho^{\otimes |W|+1}\,dx^{W\cup\{1\}}
\\&\qquad\qquad\times\int|\tilde g^{i-1-m-n}_{[j] - R - W - \{1\}}|^2 \rho^{\otimes j - |R| - |W| - 1} \,dx^{[j]-R-W-\{1\}}
\\&\qquad\leq Ce^{Ct} \Big(\sup_{(k,\ell) < (i,j)} \int |\tilde g^k_\ell|^2 \rho^{\otimes \ell}\ dx\Big)^3,
\end{align*}
where the second inequality follows by Jensen's inequality. Terms $F_{1,j}^i$ to $F_{5,j}^i$ are bounded similarly. The bounds on $F_{6,j}^i$ and $F_{8,j}^i$ are also straightforward and rely on bounding for $W\subset [j]-\{k,\ell\}$ integrals of the form
\begin{align*}
&\int\Big|K(x_k,x_\ell) \tilde g^m_{W \cup \{k\}} \tilde g^{i-1-m}_{[j] - \{1\} - W}\Big|^2 \rho^{\otimes j}\,dx
\\&\qquad\leq\|K\|_{L^\infty_\delta}^2 \int|\tilde g^m_{W \cup \{k\}}|^2\rho^{\otimes|W|+1}\,dx^{W \cup \{k\}}\int|\tilde g^{i-1-m}_{[j] - \{k\} - W}|^2\rho^{\otimes j - 1 - |W|}\, dx^{[j] - \{k\} - W}
\\&\qquad\leq Ce^{Ct}  \Big(\sup_{(k,\ell) < (i,j)} \int |\tilde g^k_\ell|^2 \rho^{\otimes \ell}\ dx\Big)^2.
\end{align*}
All together these bounds imply that
\[
\int \Big|\frac{F_j^i}{\rho^{\otimes j}}\Big|^2\rho^{\otimes j}\,dx\leq Ce^{Ct}  \Big(\sup_{(k,\ell) < (i,j)} \int |\tilde g^k_\ell|^2 \rho^{\otimes \ell}\ dx\Big)^3.
\]

Thus, in total we've found that
\[ \frac{d}{dt} \frac{1}{2}\int |\tilde g_j^i|^2 \rho^{\otimes j}\,dx\leq 3j\int |\tilde g_j^i|^2 \rho^{\otimes j}\,dx+Ce^{Ct} \Big(\sup_{(k,\ell) < (i,j)} \int |\tilde g^k_\ell|^2 \rho^{\otimes \ell}\ dx\Big)^3.\]
Applying Gr\"onwall's inequality and inducting allows us to conclude, noting that $j \leq i+1.$
\end{proof}

With the bounds on $g^i_j$ given by Proposition~\ref{prop.g-ij-bounds} in hand, we can now show the bounds on $f^i_j, \phi^i_j,$ and $R^i_j$ given in Proposition~\ref{prop.f-ij-bound} and Proposition~\ref{prop.r-ij-bound}. Before continuing, we prove a useful representation of the $f^i_j.$

\begin{lemma}
    \label{lem.f-ij-rho-rep}
    \[f^i_j = \sum_{\substack{P \subseteq [j]\\ |P| \leq 2i}}\sum_{\pi \vdash P} \sum_{\substack{(i_Q)_{Q \in \pi}\\ \sum i_Q = i\\ i_Q \geq 1}}  \rho^{\otimes (j- |P|)}(x^{[j] - P})\prod_{Q \in \pi} g^{i_Q}_Q.\]
\end{lemma}

\begin{proof}
    By the definition~\eqref{eq.f-ij-def},
\[f^i_j = \sum_{\sigma \vdash [j]} \sum_{\substack{(i_R)_{R \in \sigma}\\ \sum i_R = i}} \prod_{R \in \pi} g^{i_R}_R.\]
Since $g^k_\ell=0$ if $\ell>k+1$, the product
\[\prod_{R \in \sigma} g^{i_R}_R=0\]
unless $|R|\leq i_R+1$ for all $R\in\sigma$. Suppose that $\sigma$ corresponds to a nonzero product. Since $\sum_{R\in\pi}i_R=i$, we have that $i_R\neq 0$ for at most $i$ sets $R\in\sigma$. Thus it must be the case that
\[\sum_{\substack{R\in\sigma\\i_R\neq 0}}|R|\leq \sum_{\substack{R\in\pi\\i_R\neq 0}} i_R+1\leq 2i.\]
Letting $P=\bigcup_{i_R\neq 0}R$, then $|P|\leq 2i$, $\sigma=\pi\cup\{\{k\}:k\in[j]-Q\}$ where $\pi\vdash P$, $\sum_{Q\in\pi}i_Q=i$, $i_Q\geq 1$ for $Q\in\pi$ and $i_{\{k\}}=0$ for $k\notin P$.

Re-indexing the sum which defines $f_j^i$ and using that $g_1^0=\rho$ we thus get the above claimed representation of $f^i_j$.
\end{proof}

We now show the bounds on $f_j^i$. This will be a warm up for the more involved bounds on $R_j^i.$

\begin{proof}[Proof of Proposition~\ref{prop.f-ij-bound}] Using Lemma~\ref{lem.f-ij-rho-rep}, and the definition of $\tilde g_j^i$ given in Proposition~\ref{prop.g-ij-bounds}
\[\frac{f_j^i}{\rho^{\otimes j}}=\sum_{\substack{P \subseteq [j]\\ |P| \leq 2i}}\sum_{\pi \vdash P} \sum_{\substack{(i_Q)_{Q \in \pi}\\ \sum i_Q = i\\ i_Q \geq 1}}\prod_{Q \in \pi} \tilde g^{i_Q}_Q.\]
Thus expanding out the sums 
\begin{align*}
    \int \bigg|\frac{f_i^j}{\rho^{\otimes j}}\bigg|^2\rho^{\otimes j}\,dx=\sum_{\substack{P,R \subseteq [j]\\ |P|,|R| \leq 2i}}
    \sum_{\substack{\pi \vdash P\\\sigma \vdash R}} 
    \sum_{\substack{(i_Q)_{Q \in \pi}\\ \sum i_Q = i\\ i_Q \geq 1}}
    \sum_{\substack{(i_W)_{W \in \pi}\\ \sum i_W = i\\ i_W \geq 1}}\int \prod_{Q \in \pi}\tilde g^{i_Q}_Q\times\prod_{W \in \sigma} \tilde g^{i_W}_W \rho^{\otimes j}\,dx.
\end{align*}
Suppose that $P\neq R$, $\pi\vdash P$, $\sigma\vdash R$, and $i_Q,i_W\geq 1$ where $Q\in \pi$ and $W\in \sigma$. Then it must be the case that
\[\int\prod_{Q \in \pi}\tilde g^{i_Q}_Q\times\prod_{W \in \sigma} \tilde g^{i_W}_W \rho^{\otimes j}\,dx=0.\]
Indeed, if $x_k\in Q\in \pi$, but $x_k$ is not in $R$, then the marginalization given by Lemma~\ref{lem.marginalization} of $g_Q^{i_Q}$ implies this. We thus find that in fact
\begin{align*}
    \int \bigg|\frac{f_i^j}{\rho^{\otimes j}}\bigg|^2\rho^{\otimes j}\,dx=\sum_{\substack{P \subseteq [j]\\ |P| \leq 2i}}
    \sum_{\pi,\sigma \vdash P} 
    \sum_{\substack{(i_Q)_{Q \in \pi}\\ \sum i_Q = i\\ i_Q \geq 1}}
    \sum_{\substack{(i_W)_{W \in \pi}\\ \sum i_W = i\\ i_W \geq 1}}\int \prod_{Q \in \pi}\tilde g^{i_Q}_Q\times\prod_{W \in \sigma} \tilde g^{i_W}_W \rho^{\otimes j}\,dx.
\end{align*}
H\"older's inequality with Proposition~\ref{prop.g-ij-bounds} imply that 
\begin{align*}
    \bigg| \int \prod_{Q \in \pi}\tilde g^{i_Q}_Q\prod_{W \in \sigma} \tilde g^{i_W}_W \rho^{\otimes j}\,dx\bigg|^2&\leq \prod_{Q \in \pi}\int|\tilde g^{i_Q}_Q|^2\rho^{\otimes |Q|}\, dx^Q\times\prod_{W \in \pi}\int|\tilde g^{i_W}_W|^2\rho^{\otimes |W|}\, dx^W
    \\&\leq Ce^{C t},
\end{align*}
where this constant only depends on $i$ since there are at most $i$ terms in the products. On the other hand
\[\sum_{\substack{P \subseteq [j]\\ |P| \leq 2i}}
    \sum_{\pi,\sigma \vdash P} 
    \sum_{\substack{(i_Q)_{Q \in \pi}\\ \sum i_Q = i\\ i_Q \geq 1}}
    \sum_{\substack{(i_W)_{W \in \pi}\\ \sum i_W = i\\ i_W \geq 1}} 1\leq \sum_{\pi,\sigma \vdash P} C \leq C j^{2i}\]
    where the constant $C$ just depends on $i$. This completes the bound on the $f_j^i.$

The bound on $\phi_j^i$ is a direct consequence of this bound and the triangle inequality.
\end{proof}

\begin{proof}[Proof of Proposition~\ref{prop.r-ij-bound}] 
Throughout this proof, we somewhat abuse notation and denote
\[K*\rho(x) := \int K(x,y)\rho(y)\,dy.\]
First we note that
\[
\frac{R^i_j}{\rho^{\otimes j}}=\frac{1}{N^{i+1}} \sum_{k=1}^j e_k\otimes\sum_{\ell=1}^j \int K(x_k,x_*)\frac{f^i_{[j]\cup\{*\}}}{\rho^{\otimes j}}dx_*-K(x_k,x_\ell) \frac{f^i_j}{\rho^{\otimes j}}.
\]
Lemma~\ref{lem.f-ij-rho-rep} implies that
\[\frac{f^i_j}{\rho^{\otimes j}} = \sum_{m=1}^{2i} \sum_{\substack{P \subseteq [j]\\ |P| = m}}\sum_{\pi \vdash P} \sum_{\substack{(i_Q)_{Q \in \pi}\\ \sum i_Q = i\\ i_Q \geq 1}}  \prod_{Q \in \pi} \tilde g^{i_Q}_Q.\]
Thus
\[\frac{f^i_{[j]\cup\{*\}}}{\rho^{\otimes j}} = \frac{f^i_j}{\rho^{\otimes j}} \rho(x_*) +  \rho(x_*)\sum_{m=1}^{2i} \sum_{\substack{P \subseteq [j]\\ |P| = m-1}}\sum_{\pi \vdash P \cup \{*\}} \sum_{\substack{(i_Q)_{Q \in \pi}\\ \sum i_Q = i\\ i_Q \geq 1}} \prod_{Q \in \pi} \tilde g^{i_Q}_Q.\]
Using exchangeability we find
\begin{align*}
    \int \left|\frac{R^i_j}{\rho^{\otimes j}}\right|^2 \rho^{\otimes j}\,dx &\leq  \frac{j}{N^{2(i+1)}} \left|\sum_{\ell=1}^j \int K(x_1,x_*)\frac{f^i_{[j]\cup\{*\}}}{\rho^{\otimes j}}dx_*-K(x_1,x_\ell) \frac{f^i_j}{\rho^{\otimes j}}\right|^2 \rho^{\otimes j}\,dx
    \\&\leq \frac{2j}{N^{2(i+1)}} \int \left| \sum_{\ell=1}^j \left(K*\rho(x_1) -K(x_1,x_\ell) \right)\frac{f^i_j}{\rho^{\otimes j}}\right|^2 \rho^{\otimes j}\,dx
    \\&\qquad  + \frac{2j^3}{N^{2(i+1)}} \int \Bigg| \int K(x_1, x_*) \rho(x_*)\sum_{m=1}^{2i} \sum_{\substack{P \subseteq [j]\\ |P| = m-1}}\sum_{\pi \vdash P \cup \{*\}} \sum_{\substack{(i_Q)_{Q \in \pi}\\ \sum i_Q = i\\ i_Q \geq 1}} \prod_{Q \in \pi} \tilde g^{i_Q}_Q\,dx_*\Bigg|^2 \rho^{\otimes j}\,dx.
\end{align*}
We first consider the second term. Applying Jensen's inequality, we have
\begin{align*}
    &\int \Bigg| \int K(x_1, x_*) \rho(x_*)\sum_{m=1}^{2i} \sum_{\substack{P \subseteq [j]\\ |P| = m-1}}\sum_{\pi \vdash P \cup \{*\}} \sum_{\substack{(i_Q)_{Q \in \pi}\\ \sum i_Q = i\\ i_Q \geq 1}} \prod_{Q \in \pi} \tilde g^{i_Q}_Q\,dx_*\Bigg|^2 \rho^{\otimes j}\,dx
    \\&\quad \leq \int \Bigg| K(x_1, x_*) \sum_{m=1}^{2i} \sum_{\substack{P \subseteq [j]\\ |P| = m-1}}\sum_{\pi \vdash P \cup \{*\}} \sum_{\substack{(i_Q)_{Q \in \pi}\\ \sum i_Q = i\\ i_Q \geq 1}} \prod_{Q \in \pi} \tilde g^{i_Q}_Q \Bigg|^2 \rho^{\otimes (j+1)}\,dx dx_*
    \\&\quad \leq 2i\sum_{m=1}^{2i}\int \Bigg| K(x_1, x_*)  \sum_{\substack{P \subseteq [j]\\ |P| = m-1}}\sum_{\pi \vdash P \cup \{*\}} \sum_{\substack{(i_Q)_{Q \in \pi}\\ \sum i_Q = i\\ i_Q \geq 1}} \prod_{Q \in \pi} \tilde g^{i_Q}_Q \Bigg|^2 \rho^{\otimes (j+1)}\,dx dx_*.
\end{align*}
We now fix $m$ and analyze the term under the integral, expanding the square
\begin{align*}
    &\int \Bigg| K(x_1, x_*)  \sum_{\substack{P \subseteq [j]\\ |P| = m-1}}\sum_{\pi \vdash P \cup \{*\}} \sum_{\substack{(i_Q)_{Q \in \pi}\\ \sum i_Q = i\\ i_Q \geq 1}} \prod_{Q \in \pi} \tilde g^{i_Q}_Q \Bigg|^2 \rho^{\otimes (j+1)}\, dx dx_*
    \\&\quad \leq  \sum_{\substack{P \subseteq [j]\\ |P| = m-1}} \sum_{\substack{R \subseteq [j]\\ |R| = m-1}}\sum_{\pi \vdash P \cup \{*\}} \sum_{\sigma \vdash R \cup \{*\}}\sum_{\substack{(i_Q)_{Q \in \pi}\\ \sum i_Q = i\\ i_Q \geq 1}}  \sum_{\substack{(i_W)_{W \in \sigma}\\ \sum i_W = i\\ i_W \geq 1}}\int  |K(x_1, x_*)|^2   \prod_{Q \in \pi}\tilde g^{i_Q}_Q \prod_{W \in \sigma}  \tilde g^{i_W}_W  \rho^{\otimes (j+1)} \, dx dx_*.
\end{align*}
Note then that unless $P = R$,
\[\int  |K(x_1, x_*)|^2   \prod_{Q \in \pi}\tilde g^{i_Q}_Q \prod_{W \in \sigma} \tilde g^{i_W}_W  \rho^{\otimes (j+1)}\, dx dx_* =0.\]
To see this, suppose $P \ne R$. Since $|P| = |R|,$ then there exists $p\in P$ such that $p \not \in R$ and there exists $r \in R$ such that $r \not \in P$. We must have that $p \ne 1$ or $r \ne 1$. Let us suppose that $p \ne 1$, the other case follows symmetrically. Then let $S \in \pi$ such that $p \in S$. Then
\begin{align*}&\int  |K(x_1, x_*)|^2   \prod_{Q \in \pi} \tilde g^{i_Q}_Q \prod_{W \in \sigma} \tilde g^{i_W}_W  \rho^{\otimes j} \rho(x_*)\,dx dx_* 
\\&\quad= \int |K(x_1, x_*)|^2   \prod_{Q \in \pi - \{S\}} \tilde g^{i_Q}_Q \prod_{W \in \sigma}\tilde g^{i_W}_W \int \tilde g^{i_S}_S  \rho^{\otimes (j+1)}\, dx_p dx_1 \cdots dx_{p-1} dx_{p+1} \cdots dx_j dx_* =0,
\end{align*}
where we use that 
\[\int \tilde g^{i_S}_S \rho^{\otimes |S|}(x^S)\,dx_p = \int g^{i_S}_S\,dx_p = 0,\]
by Lemma~\ref{lem.marginalization}, since $i_S \geq 1$.

Using H\"older's inequality
\begin{align*}
    &\int  |K(x_1, x_*)|^2   \prod_{Q \in \pi}\tilde g^{i_Q}_Q\prod_{W \in \sigma}  \tilde g^{i_W}_W\rho^{\otimes (j+1)}\, dx dx_*
    \\&\qquad\leq \|K\|_{L^\infty_\delta}^2 \bigg(\int \prod_{Q \in \pi}|\tilde g^{i_Q}_Q|^2\rho^{\otimes (j+1)}\,dxdx_*\bigg)^{\frac{1}{2}}\bigg(\int\prod_{W \in \sigma}|\tilde g^{i_W}_W|^2\rho^{\otimes (j+1)}\,dxdx_*\bigg)^{\frac{1}{2}}
    \\&\qquad= \|K\|_{L^\infty_\delta}^2 \prod_{Q \in \pi}\bigg(\int|\tilde g^{i_Q}_Q|^2\rho^{\otimes|Q|}\,dx^{Q}\bigg)^{\frac{1}{2}}\prod_{W \in \sigma}\bigg(\int|\tilde g^{i_W}_W|^2\rho^{\otimes |W|}\,dx^{W}\bigg)^{\frac{1}{2}}
\end{align*}

Since $i_Q \leq i$ for all $Q\in\pi$, Proposition~\ref{eq.gij-bounds} implies that
\begin{equation}\label{eq.g-product-bound}
\prod_{Q \in \pi}\bigg(\int|\tilde g^{i_Q}_Q|^2\rho^{\otimes|Q|}\,dx^{Q}\bigg)^{\frac{1}{2}}\prod_{W \in \sigma}\bigg(\int|\tilde g^{i_W}_W|^2\rho^{\otimes |W|}\,dx^{W}\bigg)^{\frac{1}{2}}\leq (Ce^{C t})^{4i}\leq Ce^{C t}.
\end{equation}
Thus we always have that
\[\int  |K(x_1, x_*)|^2   \prod_{Q \in \pi} \tilde g^{i_Q}_Q \prod_{W \in \sigma} \tilde g^{i_W}_W  \rho^{\otimes j} \rho(x_*)\,dx dx_* \leq Ce^{C t}.\]
We also have that for any $P, R\subseteq [j]$ such that $|P| = |R| = m-1 \leq 2i,$
\[\sum_{\pi \vdash P \cup \{*\}} \sum_{\sigma \vdash R \cup \{*\}}\sum_{\substack{(i_Q)_{Q \in \pi}\\ \sum i_Q = i\\ i_Q \geq 1}}  \sum_{\substack{(i_W)_{W \in \sigma}\\ \sum i_W = i\\ i_W \geq 1}} 1 \leq C.\]
Thus 
\begin{align*}
    &\int \Bigg| K(x_1, x_*)  \sum_{\substack{P \subseteq [j]\\ |P| = m-1}}\sum_{\pi \vdash P \cup \{*\}} \sum_{\substack{(i_Q)_{Q \in \pi}\\ \sum i_Q = i\\ i_Q \geq 1}} \prod_{Q \in \pi} \tilde g^{i_Q}_Q \Bigg|^2 \rho^{\otimes j} \rho(x_*)\,dx dx_* 
\\&\quad \leq \sum_{\substack{P \subseteq [j]\\ |P| = m-1}} \sum_{\substack{R \subseteq [j]\\ |R| = m-1}} C e^{Ct} \delta_{P =R} = C  e^{Ct}\binom{j}{m-1} \leq C e^{Ct} j^{2i-1}.
\end{align*}
Putting it together, we so far have that 
\begin{equation}
\label{eq.remainder-bound-1}
    \int \left|\frac{R^i_j}{\rho^{\otimes j}}\right|^2 \rho^{\otimes j}\,dx \leq C e^{Ct} \Big(\frac{j}{N}\Big)^{2(i+1)} +  \frac{2j}{N^{2(i+1)}} \int \left| \sum_{\ell=1}^j \left(K*\rho(x_1) -K(x_1,x_\ell) \right)\frac{f^i_j}{\rho^{\otimes j}}\right|^2 \rho^{\otimes j}\,dx.
\end{equation}
All that remains therefore is to bound the second term above, which is somewhat more involved. Expanding $\frac{f^i_j}{\rho^{\otimes j}}$, pulling out one of the sums then expanding the square, we get
\begin{align}
    &\int \left| \sum_{\ell=1}^j \left(K*\rho(x_1) -K(x_1,x_\ell) \right)\frac{f^i_j}{\rho^{\otimes j}}\right|^2 \rho^{\otimes j}\,dx 
    \notag\\ &\quad \leq 2i\sum_{m=1}^{2i}\int \Bigg| \sum_{\ell=1}^j \left(K*\rho(x_1) -K(x_1,x_\ell) \right) \sum_{\substack{P \subseteq [j]\\ |P| = m}}\sum_{\pi \vdash P} \sum_{\substack{(i_Q)_{Q \in \pi}\\ \sum i_Q = i\\ i_Q \geq 1}}  \prod_{Q \in \pi} \tilde g^{i_Q}_Q\Bigg|^2 \rho^{\otimes j}\,dx 
    \notag\\&\quad\leq 2i\sum_{m=1}^{2i} \sum_{\ell=1}^j \sum_{k=1}^j \sum_{\substack{P \subseteq [j]\\ |P| = m}} \sum_{\substack{R \subseteq [j]\\ |R| = m}}\sum_{\pi \vdash P} \sum_{\sigma \vdash R}\sum_{\substack{(i_Q)_{Q \in \pi}\\ \sum i_Q = i\\ i_Q \geq 1}}   \sum_{\substack{(i_W)_{W \in \sigma}\\ \sum i_W = i\\ i_W \geq 1}} 1
    \notag\\&\qquad\quad \times\int (K*\rho(x_1) -K(x_1,x_\ell))\cdot (K*\rho(x_1) -K(x_1,x_k))  \prod_{Q \in \pi}\tilde g^{i_Q}_Q \prod_{W \in \sigma} \tilde g^{i_W}_W \rho^{\otimes j}\,dx.
    \label{eq.remainder-bound-2}
\end{align}
We then claim that
\begin{align}\notag&\int (K*\rho(x_1) -K(x_1,x_\ell))\cdot (K*\rho(x_1) -K(x_1,x_k))  \prod_{Q \in \pi}\tilde g^{i_Q}_Q \prod_{W \in \sigma} \tilde g^{i_W}_W \rho^{\otimes j}\,dx 
\\&\qquad\qquad\qquad\qquad\qquad\qquad\leq C e^{Ct} \delta_{\substack{\ell \in P\cup R \cup \{1,k\}\\k \in P \cup R \cup \{1,\ell\}}} \delta_{\substack{P \subseteq R \cup \{1,\ell,k\}\\R \subseteq P \cup \{1,\ell,k\}}}.
\label{eq.remainder-tensorization-cancellation}\end{align}
The bound by $C e^{Ct}$ follows by \eqref{eq.g-product-bound} as
\begin{align*}
&\Bigg|\int (K*\rho(x_1) -K(x_1,x_\ell))\cdot (K*\rho(x_1) -K(x_1,x_k))  \prod_{Q \in \pi}\tilde g^{i_Q}_Q \prod_{W \in \sigma} \tilde g^{i_W}_W \rho^{\otimes j}\,dx \Bigg|
\\&\qquad\leq 4\|K\|_{L^\infty_\delta}^2\prod_{Q \in \pi}\bigg(\int|\tilde g^{i_Q}_Q|^2\rho^{\otimes|Q|}\,dx^{Q}\bigg)^{\frac{1}{2}}\prod_{W \in \sigma}\bigg(\int|\tilde g^{i_W}_W|^2\rho^{\otimes |W|}\,dx^{W}\bigg)^{\frac{1}{2}}\leq C e^{Ct}.
\end{align*}
Thus we just need to show that if if any of the above four conditions fails to hold, the integral is $0$. The integral and conditions are symmetric in $\ell,k$ and also symmetric in $P,R$, so we just need to check the two conditions. If $\ell \not \in P \cup R \cup \{1,k\},$ then 
\begin{align*}
    &\int (K*\rho(x_1) -K(x_1,x_\ell))\cdot (K*\rho(x_1) -K(x_1,x_k))  \prod_{Q \in \pi}\tilde g^{i_Q}_Q \prod_{W \in \sigma} \tilde g^{i_W}_W \rho^{\otimes j}\,dx  
    \\&\quad= \int  (K*\rho(x_1) -K(x_1,x_k))  \prod_{Q \in \pi}\tilde g^{i_Q}_Q \prod_{W \in \sigma} \tilde g^{i_W}_W  \\&\qquad \quad\times \int (K*\rho(x_1) -K(x_1,x_\ell)) \rho^{\otimes j}\,dx_\ell dx_1\cdots dx_{\ell-1} dx_{\ell+1}\cdots dx_j= 0,
\end{align*}
where we use that
\[\int (K*\rho(x_1) -K(x_1,x_\ell))\rho(x_\ell)\,dx_\ell = K*\rho(x_1) - K*\rho(x_1) =0,\]
as $\ell \ne 1$.

Thus we see we get the term $\delta_{\ell \in P\cup R \cup \{1,k\}}$ in the bound and applying the argument with $k$ and $\ell$ switched, we get the term $\delta_{k \in P \cup R \cup \{1,\ell\}}$. Now suppose that $P \not \subseteq R \cup \{1,\ell,k\}$, i.e.\ there exists $p\in P$ s.t.\ $p \not \in  R \cup \{1,\ell,k\}.$ Then let $S \in \pi$ such that $p \in S$. Then we have that  
\begin{align*}
    &\int (K*\rho(x_1) -K(x_1,x_\ell))\cdot (K*\rho(x_1) -K(x_1,x_k))  \prod_{Q \in \pi}\tilde g^{i_Q}_Q \prod_{W \in \sigma} \tilde g^{i_W}_W \rho^{\otimes j}\,dx
    \\&\quad =  \int (K*\rho(x_1) -K(x_1,x_\ell))\cdot (K*\rho(x_1) -K(x_1,x_k))  
    \\&\qquad\quad \times \prod_{Q \in \pi - \{S\}}\tilde g^{i_Q}_Q\prod_{W \in \sigma} \tilde g^{i_W}_W \int \tilde g^{i_S}_S \rho^{\otimes j}\,dx_p dx_1 \cdots dx_{p-1} dx_{p+1} \cdots dx_j =0,
\end{align*}
where we use that 
 \[\int \tilde g^{i_S}_S \rho^{\otimes |S|}(x^S)\,dx_p = \int g^{i_S}_S\,dx_p = 0,\]
 by Lemma~\ref{lem.marginalization}, using $i_S \geq 1$. Thus we get the term $\delta_{P \subseteq R \cup \{1,\ell,k\}}$ and symmetrically get the term $\delta_{R \subseteq P \cup \{1,\ell,k\}}$, thus showing the claim~\eqref{eq.remainder-tensorization-cancellation}.

Thus we are left with bounding
\begin{align} &2i \sum_{m=1}^{2i} \sum_{\ell=1}^j \sum_{k=1}^j \sum_{\substack{P \subseteq [j]\\ |P| = m}} \sum_{\substack{R \subseteq [j]\\ |R| = m}}\sum_{\pi \vdash P} \sum_{\sigma \vdash R}\sum_{\substack{(i_Q)_{Q \in \pi}\\ \sum i_Q = i\\ i_Q \geq 1}}   \sum_{\substack{(i_W)_{W \in \sigma}\\ \sum i_W = i\\ i_W \geq 1}} \delta_{\substack{\ell \in P\cup R \cup \{1,k\}\\k \in P \cup R \cup \{1,\ell\}}} \delta_{\substack{P \subseteq R \cup \{1,\ell,k\}\\R \subseteq P \cup \{1,\ell,k\}}}
\notag\\&\quad= 2i \sum_{m=1}^{2i} \sum_{\ell=1}^j \sum_{k=1}^j \sum_{\substack{P \subseteq [j]\\ |P| = m}} \sum_{\substack{R \subseteq [j]\\ |R| = m}} \delta_{\substack{\ell \in P\cup R \cup \{1,k\}\\k \in P \cup R \cup \{1,\ell\}}} \delta_{\substack{P \subseteq R \cup \{1,\ell,k\}\\R \subseteq P \cup \{1,\ell,k\}}} \sum_{\pi \vdash P} \sum_{\sigma \vdash R}\sum_{\substack{(i_Q)_{Q \in \pi}\\ \sum i_Q = i\\ i_Q \geq 1}}   \sum_{\substack{(i_W)_{W \in \sigma}\\ \sum i_W = i\\ i_W \geq 1}} 1
\notag\\&\quad \leq C \sum_{m=1}^{2i} \sum_{\ell=1}^j \sum_{k=1}^j \sum_{\substack{P \subseteq [j]\\ |P| = m}} \sum_{\substack{R \subseteq [j]\\ |R| = m}}\delta_{\substack{\ell \in P\cup R \cup \{1,k\}\\k \in P \cup R \cup \{1,\ell\}}} \delta_{\substack{P \subseteq R \cup \{1,\ell,k\}\\R \subseteq P \cup \{1,\ell,k\}}},
\label{eq.remainder-bound-3}
\end{align}
where we use that
\[ \sum_{\pi \vdash P} \sum_{\sigma \vdash R}\sum_{\substack{(i_Q)_{Q \in \pi}\\ \sum i_Q = i\\ i_Q \geq 1}}   \sum_{\substack{(i_W)_{W \in \sigma}\\ \sum i_W = i\\ i_W \geq 1}} 1 =  \sum_{\pi \vdash [m]} \sum_{\sigma \vdash [m]}\sum_{\substack{(i_Q)_{Q \in \pi}\\ \sum i_Q = i\\ i_Q \geq 1}}   \sum_{\substack{(i_W)_{W \in \sigma}\\ \sum i_W = i\\ i_W \geq 1}} 1 = C(m,i) \leq C.\]
We now claim that
\[\delta_{\substack{P \subseteq R \cup \{1,\ell,k\}\\R \subseteq P \cup \{1,\ell,k\}}} = 0\]
unless $P =R$ or $P = R - \{a\}\cup \{b\}$ with $a \in R; a \ne b; a,b \in \{1, \ell,k\}.$

To see this, suppose $P \subseteq R \cup \{1,\ell,k\}$ and $R \subseteq P \cup \{1,\ell,k\}.$ Then note that the symmetric difference $P\,\Delta\,R = (P - R) \cup (R - P) \subseteq \{1,\ell,k\}.$  Then, since $|P| = |R|,$ we have that
\[|P -R| = |P| - |R\cap P| = |R| - |R \cap P| =  |R-P|.\]
Thus
\[|P\,\Delta\,R| = |P-R| + |R-P| = 2 |P-R|.\]
Thus $P\,\Delta\,R$ is an even sized subset of $\{1,\ell,k\}$, hence either $P=R$ or $P = R - \{a\}\cup \{b\}$ with $a \in R; a \ne b; a,b \in \{1, \ell,k\}$, as claimed. Let us first deal with the case that $P=R$, in which case the sum becomes 
\[ \sum_{\substack{P \subseteq [j]\\ |P| = m}} \sum_{\ell=1}^j \sum_{k=1}^j\delta_{\ell \in P  \cup \{1,k\}} \delta_{k \in P \cup \{1,\ell\}} \leq \sum_{\substack{P \subseteq [j]\\ |P| = m}} \sum_{\ell=1}^j m+2 \leq Cj^{2i+1},\]
using that $m \leq 2i.$

For $P\neq R$ the remaining part of the sum to bound is 
\begin{align*}&\sum_{\ell=1}^j \sum_{k=1}^j \sum_{\substack{R \subseteq [j]\\ |R| = m}} \sum_{a \in \{1,k,\ell\} \cap R} \sum_{b \in \{1,k,\ell\} - \{a\}} \delta_{\substack{\ell \in (R - \{a\} \cup \{b\}) \cup R \cup \{1,k\}\\ k \in (R - \{a\} \cup \{b\}) \cup R \cup \{1,\ell\}}} 
\\&\quad= \sum_{\substack{R \subseteq [j]\\ |R| = m}}\sum_{\ell=1}^j \sum_{k=1}^j \sum_{a \in \{1,k,\ell\} \cap R} \sum_{b \in \{1,k,\ell\} - \{a\}} \delta_{\substack{\ell \in R \cup \{1,k,b\}\\ k \in R \cup \{1,\ell,b\}}} 
\\&\quad \leq C_i j^{2i+1} +\sum_{\substack{R \subseteq [j]\\ |R| = m}}\sum_{\ell=2}^j \sum_{k=2, k \ne \ell}^j \sum_{a \in \{1,k,\ell\} \cap R} \sum_{b \in \{1,k,\ell\} - \{a\}} \delta_{\substack{\ell \in R \cup \{b\}\\ k \in R \cup \{b\}}},
\end{align*}
where on the last line we split off the three cases $\ell =1, k=1,$ and $\ell =k$ and apply the straightforward bounds to them separately. We lastly split the remaining term along the cases $a=1, a=k,$ and $a=\ell$. The first case $a=1$ gives
\[\sum_{\substack{R \subseteq [j]\\ |R| = m}}\sum_{\ell=2}^j \sum_{k=2, k \ne \ell}^j \sum_{b \in \{k,\ell\}} \delta_{\substack{\ell \in R \cup \{b\}\\ k \in R \cup \{b\}}} \delta_{1 \in R} \leq\sum_{\substack{W \subseteq [j]-\{1\}\\ |W| = m-1}} \sum_{\ell=2}^j \sum_{k=2, k \ne \ell}^j 2 \leq C j^{2i+1}.\]
The second case $a=k$ gives 
\[\sum_{\substack{R \subseteq [j]\\ |R| = m}}\sum_{\ell=2}^j \sum_{k\in R, k \ne \ell}  \sum_{b \in \{1,\ell\}} \delta_{\substack{\ell \in R \cup \{b\}\\ k \in R \cup \{b\}}} \leq 2m \sum_{\substack{R \subseteq [j]\\ |R| = m}}\sum_{\ell=2}^j \leq C j^{2i+1}.\]
The third case $a=\ell$ follows symmetrically. Thus 
\begin{equation}
     \sum_{m=1}^{2i} \sum_{\ell=1}^j \sum_{k=1}^j \sum_{\substack{P \subseteq [j]\\ |P| = m}} \sum_{\substack{R \subseteq [j]\\ |R| = m}}\delta_{\substack{\ell \in P\cup R \cup \{1,k\}\\k \in P \cup R \cup \{1,\ell\}}} \delta_{\substack{P \subseteq R \cup \{1,\ell,k\}\\R \subseteq P \cup \{1,\ell,k\}}}
    \leq \sum_{m=1}^{2i} C j^{2i+1} \leq C j^{2i+1}.
    \label{eq.remainder-bound-4}
\end{equation}
Combining~\eqref{eq.remainder-bound-1},~\eqref{eq.remainder-bound-2},~\eqref{eq.remainder-tensorization-cancellation},~\eqref{eq.remainder-bound-3}, and~\eqref{eq.remainder-bound-4}, we conclude.
\end{proof}

\section{Proofs of cluster expansions and perturbation theory}\label{s.algebraic-proofs}

We give some additional notation for partitions.
\begin{definition}
     We define the following partial order on partitions. If $\sigma, \pi \vdash A$, we say that $\sigma \leq \pi$ if for every $P \in \pi$, there exists $Q \in \sigma$ such that $P \subseteq Q.$ If $\sigma \leq \pi$ we say $\sigma$ is a \textit{combining of} $\pi$. We note that if $\sigma \leq \pi \vdash A,$ then $|\sigma| \leq |\pi| \leq |A|.$
\end{definition}

We note the following combinatoric lemma which we will appeal to frequently in the below proofs.

\begin{lemma}
    \label{lem.combinings-combinatorics}
    Let $S$ be a finite set, $\pi \vdash S.$ Then
    \[\sum_{\sigma \leq \pi} (-1)^{|\sigma| -1} (|\sigma|-1)! =  \begin{cases} 
    1 & |\pi| =1,\\
     0 & |\pi| \geq 2.
    \end{cases}\]
\end{lemma}
\begin{proof}
In order to evaluate these sums, we take advantage of the natural isomorphism from partitions of $\pi$ to combinings of $\pi$. For $\Pi \vdash \pi,$ we let
\[\sigma(\Pi) = \left\{\bigcup_{P \in \alpha} P : \alpha \in \Pi\right\}.\]
Note that $\sigma$ defines a bijection between partitions of $\pi$ and combinings of $\pi$, and further that $|\sigma(\Pi)| = |\Pi|$. This immediately implies that 
\[\sum_{\sigma \leq \pi} (-1)^{|\sigma| -1} (|\sigma|-1)!=\sum_{\Pi \vdash \pi} (-1)^{|\Pi| -1} (|\Pi|-1)!.\]
The lemma then follows after applying the following fact
    \[\sum_{\alpha \vdash [j]} (-1)^{|\alpha| -1} (|\alpha|-1)! = \begin{cases} 
    1 & j =1,\\
     0 & j \geq 2,
    \end{cases}\]
    which follows by the Fa\`a di Bruno's formula applied to $\log e^x$.
\end{proof}

\begin{proof}[Proof of Proposition \ref{prop.cluster-representation}]
    We prove the equality inductively in $j$. The case $j=1$ is clear. For $j \geq 2$, we have
    \begin{align*}
         g_j &=\sum_{\pi \vdash [j]} (-1)^{|\pi| -1} (|\pi|-1)!\prod_{P \in \pi} f_P
         \\&= f_j+\sum_{\pi \vdash [j],|\pi|\geq 2} (-1)^{|\pi| -1} (|\pi|-1)!\prod_{P \in \pi} f_P
         \\&=  f_j+\sum_{\pi \vdash [j],|\pi|\geq 2} (-1)^{|\pi| -1} (|\pi|-1)!\prod_{P \in \pi} \sum_{\sigma \vdash P } \prod_{Q \in \sigma} g_Q
        \\&=  f_j+\sum_{\pi \vdash [j],|\pi|\geq 2} (-1)^{|\pi| -1} (|\pi|-1)! \sum_{\sigma \geq \pi } \prod_{Q \in \sigma} g_Q
        \\&= f_j+ \sum_{\sigma \vdash [j]} \sum_{\pi \leq \sigma, |\pi| \geq 2}(-1)^{|\pi| -1} (|\pi|-1)! \prod_{P \in \sigma} g_P.
    \end{align*}
    We have now collected all the terms $\prod_{P \in \sigma} g_P$ so all that remains is to compute the combinatoric constants. To that end, using Lemma~\ref{lem.combinings-combinatorics}, we have
    \[
       \sum_{\pi \leq \sigma, |\pi| \geq 2}(-1)^{|\pi| -1} (|\pi|-1)! = -1 + \sum_{\pi \leq \sigma}(-1)^{|\pi| -1} (|\pi|-1)!=  \begin{cases}
        0 & |\sigma| = 1,\\ -1 & |\sigma| \geq 2.
        \end{cases}
    \]
    Plugging this in above allows us to conclude.
\end{proof}

\begin{proof}[Proof of Proposition \ref{prop.cluster-function-equation}]
    We start by just directly computing $(\partial_t - \Delta)g_j$ using the definition of $g_j$ in terms of the $f_j$, 
    \begin{align}
        \partial_t g_j - \Delta g_j &=  \sum_{\pi \vdash [j]} (-1)^{|\pi| -1} (|\pi|-1)!\sum_{P \in \pi} (\partial_t f_P - \Delta f_P) \prod_{\substack{Q \in \pi\\ Q \ne P}} f_Q
        \notag\\& =-\sum_{\pi \vdash [j]} (-1)^{|\pi| -1} (|\pi|-1)!\sum_{P \in \pi} \left( \frac{N-|P|}{N} \sum_{k \in P}  H_k f_{P \cup \{*\}} +  \frac{1}{N} \sum_{k,\ell \in P} S_{k,\ell} f_P\right)\prod_{\substack{Q \in \pi\\ Q \ne P}} f_Q. \label{eq.g-j-comp}
    \end{align}
    Note that, consistent with the definition of $H_k$, the variable $*$ is always the coordinate being integrated over. We consider the $H_k$ terms and the $S_{k,\ell}$ terms separately. We first consider the $S_{k,\ell}$ terms. We use Proposition \ref{prop.cluster-representation} to expand each of the $f_R$ in terms of $g_Q$
    \begin{align*}
       &\sum_{\pi \vdash [j]} (-1)^{|\pi| -1} (|\pi|-1)!\sum_{P \in \pi}   \sum_{k,\ell \in P} S_{k,\ell} f_P\prod_{\substack{Q \in \pi\\ Q \ne P}} f_Q 
       \\&\qquad =\sum_{\pi \vdash [j]} (-1)^{|\pi| -1} (|\pi|-1)!\sum_{P \in \pi}   \sum_{k,\ell \in P} \sum_{\pi \leq \sigma} S_{k,\ell}\prod_{Q \in \sigma} g_Q 
       \\&\qquad=\sum_{k,\ell=1}^j \sum_{\sigma \vdash [j]} \bigg(\sum_{\substack{\pi \leq \sigma\\\exists P \in \pi, \{k,\ell\} \subseteq P}} (-1)^{|\pi| -1} (|\pi|-1)! \bigg)S_{k,\ell} \prod_{Q \in \sigma} g_Q 
       \\&\qquad=: \sum_{k,\ell=1}^j \sum_{\sigma \vdash [j]} a_{k,\ell}^\sigma S_{k,\ell} \prod_{P \in \sigma} g_P.
    \end{align*}
    We now compute $a_{k,\ell}^\sigma$. Fix $\sigma, k,\ell$. We split into two cases, the first being that there exists $Q \in \sigma$ such that $\{k,\ell\} \subseteq Q$ and second being that there exists $Q, R \in \sigma, Q \ne R$ such that $k \in Q, \ell \in R$. Note that in the second case, it must be that $k\neq \ell$.

    In the first case, since for any $\pi \leq \sigma$, by definition of the order, there exists $P \in \pi$ such that $Q \subseteq P$, as such $\{k,\ell\} \subseteq P$. Thus 
    \begin{align*}
    a_{k,\ell}^\sigma &= \sum_{\substack{\pi \leq \sigma\\\exists P \in \pi, \{k,\ell\} \subseteq P}} (-1)^{|\pi| -1} (|\pi|-1)!
    \\& = \sum_{\pi \leq \sigma} (-1)^{|\pi| -1} (|\pi|-1)!
    \\&= \begin{cases}
        1 & |\sigma| = 1\\
        0 & |\sigma| \geq 2,
    \end{cases}
    \end{align*}
    where we use Lemma~\ref{lem.combinings-combinatorics} to conclude.

    For the second case, we first write $\sigma = \{Q,R,W_1,...,W_m\}$ such that $k \in Q, \ell \in R$. Then define
    \[\tilde \sigma := \{Q \cup R, W_1,...,W_m\}.\]
    Then we note that $\pi \leq \sigma$ for which there exists $P \in \pi$ such that $\{k,\ell\} \subseteq P$ if and only if $\pi \leq \tilde \sigma.$ Thus 
    \[
    a_{k,\ell}^\sigma =  \sum_{\pi \leq \tilde \sigma} (-1)^{|\pi| -1} (|\pi|-1)! = \begin{cases}
        1 & |\sigma| = 2\\
        0 & |\sigma| \geq 3,
    \end{cases}
    \]
    once again using Lemma~\ref{lem.combinings-combinatorics} and noting $|\sigma| = |\tilde \sigma| +1.$

    Combining these two cases, we note the complete formula for $a_{k,\ell}^\sigma$ is given by 
    \begin{equation}
    \label{eq.a-sig-formula}
    a_{k,\ell}^\sigma = \begin{cases} 1 & |\sigma| =1\\ 1 & \sigma= \{Q, R\}, k \in Q, \ell\in R\\ 
    0 & \sigma = \{Q,R\}, k,\ell \in Q\\
    0 & |\sigma| \geq 3.
    \end{cases}    
    \end{equation}

    We have thus dealt with the $S_{k,\ell}$ terms completely. We now proceed to the $H_k$ terms. Similarly, we compute
    \begin{align*}
        &\sum_{\pi \vdash [j]} (-1)^{|\pi| -1} (|\pi|-1)!\sum_{P \in \pi} \frac{N-|P|}{N} \sum_{k \in P}  H_k f_{P \cup \{*\}}\prod_{\substack{Q \in \pi\\ Q \ne P}} f_Q
        \\&\qquad= \sum_{\pi \vdash [j]} (-1)^{|\pi|-1} (|\pi|-1)! \sum_{P \in \pi} \frac{N-|P|}{N}\sum_{k \in P} \sum_{\tilde \pi \leq \sigma} H_k \prod_{Q \in \sigma} g_Q, 
    \end{align*}
    where if $\pi = \{P,W_1,...,W_m\}$, then, letting $\tilde P := P \cup \{*\},$ we define
    \[\tilde \pi := \{\tilde P,W_1,...,W_m\} \vdash [j] \cup\{*\}.\]
    Continuing the above computation and reindexing sums, we get 
    \begin{align*}&\sum_{\pi \vdash [j]} (-1)^{|\pi|-1} (|\pi|-1)! \sum_{P \in \pi} \frac{N-|P|}{N}\sum_{k \in P} \sum_{\tilde \pi \leq \sigma} H_k \prod_{Q \in \sigma} g_Q 
    \\&\qquad= \sum_{k=1}^j \sum_{ \sigma \vdash [j] \cup \{*\}}  \bigg(\sum_{\substack{\tilde \pi \leq \sigma\\ \exists \tilde P \in \tilde \pi, \{k,*\} \subseteq \tilde P }}  (-1)^{|\tilde \pi|-1} (|\tilde \pi|-1)! \frac{N-|\tilde P|+1}{N}\bigg)H_k\prod_{Q \in \sigma} g_{ Q}
    \\&\qquad =:\sum_{k=1}^j \sum_{ \sigma \vdash [j] \cup \{*\}} b_k^\sigma H_k\prod_{Q \in \sigma} g_{ Q},
    \end{align*}
    where we note $\tilde \pi$ is a partition of the larger set $[j] \cup \{*\}.$ We now compute $b_k^\sigma$. Similarly to above, we split according to whether the relevant variables $k,*$ are in the same block of $\sigma$. Writing $\sigma = \{Q,R,W_1,...,W_m\},$ the first case is that $\{k, *\} \subseteq Q$ and the second case is that $k \in Q, * \in R.$

    In the first case, we note that if $\tilde \pi \leq \sigma$, then by definition there exists $\tilde P \in \tilde \pi$ such that $Q \subseteq \tilde P$, and as such $\{k,*\} \subseteq \tilde P$. Thus 
    \begin{align}\notag b^\sigma_k &= \sum_{\tilde \pi \leq \sigma}  (-1)^{|\tilde \pi|-1} (|\tilde \pi|-1)! \frac{N-|\tilde P|+1}{N}
    \\&= \sum_{\rho \leq \sigma - \{Q\}} (-1)^{|\rho|} |\rho|! \frac{N-|Q|+1}{N}+  \sum_{S \in \rho}  (-1)^{|\rho|-1} (| \rho|-1)! \frac{N-|S \cup Q|+1}{N} \label{eq.b-sig-comp},
    \end{align}
    where for the second equality, in order to deal with the term $|\tilde P|$, we look at possible ways of constructing $\tilde P$. We note that any combining $\tilde \pi \leq \sigma$ is generated by first taking a combing $\rho \leq \sigma - \{Q\}$ and then either adding $Q$ as its own block or unioning $Q$ with a block of $\rho$. This corresponds to the first and second term in the sum respectively.

    The above computation is valid in the case that $Q = [j] \cup \{*\},$ but for the sake the analysis to follow let us deal with this edge case now. A direct computation verifies that in that case, we have that $|\sigma| = 1$ and 
    \[b_k^\sigma = \frac{N-j}{N}.\]
    Continuing the above computation with the additional assumption that $|\sigma| \geq 2$, we first note that
    \begin{align}\notag
    &\sum_{S \in \rho}  (-1)^{|\rho|-1} (| \rho|-1)! \frac{N-|S \cup Q|+1}{N} 
    \\&\notag\qquad= \sum_{S \in \rho}  (-1)^{|\rho|-1} (| \rho|-1)! \frac{N-|Q|+1}{N}- \sum_{S \in \rho}  (-1)^{|\rho|-1} (| \rho|-1)! \frac{|S|}{N}
    \\&\qquad = (-1)^{|\rho|-1} | \rho|! \frac{N-|Q|+1}{N} -  (-1)^{|\rho|-1} (| \rho|-1)! \frac{j+1-|Q|}{N}.\label{eq.rho-comp}
    \end{align}
    For the last equality, we use that the first term doesn't depend on $S$, and as such we just get a multiplicative factor of $|\rho|$, which then goes into the factorial. For the second term, we use that 
    \[\sum_{S \in \rho} |S|= |[j] \cup \{*\} - Q| = j+1-|Q|.\]
    Then, plugging~\eqref{eq.rho-comp} into~\eqref{eq.b-sig-comp} and noting the cancellation of the first two terms, we have that
    \[b^\sigma_k =-\frac{j+1-|Q|}{N}\sum_{\rho \leq \sigma - \{Q\}}(-1)^{|\rho|-1} (| \rho|-1)! = \begin{cases}
        -\frac{j+1-|Q|}{N} & |\sigma| = 2\\
        0 & |\sigma| \geq 3,
    \end{cases}  \]
    where we have once again used Lemma~\ref{lem.combinings-combinatorics}. Then recalling the above remarks on the case that $|\sigma| =1,$ we have that
    \[b^\sigma_k = \begin{cases}
        \frac{N-j}{N} & |\sigma| = 1\\
        -\frac{j+1-|Q|}{N} & |\sigma| = 2\\
        0 & |\sigma| \geq 3.
    \end{cases}\]

    We now consider the other case, that $k \in Q, * \in R.$ We then, similarly to the analysis for the $S_{k,\ell}^\sigma$ terms, define
    \[\tilde \sigma := \{Q \cup R, W_1,...,W_m\}.\]
    Then we note that 
    \[b^\sigma_k = \sum_{\substack{\tilde \pi \leq \sigma\\ \exists \tilde P \in \tilde \pi, \{k,*\} \subseteq \tilde P }}  (-1)^{|\tilde \pi|-1} (|\tilde \pi|-1)! \frac{N-|\tilde P|+1}{N}
    = \sum_{\tilde \pi \leq \tilde \sigma}  (-1)^{|\tilde \pi|-1} (|\tilde \pi|-1)! \frac{N-|\tilde P|+1}{N}.\]
    We note now that we are in the same setting as we were for the previous case, except with $\tilde \sigma$ in place of $\sigma$ and $Q\cup R$ in place of $Q$. As such, the same computations demonstrate that, in this case,
    \[b^\sigma_k  =  \begin{cases}
        \frac{N-j}{N} & |\sigma| = 2\\
        -\frac{j+1-|Q| - |R|}{N} & |\sigma| = 3\\
        0 & |\sigma| \geq 4,
    \end{cases}\]
    where we note that $|\sigma| = |\tilde \sigma|+1$. Thus, in total, we have that
    \begin{equation}
    \label{eq.b-sig-formula}
    b^\sigma_k = \begin{cases}
        \frac{N-j}{N} & |\sigma| =1\\
        -\frac{j+1- |Q|}{N} & \sigma = \{Q,R\}, \{k,*\} \subseteq Q\\
        \frac{N-j}{N} & \sigma = \{Q,R\}, k \in Q, * \in R\\
        0 & \sigma = \{Q,R,W\}, \{k,*\} \subseteq Q\\
        -\frac{j+1 - |Q| - |R|}{N} & \sigma = \{Q,R, W\}, k \in Q, * \in R\\
        0 & |\sigma| \geq 4.
    \end{cases}
    \end{equation}

    We have thus computed all the coefficients, so we can plug in~\eqref{eq.a-sig-formula} and~\eqref{eq.b-sig-formula} into~\eqref{eq.g-j-comp} to give the PDE $g_j$ solves.

    For the initial conditions, we remark that as $f_j(0,\cdot)=f^{\otimes j}$, the equation \eqref{eq.forward-cluster-definition} gives that
    \[
g_j= f^{\otimes j}\sum_{\pi \vdash [j]} (-1)^{|\pi| -1} (|\pi|-1)! 
    \]
    thus Lemma \ref{lem.combinings-combinatorics} gives the stated initial conditions.
\end{proof}

\begin{proof}[Proof of Proposition \ref{prop.f-ij-equations}]
    Computing $(\partial_t - \Delta)f^i_j$ using its definition we get
    \[ (\partial_t - \Delta)f^i_j = \sum_{\pi \vdash [j]} \sum_{\substack{(i_P)_{P \in \pi}\\ \sum i_P = i}} \sum_{P \in \pi} (\partial_t - \Delta)g^{i_P}_P\prod_{Q \in \pi -\{P\}} g^{i_Q}_Q.\]
    Then using~\eqref{eq.g-ij-PDE}, this becomes
    \begin{align*}
    \sum_{\pi \vdash [j]} \sum_{\substack{(i_P)_{P \in \pi}\\ \sum i_P = i}} \sum_{P \in \pi} \Bigg(& -\sum_{k \in P} H_k g^{i_P}_{P \cup \{*\}}\prod_{Q \in \pi -\{P\}} g^{i_Q}_Q
    \\&- \sum_{k \in P}\sum_{W \subseteq P -\{k\}} \sum_{m=0}^{i_P} H_k g^m_{W \cup \{k\}} g^{i_P-m}_{P \cup \{*\} - W -\{k\}} \prod_{Q \in \pi -\{P\}} g^{i_Q}_Q
    \\&+  \sum_{k \in P} |P| H_k g^{i_P-1}_{P \cup \{*\}}\prod_{Q \in \pi -\{P\}} g^{i_Q}_Q
    \\&+\sum_{k \in P} \sum_{W \subseteq P -\{k\}} (|P|-1-|W|) \sum_{m=0}^{i_P-1} H_k g^m_{W \cup \{k,*\}} g^{i_P-1-m}_{|P| - \{k\} - W} \prod_{Q \in \pi -\{P\}} g^{i_Q}_Q
    \\&+ \sum_{k \in P} \sum_{W \subseteq P - \{k\}} \sum_{m=0}^{i_P-1} |P| H_k g^m_{W \cup \{k\}} g^{i_P-1-m}_{P \cup \{*\} - W -\{k\}} \prod_{Q \in \pi -\{P\}} g^{i_Q}_Q
    \\&+  \sum_{k \in P} \sum_{W \subseteq P - \{k\}} \sum_{R \subseteq P - \{k\} - W} (|P| - 1- |W| - |R|) 
    \\&\qquad\qquad\times\sum_{m=0}^{i_P-1} \sum_{n=0}^{i_P-1-m} H_k g^m_{W \cup \{k\}} g^n_{R \cup \{*\}} g^{i_P-1-m-n}_{P - R - W - \{k\}} \prod_{Q \in \pi -\{P\}} g^{i_Q}_Q
    \\&-  \sum_{k,\ell \in P} S_{k,\ell} g^{i_P -1}_P \prod_{Q \in \pi -\{P\}} g^{i_Q}_Q
    \\&- \sum_{\substack{k,\ell \in P\\k\neq\ell}} \sum_{W \subseteq P - \{k,\ell\}} \sum_{m=0}^{i_P-1} S_{k,\ell} g^m_{W \cup \{k\}} g^{i_P -1 -m}_{P - \{k\} -W}\prod_{Q \in \pi -\{P\}} g^{i_Q}_Q\Bigg).
\end{align*}
In order to conclude, we must show the above is equal to
\begin{align*}
&-\sum_{k=1}^j H_k f^i_{[j]\cup\{*\}} + j\sum_{k=1}^j H_k f^{i-1}_j - \sum_{k,\ell=1}^j S_{k,\ell} f_j^{i-1}
\\&\qquad= -\sum_k  \sum_{\pi \vdash [j] \cup \{*\}} \sum_{\substack{(i_P)_{P \in \pi}\\ \sum i_P = i}} H_k\prod_{P \in \pi} g^{i_P}_P
\\&\qquad\qquad+ j \sum_k  \sum_{\pi \vdash [j] \cup \{*\}} \sum_{\substack{(i_P)_{P \in \pi}\\ \sum i_P = i-1}} H_k\prod_{P \in \pi} g^{i_P}_P
\\&\qquad\qquad- \sum_{k,\ell}  \sum_{\pi \vdash [j]} \sum_{\substack{(i_P)_{P \in \pi}\\ \sum i_P = i-1}} S_{k,\ell}\prod_{P \in \pi} g^{i_P}_P.
\end{align*}
In particular, we show the following three claims.

\textit{Claim 1:}
\begin{align*}
     &\sum_k  \sum_{\pi \vdash [j] \cup \{*\}} \sum_{\substack{(i_P)_{P \in \pi}\\ \sum i_P = i}} H_k\prod_{P \in \pi} g^{i_P}_P  
     =\sum_{\pi \vdash [j]} \sum_{\substack{(i_P)_{P \in \pi}\\ \sum i_P = i}} \sum_{P \in \pi} \sum_{k \in P} H_k g^{i_P}_{P \cup \{*\}}\prod_{Q \in \pi -\{P\}} g^{i_Q}_Q
    \\&\qquad\qquad+ \sum_{\pi \vdash [j]} \sum_{\substack{(i_P)_{P \in \pi}\\ \sum i_P = i}} \sum_{P \in \pi} \sum_{k \in P}\sum_{W \subseteq P -\{k\}} \sum_{m=0}^{i_P} H_k g^m_{W \cup \{k\}} g^{i_P-m}_{P \cup \{*\} - W -\{k\}} \prod_{Q \in \pi -\{P\}} g^{i_Q}_Q.
\end{align*}

\textit{Claim 2:}

\begin{align}
   & \notag j \sum_k  \sum_{\pi \vdash [j] \cup \{*\}} \sum_{\substack{(i_P)_{P \in \pi}\\ \sum i_P = i-1}} H_k\prod_{P \in \pi} g^{i_P}_P 
   \\&\label{eq.claim-2-1}\quad=  \sum_{\pi \vdash [j]} \sum_{\substack{(i_P)_{P \in \pi}\\ \sum i_P = i}} \sum_{P \in \pi}  \sum_{k \in P} |P| H_k g^{i_P-1}_{P \cup \{*\}}\prod_{Q \in \pi -\{P\}} g^{i_Q}_Q
    \\&\label{eq.claim-2-2}\qquad\sum_{\pi \vdash [j]} \sum_{\substack{(i_P)_{P \in \pi}\\ \sum i_P = i}} \sum_{P \in \pi} \sum_{k \in P} \sum_{W \subseteq P -\{k\}} (|P|-1-|W|) \sum_{m=0}^{i_P-1} H_k g^m_{W \cup \{k,*\}} g^{i_P-1-m}_{|P| - \{k\} - W} \prod_{Q \in \pi -\{P\}} g^{i_Q}_Q
    \\&\label{eq.claim-2-3}\qquad \sum_{\pi \vdash [j]} \sum_{\substack{(i_P)_{P \in \pi}\\ \sum i_P = i}} \sum_{P \in \pi} \sum_{k \in P} \sum_{W \subseteq P - \{k\}} \sum_{m=0}^{i_P-1} |P| H_k g^m_{W \cup \{k\}} g^{i_P-1-m}_{P \cup \{*\} - W -\{k\}} \prod_{Q \in \pi -\{P\}} g^{i_Q}_Q
    \\&\label{eq.claim-2-4}\qquad  \sum_{\pi \vdash [j]} \sum_{\substack{(i_P)_{P \in \pi}\\ \sum i_P = i}} \sum_{P \in \pi} \sum_{k \in P} \sum_{W \subseteq P - \{k\}} \sum_{R \subseteq P - \{k\} - W} (|P| - 1- |W| - |R|) 
    \\&\notag\qquad\qquad\qquad\times\sum_{m=0}^{i_P-1} \sum_{n=0}^{i_P-1-m} H_k g^m_{W \cup \{k\}} g^n_{R \cup \{*\}} g^{i_P-1-m-n}_{P - R - W - \{k\}} \prod_{Q \in \pi -\{P\}} g^{i_Q}_Q.
\end{align}

\textit{Claim 3:}

\begin{align*}
    &\sum_{k,\ell}  \sum_{\pi \vdash [j]} \sum_{\substack{(i_P)_{P \in \pi}\\ \sum i_P = i-1}} S_{k,\ell}\prod_{P \in \pi} g^{i_P}_P = \sum_{\pi \vdash [j]} \sum_{\substack{(i_P)_{P \in \pi}\\ \sum i_P = i}} \sum_{P \in \pi}\sum_{k,\ell \in P} S_{k,\ell} g^{i_P -1}_P \prod_{Q \in \pi -\{P\}} g^{i_Q}_Q
    \\&\qquad\qquad+ \sum_{\pi \vdash [j]} \sum_{\substack{(i_P)_{P \in \pi}\\ \sum i_P = i}} \sum_{P \in \pi}\sum_{\substack{k,\ell \in P\\k\neq\ell}} \sum_{W \subseteq P - \{k,\ell\}} \sum_{m=0}^{i_P-1} S_{k,\ell} g^m_{W \cup \{k\}} g^{i_P -1 -m}_{P - \{k\} -W}\prod_{Q \in \pi -\{P\}} g^{i_Q}_Q.
\end{align*}

For Claim 1, we note that the first term of the right hand side is simply a sum over all partitions $\pi \vdash [j] \cup \{*\}$ such that $k$ and $*$ are in the same block of $\pi $ (together with all choices of orders $i_P$). Then the second term on the right hand side is a sum over all partitions $\pi$ such that $k$ and $*$ are in the different blocks of $\pi$. Thus together they give a sum over all partitions, which is equal then to the left hand side. Symbolically 
\begin{align*}
    &\sum_{\pi \vdash [j]} \sum_{\substack{(i_P)_{P \in \pi}\\ \sum i_P = i}} \sum_{P \in \pi} \sum_{k \in P} H_k g^{i_P}_{P \cup \{*\}}\prod_{Q \in \pi -\{P\}} g^{i_Q}_Q
    \\&\qquad\qquad+ \sum_{\pi \vdash [j]} \sum_{\substack{(i_P)_{P \in \pi}\\ \sum i_P = i}} \sum_{P \in \pi} \sum_{k \in P}\sum_{W \subseteq P -\{k\}} \sum_{m=0}^{i_P} H_k g^m_{W \cup \{k\}} g^{i_P-m}_{P \cup \{*\} - W -\{k\}} \prod_{Q \in \pi -\{P\}} g^{i_Q}_Q
    \\&= \sum_k  \sum_{\substack{\pi \vdash [j] \cup \{*\}\\ \exists P \in \pi, \{k, *\} \subseteq P}} \sum_{\substack{(i_P)_{P \in \pi}\\ \sum i_P = i}} H_k\prod_{P \in \pi} g^{i_P}_P  
    \\&\qquad\qquad+ \sum_k  \sum_{\substack{\pi \vdash [j] \cup \{*\}\\ \exists P \in \pi, k \in P, * \not \in P}} \sum_{\substack{(i_P)_{P \in \pi}\\ \sum i_P = i}} H_k\prod_{P \in \pi} g^{i_P}_P  
    \\&= \sum_k  \sum_{\pi \vdash [j] \cup \{*\}} \sum_{\substack{(i_P)_{P \in \pi}\\ \sum i_P = i}} H_k\prod_{P \in \pi} g^{i_P}_P.
\end{align*}

For Claim 2, we note the first and second terms,~\eqref{eq.claim-2-1} and~\eqref{eq.claim-2-2}, sum over the same partitions, namely those partitions $\pi \vdash [j] \cup \{*\}$ such that $k,*$ are in the same block of $\pi$. Thus there is ``overcounting'' and we have to compute the correct constant prefactor on each such partition. Reindexing~\eqref{eq.claim-2-1}, we get
\[ \sum_{\pi \vdash [j]} \sum_{\substack{(i_P)_{P \in \pi}\\ \sum i_P = i}} \sum_{P \in \pi} \sum_{k \in P} |P|H_k g^{i_P-1}_{P \cup \{*\}}\prod_{Q \in \pi -\{P\}} g^{i_Q}_Q =  \sum_k \sum_{\substack{\pi \vdash [j] \cup \{*\}\\ \exists P \in \pi, \{k,*\} \subseteq P}} \sum_{\substack{(i_P)_{P \in \pi}\\ \sum i_P = i-1}} |P|H_k \prod_{P \in \pi} g^{i_P}_P.\]

Then reindexing~\eqref{eq.claim-2-2}, we get
\begin{align*}
    &\sum_{\pi \vdash [j]} \sum_{\substack{(i_P)_{P \in \pi}\\ \sum i_P = i}} \sum_{P \in \pi}\sum_{k \in P} \sum_{W \subseteq P -\{k\}} (|P|-1-|W|) \sum_{m=0}^{i_P-1} H_k g^m_{W \cup \{k,*\}} g^{i_P-1-m}_{P - \{k\} - W} \prod_{Q \in \pi -\{P\}} g^{i_Q}_Q
        \\&\qquad= \sum_k \sum_{\substack{\pi \vdash [j] \cup \{*\}\\ \exists A \in \pi, \{k,*\} \subseteq A}} \sum_{\substack{(i_P)_{P \in \pi}\\ \sum i_P = i-1}} \sum_{B\in\pi-\{A\}}|B| H_k g^{i_A}_{A} g^{i_B}_{B}\prod_{Q \in \pi -\{A,B\}} g^{i_Q}_Q
        \\&\qquad=\sum_k \sum_{\substack{\pi \vdash [j] \cup \{*\}\\ \exists P \in \pi, \{k,*\} \subseteq P}} \sum_{\substack{(i_P)_{P \in \pi}\\ \sum i_P = i-1}} (j-|P|)H_k \prod_{Q \in \pi} g^{i_Q}_Q.
\end{align*}
Thus adding them together, we get 
\begin{equation}
\label{eq.claim-2-together}
j \sum_k  \sum_{\substack{\pi \vdash [j] \cup \{*\}\\\exists P \in \pi, \{k,*\}\subseteq P}} \sum_{\substack{(i_P)_{P \in \pi}\\ \sum i_P = i-1}} H_k\prod_{P \in \pi} g^{i_P}_P.
\end{equation}

Similarly, the third and fourth terms,~\eqref{eq.claim-2-3} and~\eqref{eq.claim-2-4}, sum over the same set of partitions, namely those partitions $\pi \vdash [j] \cup \{*\}$ such that $k,*$ are in different blocks. So we again reindex to compute the constant prefactors. Reindexing~\eqref{eq.claim-2-3}, we get 
\begin{align*}
&\sum_{\pi \vdash [j]} \sum_{\substack{(i_P)_{P \in \pi}\\ \sum i_P = i}} \sum_{P \in \pi} \sum_{k \in P} \sum_{W \subseteq P - \{k\}} \sum_{m=0}^{i_P-1} |P| H_k g^m_{W \cup \{k\}} g^{i_P-1-m}_{P \cup \{*\} - W -\{k\}} \prod_{Q \in \pi -\{P\}} g^{i_Q}_Q
\\&\qquad= \sum_k \sum_{\substack{\pi \vdash [j] \cup \{*\}\\ \exists A,B \in \pi, k \in A, * \not \in B, A \ne B}} \sum_{\substack{(i_P)_{P \in \pi}\\ \sum i_P = i}} (|A| + |B|-1) H_k \prod_{P \in \pi} g^{i_P}_P.
\end{align*}

Finally, reindexing~\eqref{eq.claim-2-4}, we get
\begin{align*}
    &\sum_{\pi \vdash [j]} \sum_{\substack{(i_P)_{P \in \pi}\\ \sum i_P = i}} \sum_{P \in \pi} \sum_{k \in P} \sum_{W \subseteq P - \{k\}} \sum_{R \subseteq P - \{k\} - W} (|P| - 1- |W| - |R|) 
    \\&\qquad\qquad\times\sum_{m=0}^{i_P-1} \sum_{n=0}^{i_P-1-m} H_k g^m_{W \cup \{k\}} g^n_{R \cup \{*\}} g^{i_P-1-m-n}_{P - R - W - \{k\}} \prod_{Q \in \pi -\{P\}} g^{i_Q}_Q
    \\&\qquad=\sum_k \sum_{\substack{\pi \vdash [j] \cup \{*\}\\ \exists A,B \in \pi, k \in A, * \not \in B, A \ne B}}  \sum_{\substack{(i_P)_{P \in \pi}\\ \sum i_P = i}} \sum_{C \in \pi -\{A,B\}} |C| H_k g^{i_A}_A g^{i_B}_B g^{i_C}_C \prod_{Q \in \pi - \{A,B,C\}} g^{i_Q}_Q  
    \\&\qquad=\sum_k \sum_{\substack{\pi \vdash [j] \cup \{*\}\\ \exists A,B \in \pi, k \in A, * \not \in B, A \ne B}}  \sum_{\substack{(i_P)_{P \in \pi}\\ \sum i_P = i}} \sum_{C \in \pi -\{A,B\}} (j+1 - |A| - |B|) H_k \prod_{P \in \pi} g^{i_P}_P. 
\end{align*}
So adding these together, we get
\begin{equation}
\label{eq.claim-2-apart}
j \sum_k  \sum_{\substack{\pi \vdash [j] \cup \{*\}\\\exists P \in \pi, k \in P, * \not \in P}} \sum_{\substack{(i_P)_{P \in \pi}\\ \sum i_P = i-1}} H_k\prod_{P \in \pi} g^{i_P}_P.
\end{equation}
Thus, adding~\eqref{eq.claim-2-together} and~\eqref{eq.claim-2-apart}, we have shown Claim 2.

Lastly, Claim 3 follows exactly as Claim 1.
\end{proof}

\section{Extending Theorem~\ref{thm.main-result-intro} to the general setting}

\label{s.extend}

We now remove the assumptions of Section~\ref{s.hierarchy} that $\Omega = \T^d$ and the initial data $f$ is such that $f + f^{-1} \in L^\infty$, generalizing Proposition~\ref{prop.main-result-special} to Theorem~\ref{thm.main-result-intro}. In this section, we will focus on the case $\Omega = \R^d$ and $f + f^{-1}$ not necessarily in $L^\infty$---the case that $\Omega = \T^d$ and $f + f^{-1} \not \in L^\infty$ is strictly simpler and follows by a subset of the arguments presented below. The arguments of this section are rather technical and likely standard to experts; we present them here for completeness and for the convenience of the reader. The main idea of this section is to exploit that Proposition~\ref{prop.main-result-special} has no quantitative dependence on $\|f+f^{-1}\|_{L^\infty}$. As such, this qualitative assumption can be removed by a suitable approximation argument, adding an upper and lower bound to the initial data, then removing those bounds and passing the estimate to the limit. Similarly, we use that there is no quantitative dependence on the sidelength of the torus. As such, we can pass the estimate to $\R^d$ by treating $\R^d$ as the limit of tori with large sidelengths.

Throughout this section, we fix some initial data $f \geq 0, f \in L^1(\R^d).$ We then let $f_j, g^i_j : \R^{dj} \to \R$ solve~\eqref{eq.hierarchy-equation} and~\eqref{eq.g-ij-PDE} respectively with initial data induced by $f$. From these $f_j, g^i_j$, we define $\gamma^i_j$ as in~\eqref{eq.gamma-ij-def}---using first the definitions~\eqref{eq.f-ij-def} and~\eqref{eq.phi-ij-def} to define the intermediate objects $f^i_j$ and $\phi^i_j$. We also let $\rho := g^0_1$, as above. In order to conclude, it is our goal to prove under the appropriate assumptions that
\[\int_{\R^{dj}} \bigg|\frac{\gamma^{i}_j}{\rho^{\otimes j}}\bigg|^2 \rho^{\otimes j}\,dx \leq C e^{Ct} \Big(\frac{j}{N}\Big)^{2(i+1)}.\]
This will be done by approximating $\gamma^i_j$ and $\rho$ using solutions to the same equations on tori and with different initial data. We now introduce those approximating solutions.

\subsection{Defining $f^R_j$ and $g^{i,R}_j$}

Throughout, we will identify the $n$-dimensional torus with sidelength $2R,\T^n_R,$ with the set $[-R,R]^n$. We also naturally identify functions $h : \T^n_R \to \R$ (or equivalently $h: [-R,R]^n \to \R$) with functions $h : \R^n \to \R$ using extension by $0$, that is $h(x) = 0$ for $x \not \in [-R,R]^n$.

We first define, for each $R \geq 1$, the initial data $f^R : [-R,R]^d \to \R$ given by
\[f^R(x) := A_R \mathbf{1}_{[-R,R]^d}(x)(f(x) \lor R^{-2d})\land R,\]
where $A_R \in (0,\infty)$ is a normalization constant to ensure $\int f^R = \int_{[-R,R]^d} f^R= 1.$ Using dominated convergence, one can readily verify that as $R \to \infty, f^R \to f$ in $L^1(\R^d).$ We note that $f^R$ is defined so that $f^R + (f^R)^{-1} \in L^\infty(\T^d_R)$ so is suitable initial data to apply Proposition~\ref{prop.main-result-special}.

We now let $f_j^R, g^{i,R}_j : \T^{dj}_R \to \R$ respectively solve~\eqref{eq.hierarchy-equation} and~\eqref{eq.g-ij-PDE} with initial data induced by $f^R$. We note that we keep a common interaction $K : \R^{2d} \to \R$ throughout, though only part of the interaction $K$, namely $K|_{[-R,R]^{2d}}$ appears in the equations~\eqref{eq.hierarchy-equation} and~\eqref{eq.g-ij-PDE} for $f^R_j$ and $g^{i,R}_j$. Similarly to above, we then use the $f^R_j, g^{i,R}_j$ to define $\gamma^{i,R}_j$, and we denote $\rho^R := g^{0,R}_1.$ 

\subsection{Defining $\tilde f^R_j$ and $\tilde g^{i,R}_j$}

It turns out to be convenient to define an additional family of functions $\tilde f^R_j$ and $\tilde g^{i,R}_j$. To see this, let us briefly consider an analogous case. For a drift-diffusion equation
\[\begin{cases}
\partial_t h -\Delta h + \nabla \cdot (bh) =0\\
h(0,\cdot) =h_0,
\end{cases}
\]
for $b : \T_R^d \to \R^d,h_0: \T^d_R \to \R$, $h_0\geq 0$, and $h : [0,\infty) \times \T^d_R \to \R,$ there are multiple natural ways to turn this equation into a PDE on $\R^d$. Here we define $b_R : \R^d \to \R^d$ by
\[b_R(x+k) = b(x),\quad \text{for all } x \in [-R,R]^d, k \in 2R \Z^d,\]
that is $b_R$ is the periodization of $b$. On the other hand, we do not take the periodization of $h_0$, but rather use the identification we take throughout of $h_0 : \R^d \to \R$ by $h_0(x) = 0$ for $x \not \in [-R,R]^d$. Then we let $\tilde h : \R^d \to \R$ solve
\[\begin{cases}
    \partial_t \tilde h - \Delta \tilde h + \nabla \cdot (b_R \tilde h) = 0\\
    \tilde h(0,\cdot) = h_0.
\end{cases}\]
This is then a sensible drift diffusion equation on $\R^d$. One can then readily verify that for $x \in [-R,R]^d$ 
\[h(t,x) = \sum_{k \in 2R\Z^d} \tilde h(t,x+k),\]
that is $h$ is directly computable from $\tilde h$. We note that this notion of viewing the equation on $\T^d_R$ as an equation of $\R^d$ is different from periodizing both the initial data and the coefficients, e.g.\ as in that case of periodization one would generically have that the solution has infinite $L^1$ norm, while in the case we consider here, we have that
\[\|\tilde h\|_{L^1(\R^d)} = \|h\|_{L^1(\T_R^d)}.\]

It is through an analogous process to this that we define $\tilde f^R_j, \tilde g^{i,R}_j$. We first define $K^R : \R^{2d} \to \R$ by
\[K^R(x+k,y+j) = K(x,y),\quad \text{for all }x,y \in [-R,R]^d, k,j \in 2R\Z^d.\]
Then we let $\tilde f^R_j, \tilde g^{i,R}_j : \R^{dj} \to \R$ be defined by equations~\eqref{eq.hierarchy-equation} and~\eqref{eq.g-ij-PDE} respectively, with $K^R$ in place of $K$ and initial data induced by $f^R$. One can then verify that, analogous to the example above, for $x \in [-R,R]^{dj},$ 
\[f^R_j(x) = \sum_{k \in 2R \Z^d} \tilde f^R_j(x+k)\quad \text{and}\quad g^{i,R}_j(x) = \sum_{k \in 2R\Z^d} \tilde g^{i,R}_j(x+k).\]
We also define $\tilde \rho^R := \tilde g^{0,R}_1.$ Since the $\tilde f^R_j$ and $\tilde g^{i,R}_j$ solve equations on $\R^{dj}$ similar to the equations of $f_j$ and $g^i_j$, we will be able to directly show that $\tilde f^R_j \to f_j$ and $\tilde g^{i,R}_j \to g^i_j$ in the appropriate sense. We can then use this and Lemma~\ref{lem:R^d-to-torus} below to show that we really have that $f^R_j \to f_j$ and $g^{i,R}_j \to g^i_j$. This, together with some bounds on $\tilde \rho^R$, will allow us to pass the bounds on $\gamma^{i,R}_j$ given by Proposition~\ref{prop.main-result-special} to the bounds on $\gamma^i_j$ needed to conclude Theorem~\ref{thm.main-result-intro}.

\begin{lemma}\label{lem:R^d-to-torus}
Suppose as $R\rightarrow\infty,\tilde h^R\rightarrow h$ in $C^0([0,T],L^1(\R^{dj})).$ Then
if we define
\[h^R(x)=
\begin{cases}
\sum_{k\in 2R\Z^{dj}}\tilde h^R(x+k)&x\in[-R,R]^{dj},\\
0& \text{otherwise},
\end{cases}
\]
then as $R\rightarrow\infty, h^R\rightarrow h$ in $C^0([0,T],L^1(\R^{dj}))$.
\end{lemma}
\begin{proof}
We have that
\[\|h^R-\tilde h^R\|_{C^0_t L^1_x}\leq 2 \sup_{t \in [0,T]} \int_{([-R,R]^{dj})^C}|\tilde h^R|(t,y)\,dy\leq 2\|\tilde h^R-h\|_{C^0_tL^1_x}+2 \sup_{t \in [0,T]} \int_{([-R,R]^{dj})^C}|h|(t,y)\,dy.\]
The first term on the right hand side goes to $0$ by assumption. For second term, we note that $t \mapsto h(t,\cdot)$ is continuous in $L^1(\R^{dj})$, thus in $TV(\R^{dj})$. Therefore the family of functions $\{h(t,\cdot) : t \in [0,T]\}$ is compact in $TV$, hence uniformly tight. Thus the second term goes to $0$ as well, so $h^R \to h$ in $C^0_tL^1_x$. 
\end{proof}

\subsection{Estimating the McKean-Vlasov equation}

We frequently use the heat kernel on $\R^n$,
\[\Phi_t(x) := (4\pi t)^{-n/2} \exp\Big(\frac{|x|^2}{4t}\Big).\]
Throughout the remainder of this section, we very often use the mild form of the various equations for $f_j, g^i_j$. Generically, for an equation of the form
\[\partial_t h - \Delta h = \nabla \cdot G,\]
the mild form---or equivalently Duhamel's principle---gives that
\[h(t,x) = \Phi_t*h(0,x) + \int_0^t \nabla \Phi_{t-s} * G(s,x)\, ds.\]

\begin{proposition}\label{prop:rho-convergence} For all $T>0$, as $R\rightarrow\infty$, $\tilde \rho^{R}\rightarrow \rho$ in $C^{0}([0,T],L^1(\R^d))$.
\end{proposition}

\begin{proof}
We have that
\begin{align*}
(\tilde\rho^R-\rho)(t)&=\Phi_t*(f^R-f)+\int_0^t \nabla\Phi_{t-s}*\bigg((\rho-\tilde\rho^R)\int K^R(\cdot,z)\tilde\rho^R(z)\,dz\bigg)(s)\,ds
\\&\qquad + \int \nabla \Phi_{t-s}*\bigg(\rho\int K^R(\cdot,z)(\rho(z)-\tilde\rho^R(z))\,dz \bigg)(s)\,ds 
\\&\qquad+\int\nabla\Phi_{t-s}*\bigg(\rho\int (K-K^R)(\cdot,z)\rho(z)\,dz\bigg)(s)\,ds.
\end{align*}
Using that there exits $C(d)<\infty$ so that for all $t>0,$
\[|\nabla\Phi_t|(x)\leq Ct^{-1/2}\Phi_{2t}(x),\]
we find that
\begin{align*}
\|\rho-\tilde\rho^R\|_{L^1_x}&\leq \|f-f^R\|_{L^1_x}+C\|K\|_{L^\infty_{\delta}}\int_0^t(t-s)^{-1/2}\|\rho-\tilde\rho^R\|_{L^1_x}\,ds
\\&\quad+\int_0^t\bigg\|\nabla\Phi_{t-s}*\bigg(\rho\int (K-K^R)(\cdot,z)\rho(z)\,dz\bigg)\bigg\|_{L^1_x}\,ds,
\end{align*}
where we use that $\|\rho\|_{L^1_x}(t)=\|\tilde\rho^R\|_{L^1_x}(t)=1$ for all $t\in[0,T].$ Again using our bound on $\nabla\Phi_t$ we find that
\begin{align*}
  &\bigg\|\nabla\Phi_{t-s}*\bigg(\rho\int K-K^R(\cdot,z)\rho(z)\,dz\bigg)\bigg\|_{L^1_x} 
  \\&\qquad\leq C(t-s)^{-1/2}\int\int\int \Phi_{2(t-s)}(x-y)\rho(y)|K-K^R|(y,z)\rho(z)\,dz\,dy\,dx.
\end{align*}
Next we note that $|K-K^R|(y,z)\leq 2\|K\|_{L^\infty_{\delta}} \Big(\mathbf{1}_{([-R,R]^d)^C}(y)+\mathbf{1}_{([-R,R]^d)^C}(z)\Big)$ thus
\begin{align*}
   &\int\int\int \Phi_{2(t-s)}(x-y)\rho(y)|K-K^R|(y,z)\rho(z)\,dz\,dy\,dx 
   \\&\quad\leq 2\|K\|_{L^\infty_{\delta}}\int\int_{([-R,R]^d)^C}\int \Phi_{2(t-s)}(x-y)\rho(y)\rho(z)\,dz\,dy\,dx
   \\&\quad\qquad+2\|K\|_{L^\infty_{\delta}}\int\int\int_{([-R,R]^d)^C} \Phi_{2(t-s)}(x-y)\rho(y)\rho(z)\,dz\,dy\,dx 
   \\&\quad= 2\|K\|_{L^\infty_{\delta}}\int_{([-R,R]^d)^C} \rho(y)\,dy+2\|K\|_{L^\infty_\delta}\int_{([-R,R]^d)^C} \rho(z)\,dz.
\end{align*}
In total we have found that there exists $C(d,T,M)<\infty$ so that for all $t\leq T$
\begin{align*}
\|\rho-\tilde\rho^R\|_{L^1_x}(t)&\leq \|f-f^R\|_{L^1_x}+C\sup_{s\in[0,T]}\int_{([-R,R]^d)^C}\rho(s,y)\,dy
    \\ &\quad + C\int_0^t(t-s)^{-1/2}\|\rho-\tilde\rho^R\|_{L^1_x}(s)\,ds.
\end{align*}
Taking $\tau$ so that $C\int_0^\tau(\tau-s)^{-1/2}\,ds<\frac{1}{2}$, we get that 
\[\sup_{t\in[0,\tau]}\|\rho-\tilde\rho^R\|_{L^1_x}(t)\leq 2\|f-f^R\|_{L^1_x}+C\sup_{s\in[0,T]}\int_{([-R,R]^d)^C}\rho(s,y)\,dy.\]
Iterating this bound we find there exists $C(d,T,\|K\|_{L^\infty_{\delta}})<\infty$ so that
\[\|\rho-\tilde\rho^R\|_{C^0_tL^1_x}\leq C\|f-f^R\|_{L^1_x}+C\sup_{s\in[0,T]}\int_{([-R,R]^d)^C}\rho(s,y)\,dy.\]
Clearly the first term on the right hand side above goes to $0$ as $R\rightarrow\infty$. The second term goes to $0$ using that the family $\{\rho(s,\cdot):s\in[0,T]\}$ is compact in $TV(\R^d)$, thus uniformly tight.
\end{proof}

\begin{proposition}\label{prop:rho-positivity}
	For all $L<\infty$ and all $t>0$, there exists $C(d,t,L,f,\|K\|_{L^\infty_{\delta}})< \infty$ so that for all sufficiently large $R,$
	\[\inf_{x\in B_L}\tilde\rho^R(t,x) \geq C^{-1}.\]
\end{proposition}
\begin{proof}
We fix $M<\infty$ so that $\|f\|_{L^1(B_M)}\geq\frac{3}{4}.$ Since $f^R\rightarrow f$ in $L^1(\R^d)$, $\|f^R\|_{L^1(B_M)}\geq \frac{1}{2}$ for all sufficiently large $R.$ As $f$ is positive there exists $C(d,M)<\infty$ so that for all $t\leq 1$ and sufficiently large $R$
\[\|\Phi_t*f^R\|_{L^1(B_M)}\geq C^{-1},\]
On the other hand,
\[\|\tilde\rho^R(t)-\Phi_t*f^R\|_{L^1_x}\leq C\|K\|_{L^\infty_{\delta}}\int_0^t(t-s)^{-1/2}\,ds,\]
hence there exists $\tau(d,M,\|K\|_{L^\infty_{\delta}})>0$ so that for all $t\leq\tau$
\begin{equation}\label{eq:local-lower-bound}
    \|\tilde\rho^{R}\|_{L^1(B_M)}\geq C^{-1}.
\end{equation}
By~\cite[Corollary 6.4.3]{bogachev_fokker-planck-kolmogorov_2015} we see that $\tilde\rho^R$ satisfies the conditions of~\cite[Theorem 8.2.1]{bogachev_fokker-planck-kolmogorov_2015} so there exists $C(d)<\infty$ so that for all $x,y\in\R^d$ and $0<s<t,$
\[\tilde\rho^R(t,x)\geq \tilde\rho^R(s,y)\exp\Big({-C}\Big(1+\frac{t-s}{s}\|K\|_{L^\infty_{\delta}}^2+\frac{|x-y|^2}{t-s}\Big)\Big).\]
We thus find that there exists $C(d,t,s,\|K\|_{L^\infty_{\delta}})<\infty$ so that for all $x\in B_L$ and $y\in B_M$
\[\tilde\rho^R(t,x)\geq C^{-1}\tilde\rho^R(s,y)\exp\Big(-C\sup\limits_{z\in B_L,w\in B_M}\frac{|z-w|^2}{t-s}\Big).\]
We then conclude the proposition by averaging both sides over $y\in B_M$, choosing $s=\tau\wedge \frac{t}{2}$, and using~\eqref{eq:local-lower-bound}.
\end{proof}

\subsection{Convergence of the $f_j^R$ and $g^{i,R}_j$ }

\begin{proposition}\label{prop:f_ij-convergence}
For all $T>0$ and $1\leq j\leq N$, as $R\rightarrow \infty$, $\tilde f_j^R\rightarrow f_j$ in $C^0([0,T],L^1(\R^{dj})).$    
\end{proposition}

\begin{proof}
We note that it suffices to prove that $\tilde f_N^R\rightarrow f_N$ in $C^0([0,T],L^1(\R^{dN}))$ since the fact that $\tilde f_N^R-f_N$ marginalizes to $\tilde f_j^R-f_j$ implies that
\[\|\tilde f_j^R-f_j\|_{L^1(\R^{dj})}\leq \|\tilde f_N^R-f_N\|_{L^1(\R^{dN})}.\]
Letting
\[J(x):=-\frac{1}{N}\sum_{k=1}^Ne_k\otimes\sum_{\ell=1}^NK(x_k,x_\ell) \]
and
\[J^R(x):=-\frac{1}{N}\sum_{k=1}^Ne_k\otimes\sum_{\ell=1}^NK^R(x_k,x_\ell),\]
we have that
\begin{align*}
(f_N-\tilde f_N^R)(t)&=\Phi_t*(f^{\otimes N}-(f^R)^{\otimes N})+\int_0^t\nabla\Phi_{t-s}*\Big(Jf_N-J^R\tilde f^R_N\Big)(s)\,ds
\\&= \Phi_t*(f^{\otimes N}-(f^R)^{\otimes N})+\int_0^t\nabla\Phi_{t-s}*\Big((J-J^R)f_N\Big)(s)\,ds
\\&\qquad\qquad\qquad\qquad\qquad\qquad\qquad+\int_0^t\nabla\Phi_{t-s}*\Big(J^R(f_N-\tilde f^R_N)\Big)(s)\,ds.
\end{align*}
We thus find that
\begin{align*}
\|f_N-\tilde f_N^R\|_{L^1_x}(t)&\leq \|f^{\otimes N}- (f^R)^{\otimes N}\|_{L^1_x}+C\|J_R\|_{L^\infty_{t,x}}\int_0^t(t-s)^{-1/2}\|f_N-\tilde f^R_N\|_{L^1_x}(s)\,ds
\\&\quad+C\int_0^t(t-s)^{-1/2}\|(J-J^R)f_N\|_{L^1_x}(s)\,ds
\end{align*}
Using that $J=J^R$ on $[-R,R]^{dN}$ we thus find that there exists some $C(d,N,T,\|K\|_{L^\infty_{\delta}})<\infty$ so that
\begin{align*}
\|f_N-\tilde f_N^R\|_{L^1_x}(t)&\leq \|f^{\otimes N}-(f^R)^{\otimes N}\|_{L^1_x}+C\int_0^t(t-s)^{-1/2}\|f_N-\tilde f^R_N\|_{L^1_x}(s)\,ds
\\&\quad+C\sup_{t\in[0,T]}\int_{([-R,R]^{dN})^C}f_N(s,y)\,dy.
\end{align*}
Using the same iteration argument as the end of the proof of Proposition~\ref{prop:rho-convergence} we thus find there exists some $C(d,N,T,\|K\|_{L^\infty_{\delta}})<\infty$ so that
\[\|f_N-\tilde f_N^R\|_{C^0_tL^1_x}\leq C\|f^{\otimes N}-(f^R)^{\otimes N}\|_{L^1_x}+C\sup_{s\in[0,T]}\int_{([-R,R]^{dN})^C}f_N(s,y)\,dy.\]
We conclude by noting that $(f^R)^{\otimes N}\rightarrow f^{\otimes N}$ in $L^1(\R^{dN})$ and as $R\rightarrow\infty$
\[\sup_{s\in[0,T]}\int_{([-R,R]^{dN})^C}f_N(s,y)\,dy\rightarrow 0,\]
by the uniform tightness of $\{f_N(s,\cdot):s\in[0,T]\}.$
\end{proof}

\begin{lemma}\label{lem:forcing-convergence}
Suppose as $R\rightarrow\infty$, $\tilde h^R\rightarrow h$ in $C^0([0,T],L^1(\R^{dj})).$ Then, as $R\rightarrow\infty,$
\begin{itemize}
    \item for all $j \geq 1,$ $k,\ell\in[j]$, 
    \[K^R(x_k,x_\ell)\tilde h^R(x)\rightarrow K(x_1,x_2)h(x) \text{ in } C^0([0,T],L^1(\R^{dj})),\]
    \item for all $j \geq 2$, 
    \[\int K^R(x_1,x_2)\tilde h^R(x)\,dx_2\rightarrow \int K(x_1,x_2)h(x)\,dx_2\text{ in }C^0([0,T],L^1(\R^{d(j-1)})).\]
\end{itemize}
\end{lemma}
\begin{proof}
For the first point we have that
\[\|K^Rh^R-Kh\|_{L^1_x}\leq \|K^R(h^R-h)\|_{L^1_x}+\|(K^R-K)h\|_{L^1_x}\leq \|K\|_{L^\infty_{\delta}}\|h^R-h\|_{L^1_x}+\|K\|_{L^\infty_{\delta}}\int_{([-R,R]^{dj})^C} |h|\,dy.\]
The statement then follows by the uniform tightness of $\{h(s,\cdot):s\in[0,T]\}$. The second point follows similarly.
\end{proof}

\begin{proposition}\label{prop:g_ij-convergence}
For all $T>0$ and $(i,j)\in D$, as $R\rightarrow \infty$, $\tilde{g}^{i,R}_j\rightarrow g^i_j$ in $C^0([0,T],L^1(\R^{dj})).$   
\end{proposition}

\begin{proof}
We will prove this inductively on $D$ using the ordering given in Definition~\ref{def.order}.
We will thus inductively show that $\|\tilde g^{i,R}_j\|_{C^0_tL^1_x}$ are bounded uniformly in $R$ and $\tilde g^{i,R}_j\rightarrow g^{i}_j$ in $C^0([0,T],L^1(\R^{dj}))$.

The base case $(i,j)=(0,1)$ is immediately given by Proposition~\ref{prop:rho-convergence} and that $\|\tilde\rho^R\|_{L^1_x}(t)=1$ for all $R$ and $t\in[0,T].$

For $(i,j)>(0,1)$ we then find that
\begin{align*}
\tilde g^{i,R}_j(t)&=-\int_0^t\nabla\Phi_{t-s}*\Big(\sum_{k=1}^j\tilde g^{i,R}_je_k\otimes \int K^R(x_k,x_*)\tilde\rho^R(x_*)dx_*\Big)(s)\,ds
\\&\quad-\int_0^t\nabla\Phi_{t-s}*\Big(\sum_{k=1}^j\tilde\rho^R(x_k)e_k\otimes \int K^R(x_k,x_*)\tilde g^{i,R}_j(x^{[j]-\{k\}\cup\{x_*\}})\,dx_*\Big)(s)\,ds
\\&\quad+\int_0^t\nabla\Phi_{t-s}*\tilde F^{i,R}_j(s)\,ds,
\end{align*}
where $\tilde F^{i,R}_j$ is the forcing and depends on $\tilde g^{k,R}_{\ell}$ for all $(k,\ell)<(i,j).$ We thus find that for some $C(d,j,T)<\infty$
\begin{align*}
\|\tilde g^{i,R}_j\|_{L^1_x}(t)&\leq C\|K\|_{L^\infty_{\delta}}\|\tilde\rho^R\|_{C^0_tL^1_x}\int_0^t(t-s)^{-1/2}\|\tilde g^{i,R}_j\|_{L^1_x}(s)\,ds
\\&\quad+C\|\tilde F^{i,R}_j\|_{C^0_tL^1_x}.
\end{align*}
Using the iteration argument at the end of the proof of Proposition~\ref{prop:rho-convergence} we find there exists $C(d,j,T,\|K\|_{L^\infty_{\delta}})<\infty$ so that
\[\|\tilde g^{i,R}_j\|_{L^1_x}(t)\leq C\|\tilde F^{i,R}_j\|_{C^0_tL^1_x}.\]
Since $\|K^R\|_{L^\infty_{\delta}}\leq \|K\|_{L^\infty_{\delta}}$ it is clear that if the $\tilde g^{k,R}_{\ell}$ are all uniformly bounded over $R$, then so is $\tilde F^{i,R}_j$, and thus so is $\tilde g^{i,R}_j.$

We next prove the convergence. Letting $F^i_j$ be the forcing of $g^i_j$, then
\begin{align*}
(\tilde g^{i,R}_j-g^i_j)(t)&=\int_0^t\nabla\Phi_{t-s}*\Big(\sum_{k=1}^jg^i_je_k\otimes \int K(x_k,x_*)\rho(x_*)\,dx_*-\tilde g^{i,R}_je_k\otimes \int K^R(x_k,x_*)\rho^R(x_*)\,dx_*\Big)\,ds
\\&\quad-\int_0^t\nabla\Phi_{t-s}*\Big(\sum_{k=1}^j\tilde\rho^R(x_k)e_k\otimes \int K^R(x_k,x_*)\tilde g^{i,R}_j(x^{[j]-\{k\}\cup\{x_*\}})\,dx_*\Big)(s)\,ds
\\&\quad+\int_0^t\nabla\Phi_{t-s}*\Big(\sum_{k=1}^j\rho(x_k)e_k\otimes \int K(x_k,x_*) g^i_j(x^{[j]-\{k\}\cup\{x_*\}})\,dx_*\Big)(s)\,ds
\\&\quad+\int_0^t\nabla\Phi_{t-s}*(\tilde F^{i,R}_j-F^i_j)(s)\,ds,
\end{align*}
thus
\begin{align}
\notag\|\tilde g^{i,R}_j-g^i_j\|_{L^1}(t)&\leq\int_0^t\Big\|\nabla\Phi_{t-s}*\Big(\sum_{k=1}^j\tilde g^{i,R}_je_k\otimes \int K^R(x_k,x_*)\tilde\rho^R(x_*)\,dx_*
\\&\qquad\qquad\qquad-g^i_je_k\otimes \int K(x_k,x_*)\rho(x_*)\,dx_*\Big)\Big\|_{L^1_x}(s)\,ds
\notag\\&\quad+\int_0^t\Big\|\nabla\Phi_{t-s}*\Big(\sum_{k=1}^j\tilde\rho^R(x_k)e_k\otimes \int K^R(x_k,x_*)\tilde g^{i,R}_j(x^{[j]-\{k\}\cup\{x_*\}})\,dx_*
\notag\\&\qquad\qquad\qquad-\rho(x_k)e_k\otimes \int K(x_k,x_*) g^i_j(x^{[j]-\{k\}\cup\{x_*\}})\,dx_*\Big)\Big\|_{L^1_x}(s)\,ds
\notag\\&\quad+\int_0^t\|\nabla\Phi_{t-s}*(\tilde F^{i,R}_j-F^i_j)(s)\|_{L^1_x}(s)\,ds.\label{eq:g_ij-L^1-bound}
\end{align}
Using exchangeability, we bound the first term in~\eqref{eq:g_ij-L^1-bound} as follows
\begin{align*}
&\int_0^t\Big\|\nabla\Phi_{t-s}*\Big(\sum_{k=1}^j\tilde g^{i,R}_je_k\otimes \int K^R(x_k,x_*)\tilde\rho^R(x_*)\,dx_*-g^i_je_k\otimes \int K(x_k,x_*)\rho(x_*)\,dx_*\Big)\Big\|_{L^1_x}(s)\,ds
\\&\quad\leq C\int_0^t (t-s)^{-1/2} \Big\|\int \tilde g^{i,R}_j(x)K^R(x_1,x_*)\tilde\rho^R(x_*)-g^i_j(x)K(x_1,x_*)\rho(x_*)\,dx_*\Big\|_{L^1_x}(s)\,ds
\\&\quad\leq C\int_0^t (t-s)^{-1/2} \Big(\|\tilde g^{i,R}_j-g^i_j\|_{L^1_x} \|K^R\|_{L^\infty_\delta} \|\tilde\rho^R\|_{L^1_x}+\|g^i_j\|_{L^1_x} \|K^R\|_{L^\infty_\delta} \|\tilde\rho^R-\rho\|_{L^1_x}
\\&\qquad\qquad\qquad\qquad\qquad\qquad\qquad\qquad\qquad+\|g^i_j\|_{L^1_x}\Big\|\int (K^R-K)(x_1,x_*)\rho(x_*)\,dx_*\Big\|_{L^1_x}\Big)(s)\,ds 
\\&\quad\leq C\int_0^t(t-s)^{-1/2}\|\tilde g^{i,R}_j-g^i_j\|_{L^1_x}(s)\,ds+C\|\tilde\rho^R-\rho\|_{C^0_tL^1_x}+C\sup_{s\in[0,T]}\int_{([-R,R]^{d})^C}\rho(s,y)\,dy.
\end{align*}
Similarly we can bound the second term in~\eqref{eq:g_ij-L^1-bound} by
\[C\int_0^t(t-s)^{-1/2}\|\tilde g^{i,R}_j-g^i_j\|_{L^1_x}(s)\,ds+C\|\tilde g^{i,R}_j\|_{C^0_tL^1_x}\|\rho-\tilde\rho^R\|_{C^0_tL^1_x}+C\sup_{s\in[0,T]}\int_{([-R,R]^{dj})^C}|g^i_j|(s,y)\,dy.\]
The last term in~\eqref{eq:g_ij-L^1-bound} is bounded by $\|\tilde F^{i,R}_j-F^i_j\|_{C^0_tL^1_x}.$ In total, using the uniform bounds on $\tilde g^{i,R}_j$ we find that there exists $C(d,i,j,T,\|K\|_{L^\infty_{\delta}})<\infty$ so that
\begin{align*}
\|\tilde g^{i,R}_j-g^i_j\|_{L^1_x}(t)\leq& C\Big(\|\rho-\tilde\rho^R\|_{C^0_tL^1_x}+\|\tilde F^{i,R}_j-F^i_j\|_{C^0_tL^1_x}
\\&+\sup_{s\in[0,T]}\int_{([-R,R]^{d})^C}\rho(s,y)\,dy+\sup_{s\in[0,T]}\int_{([-R,R]^{dj})^C}|g^i_j|(s,y)\,dy\Big)
\\&+C\int_0^t(t-s)^{-1/2}\|\tilde g^{i,R}_j-g^i_j\|_{L^1_x}(s)\,ds.
\end{align*}
Using the same iteration argument as the end of the proof of Proposition~\ref{prop:rho-convergence}, to conclude it suffices to show that as $R \to \infty,$
\begin{gather}
    \label{eq.g-ij-convergence-1}\|\rho-\tilde\rho^R\|_{C^0_tL^1_x}\rightarrow 0,\\
    \label{eq.g-ij-convergence-2}\sup_{s\in[0,T]}\int_{([-R,R]^{d})^C}\rho(s,y)\,dy+\sup_{s\in[0,T]}\int_{([-R,R]^{dj})^C}|g^i_j|(s,y)\,dy\rightarrow 0,\\
    \label{eq.g-ij-convergence-3}\|\tilde F^{i,R}_j-F^i_j\|_{C^0_tL^1_x}\rightarrow 0.
\end{gather}
We note that~\eqref{eq.g-ij-convergence-1} follows from Proposition~\ref{prop:rho-convergence},~\eqref{eq.g-ij-convergence-2} follows from the uniform tightness of $\rho$ and $g^i_j$, and~\eqref{eq.g-ij-convergence-3} follows from the inductive hypothesis, Lemma~\ref{lem:forcing-convergence}, and the definitions of $\tilde F^{i,R}_j$ and $F^i_j.$
\end{proof}

\subsection{Completing the proof of Theorem~\ref{thm.main-result-intro}}

\begin{proof}[Proof of Theorem~\ref{thm.main-result-intro}] We fix $t>0$ throughout. By Proposition~\ref{prop.main-result-special}, there exists $C(\|K\|_{L^\infty_{\delta}},i)<\infty$ so that for all 
    \[j \leq C^{-1} e^{-Ct^2} N^{2/3},\]
    we have the bound for all $R \geq 1$,
    \begin{equation}
    \int_{[-R,R]^{dj}} \bigg|\frac{\gamma^{i,R}_j}{(\rho^{R})^{\otimes j}}\bigg|^2 (\rho^R)^{\otimes j}\,dx \leq  Ce^{Ct} \Big(\frac{j}{N}\Big)^{2(i+1)}.
    \end{equation}
    In order to conclude, we need to show the same bound for the same range of $j$ and the same constant for $\gamma^i_j$ in place of $\gamma^{i,R}_j.$

    We note that
    \[\int \bigg|\frac{\gamma^{i,R}_j}{((\rho^{R})^{\otimes j})^{1/2}}\bigg|^2 dx=\int \bigg|\frac{\gamma^{i,R}_j}{(\rho^R)^{\otimes j}}\bigg|^2 (\rho^R)^{\otimes j}\leq  Ce^{Ct} \Big(\frac{j}{N}\Big)^{2(i+1)}.\]
  Thus $\gamma^{i,R}_j((\rho^R)^{\otimes j})^{-1/2}$ is uniformly bounded in $L^2(\R^{dj})$, so by weak compactness there is a subsequence $R_k \to \infty$ along which $\gamma^{i,R_k}_j((\rho^{R_k})^{\otimes j})^{-1/2}$  weakly converges to some $G\in L^2(\R^{dj})$ with
  \begin{equation}
  \label{eq.G-ineq}\int |G|^2 \leq Ce^{Ct} \Big(\frac{j}{N}\Big)^{2(i+1)}.\end{equation}
  To conclude then it suffices to show that this limit $G$ is what we expect it to be: $\gamma^i_j (\rho^{\otimes j})^{-1/2}.$
  
  Next Propositions~\ref{prop:rho-convergence},~\ref{prop:f_ij-convergence}, and~\ref{prop:g_ij-convergence} with Lemma~\ref{lem:R^d-to-torus} and the definitions of $\gamma^{i,R}_j,\gamma^{i}_j$ imply
    \[\|\gamma^{i,R}_j - \gamma^i_j\|_{C^0_tL^1_x} \to 0 \text{ as } R \to \infty,\]
    and
    \[\|\rho^{R}-\rho\|_{C^0_tL^1_x} \to 0 \text{ as } R \to \infty.\]
    For all $L>0,$ Proposition~\ref{prop:rho-positivity} implies there exists some $C<\infty$ so that $\rho,\rho^{R}\geq C^{-1}$ on $B_L$ for all sufficiently large $R$. Accordingly we find that
\begin{align}
    \notag\int_{B_L} \bigg|\frac{\gamma^{i,R}_j}{((\rho^R)^{\otimes j})^{1/2}}-\frac{\gamma^i_j}{(\rho^{\otimes j})^{1/2}}\bigg|&\leq \int_{B_L} \bigg|\frac{\gamma^{i,R}_j(\rho^{\otimes j})^{1/2}-\gamma^i_j((\rho^R)^{\otimes j})^{1/2}}{(\rho^{\otimes j})^{1/2}((\rho^{R})^{\otimes j})^{1/2}}\bigg|
    \\\notag&\leq \int_{B_L} \frac{|\gamma^{i,R}_j|}{((\rho^R)^{\otimes j})^{1/2}}\frac{|(\rho^{\otimes j})^{1/2}-((\rho^{R})^{\otimes j})^{1/2}|}{(\rho^{\otimes j})^{1/2}}+C\int_{B_L} \frac{|\gamma^{i,R}_j-\gamma^i_j|}{((\rho^R)^{\otimes j})^{1/2}}
    \\\label{eq:local-ratio-convergence}&\leq C\int \frac{|\gamma^{i,R}_j|}{(\rho^{\otimes j})^{1/2}}|(\rho^{\otimes j})^{1/2}-((\rho^{R})^{\otimes j})^{1/2}|+C\int |\gamma^{i,R}_j-\gamma^i_j|.
\end{align}
Since $\rho^R$ converges to $\rho$ in $L^1(\R^d)$, the Vitali convergence thoerem implies that $((\rho^R)^{\otimes j})^{1/2}$ converges to $(\rho^{\otimes j})^{1/2}$ in $L^2(\R^{dj})$. Together with the uniform boundedness of $\gamma^{i,R}_j((\rho^R)^{\otimes j})^{-1/2}$ in $L^2(\R^{dj})$, this implies that the first term in the right hand side of~\eqref{eq:local-ratio-convergence} converges to 0. The second term converges to $0$ since $\gamma^{i,R}_j$ converges to $\gamma^i_j$ in $L^1(\R^{dj})$. We have thus found that $\gamma^{i,R}_j((\rho^R)^{\otimes j})^{-1/2}$ converges to $\gamma^i_j(\rho^{\otimes j})^{-1/2}$ in $L^1_{loc}(\R^{dj})$, and so we must have that
\[G=\frac{\gamma^i_j}{(\rho^{\otimes j})^{1/2}},\]
from which~\eqref{eq.G-ineq} allows us to conclude.
\end{proof}

{\small
\bibliographystyle{alpha}
\bibliography{references,references2}
}
 
\end{document}